\documentclass[12pt]{amsart}
\usepackage{ amsmath, amsthm, amsfonts, amssymb, color}
 \usepackage{mathrsfs}
\usepackage{amsfonts, amsmath}
 \usepackage{amsmath,amstext,amsthm,amssymb,amsxtra}
 \usepackage{txfonts} %also pxfonts
 \usepackage[colorlinks, citecolor=blue,pagebackref,hypertexnames=false]{hyperref}
 \allowdisplaybreaks
 \usepackage{pgf,tikz}
 \usepackage{multirow}%newtablepackage1
 \usepackage{diagbox} %newtablepackage2

 \usepackage{tikz}

\usetikzlibrary{decorations.pathreplacing}

\begin{document}

 \baselineskip 16.6pt
\hfuzz=6pt

\widowpenalty=10000

\newtheorem{cl}{Claim}
\newtheorem{theorem}{Theorem}[section]
\newtheorem{proposition}[theorem]{Proposition}
\newtheorem{coro}[theorem]{Corollary}
\newtheorem{lemma}[theorem]{Lemma}
\newtheorem{definition}[theorem]{Definition}
\newtheorem{assum}{Assumption}[section]
\newtheorem{example}[theorem]{Example}
\newtheorem{remark}[theorem]{Remark}
\renewcommand{\theequation}
{\thesection.\arabic{equation}}

\def\SL{\sqrt H}

\newcommand{\mar}[1]{{\marginpar{\sffamily{\scriptsize
        #1}}}}

\newcommand{\as}[1]{{\mar{AS:#1}}}

\newcommand\R{\mathbb{R}}
\newcommand\RR{\mathbb{R}}
\newcommand\CC{\mathbb{C}}
\newcommand\NN{\mathbb{N}}
\newcommand\ZZ{\mathbb{Z}}
\newcommand\HH{\mathbb{H}}
\newcommand\Z{\mathbb{Z}}
\def\RN {\mathbb{R}^n}
\renewcommand\Re{\operatorname{Re}}
\renewcommand\Im{\operatorname{Im}}

\newcommand{\mc}{\mathcal}
\newcommand\D{\mathcal{D}}
\def\hs{\hspace{0.33cm}}
\newcommand{\la}{\alpha}
\def \l {\alpha}
\newcommand{\eps}{\tau}
\newcommand{\pl}{\partial}
\newcommand{\supp}{{\rm supp}{\hspace{.05cm}}}
\newcommand{\x}{\times}
\newcommand{\lag}{\langle}
\newcommand{\rag}{\rangle}

\newcommand\wrt{\,{\rm d}}

\newcommand{\norm}[2]{|#1|_{#2}}
\newcommand{\Norm}[2]{\|#1\|_{#2}}

\title[]{Schatten--Lorentz characterization of Riesz transform commutator associated with Bessel operators}

\author{Zhijie Fan}
\address{Zhijie Fan, School of Mathematics and Information Science,
Guangzhou University, Guangzhou 510006, China}
\email{fanzhj3@mail2.sysu.edu.cn}

\author{Michael Lacey}
\address{Michael Lacey, Department of Mathematics, Georgia Institute of Technology
Atlanta, GA 30332, USA}
\email{lacey@math.gatech.edu}

\author{Ji Li}
\address{Ji Li, Department of Mathematics, Macquarie University, Sydney}
\email{ji.li@mq.edu.au}

\author{Xiao Xiong}
\address{Xiao Xiong, Institute for Advanced Study in Mathematics, Harbin Institute of Technology, 150001 Harbin, China}
\email{xxiong@hit.edu.cn}

 % \date{\today}


\begin{abstract}
Let $\Delta_\lambda$ be the Bessel operator on the upper half space $\mathbb{R}_+^{n+1}$ with $n\geq 0$ and $\lambda>0$, and $R_{\lambda,j}$ be the $j-$th Bessel Riesz transform, $j=1,\ldots,n+1$.
We demonstrate that the Schatten--Lorentz norm ($S^{p,q}$, $1<p<\infty$, $1\leq q\leq \infty$) of the  commutator $[b,R_{\lambda,j}]$ can be characterized in terms of the oscillation space norm of the symbol $b$. In particular, for the case $p=q$, the Schatten norm of $[b,R_{\lambda,j}]$ can be further characterized in terms of the Besov norm of the symbol. Moreover, the critical index is also studied, which is  $p=n+1$, the lower dimension of the Bessel measure (but not the upper dimension). Our approach relies on martingale and dyadic analysis, which enables us  to bypass the use of Fourier analysis effectively.
\end{abstract}

\maketitle

\tableofcontents

%\newpage
\section{Introduction}
\subsection{Background and motivation}

The study of Schatten--Lorentz class is helpful for better understanding the behavior of compact operators in infinite dimensional spaces, playing a significant role in a variety of problems in many branches of analysis and mathematical quantum theory. In particular, as a more in-depth study characterizing the boundedness and compactness of operators, the establishment of Schatten class characterizations for various types of integral operators has been a central focus in the field of complex analysis for over half a century (see e.g. \cite{AFP,MR1359927,MR1323687,FR,MR3158507,MR899655,MR3010276,MR1766114,MR1087805}). Moreover, attention to this topic has significantly increased within the interdisciplinary realms of functional and harmonic analysis in recent years (see e.g. \cite{MR4660138,MR3212723,MR4312282,MR3650324,MR4552581,MR4654013,MR3494249,MR4687434}).

Originating from the famous Nehari's \cite{Ne} problem in complex analysis, Riesz transform commutator is a type of classical non-convolution singular integral operator. Its Schatten (or weak Schatten) characterization  was investigated by many authors in different contexts (see e.g. \cite{MR4552581,MR4654013,MR4549699,JW,LLW,LMSZ,MSX2018,MSX2019,RS1988,RS}) due to its close connection with quantized derivatives in non-commutative geometry (introduced by Connes \cite[IV]{Connes} and further discussed in \cite{CST,MR4667745, LMSZ, MSX2018, MSX2019}). Specifically, let $[b,T]$ be the commutator of $b$ with an operator $T$ (in this context, the $j$-th Riesz transform $R_j$ on $\mathbb{R}^n$) given by
$[b,T](f)(x):= b(x)T(f)(x) - T(bf)(x).  $

Moreover, let ${d\hspace*{-0.08em}\bar{}\hspace*{0.1em}}b$ be the Euclidean quantised derivatives (see \cite[p.2355]{LMSZ} for the definition). Then ${d\hspace*{-0.08em}\bar{}\hspace*{0.1em}}b$ can be represented via Riesz transform commutator as follows (see   \cite[p.1005]{MR4667745}):
\begin{align}\label{quande}
{d\hspace*{-0.08em}\bar{}\hspace*{0.1em}}b=\sum_{j=1}^n\gamma_j\otimes [b,R_j],
\end{align}
where $\gamma_1,\ldots,\gamma_n$ are self-adjoint complex matrices satisfying the anticommutation relation
$\gamma_j\gamma_k+\gamma_k\gamma_j=2\delta_{j,k},\, 1\leq j,k\leq n,$ and $\delta_{j,k}$ is Kronecker delta.

To illustrate more background, we first recall the definitions of the Lebesgue--Lorentz sequence space and Schatten--Lorentz class. A sequence $\{a_k\}$ is in the $\ell^{p,q}$ for some $0<p<\infty$ and $0<q<\infty$ provided the non-increasing rearrangement of $\{a_k\}$, denoted by $\{a_k^*\}$, satisfies $\sum_{k}a_k^{*q}k^{q/p-1}<+\infty$. For the endpoint $q=\infty$, a sequence $\{a_k\}$ is in the $\ell^{p,\infty}$ for some $0<p<\infty$ provided $\sup_{k}\{k^{1/p}a_k^{*}\}<+\infty$. See also \cite[Proposition 1.4.9]{MR3243734} for equivalent definitions. Furthermore, note that if $T$ is any compact operator on a Hilbert space $H$, then $T^{*}T$ is compact, positive and therefore diagonalizable. For $0<p<\infty$ and $0<q\leq \infty$, we say that $T\in S^{p,q}(H)$ if $\{\lambda_{k}\}\in \ell^{p,q}$, where $\{\lambda_{k}\}$ is the sequence of square roots of eigenvalues of $T^{*}T$ (counted according to multiplicity).

Let ${\rm OSC}_{p,q}(\mathbb R^n)$ and ${\rm B}_{p}(\mathbb R^n)$ be the classical oscillation space and Besov space, respectively (see Definitions \ref{oscdef} and \ref{beso} for specific definitions in an almost parallel setting).
In the classical setting, it was proven by Rochberg--Semmes \cite{RS} that for any $1<p< \infty$ and $1\leq q\leq \infty$, $[b,R_j]\in S^{p,q}(L^2(\mathbb{R}^n)) $ if and only if $b$ is in the oscillation space ${\rm OSC}_{p,q}(\mathbb R^n)$, which together with \eqref{quande}, implies that ${d\hspace*{-0.08em}\bar{}\hspace*{0.1em}}b\in S^{p,q}(L^2(\mathbb{R}^n))$ if and only if $b$ is in ${\rm OSC}_{p,q}(\mathbb R^n)$. In particular, for $p=q$, there holds:

(1) In the case of dimension $n=1$, one has $[b,\mathcal{H}]\in S^p(L^2(\mathbb{R})) $ if and only if $b\in {\rm B}_{p}(\mathbb R)$, where $\mathcal{H}$ is the Hilbert transform and $0<p<\infty$ (see \cite{P}).

(2) In the case of dimension $n\geq 2$,  one has $[b,R_j]\in S^p(L^2(\mathbb{R}^n)) $ if and only if $b\in {\rm B}_{p}(\mathbb R^n)$ when $p>n$, whereas $[b,R_j]\in S^p(L^2(\mathbb{R}^n)) $ if and only if $b$ is a constant when $0<p\leq n$ (see \cite{JW,RS1988,RS}).

Schatten class characterizations have been established for many significant operators in complex analysis, such as Szeg\H{o} projection \cite{FR}, big and little Hankel operators on the unit ball and  Heisenberg group \cite{FR},  big Hankel operator on Bergman space of the disk  \cite{AFP}, and  Hankel operators on the Bergman space of the unit ball \cite{MR1087805}, Riesz transform commutators on Heisenberg group \cite{MR4552581}. Bessel operator is another important model arising in the study of complex analysis and partial differential equations (see e.g. \cite{MR64284,MS}).  For $\lambda> 0$, the one-dimensional Bessel operator  on $\R^{1}_+:= (0,\infty)$ (see \cite{MS})  is defined by
\begin{align*}
\Delta_{\lambda}^{(1)} := -{d^2\over dx^2} -{2\lambda\over x} {d\over dx}.
\end{align*}
In higher dimension, we also consider the $(n+1)$-dimensional Bessel operator on $\R^{n+1}_+:=\R^n\times (0,\infty)$ ($n\geq 1$) from a seminal work of Huber \cite{MR64284}, which is defined by
\begin{align}\label{Dlambda}
\Delta_\lambda^{(n+1)} := -{\partial^2\over \partial x_1^2}\cdots-{\partial^2\over \partial x_n^2} -{\partial^2\over \partial x_{n+1}^2} -{2\lambda\over x_{n+1}} {\partial\over \partial x_{n+1}}.
\end{align}
The operators $\Delta_\lambda^{(1)}$ and $\Delta_\lambda^{(n+1)}$ are densely-defined non-negative self-adjoint operators on $L^2(\mathbb R_+, dm_\lambda^{(1)})$ and $ L^2(\mathbb{R}_+^{n+1}, dm_\lambda^{(n+1)})$, respectively,
 where $d m_\lambda^{(1)}(x):=x^{2\lambda}dx$ and
$$dm_\lambda^{(n+1)}(x):=\prod_{j=1}^n dx_j x_{n+1}^{2\lambda}dx_{n+1}.$$ 
For simplicity, we will introduce the following unified notations in the sequel.
$$\Delta_\lambda:=\left\{\begin{array}{ll}\Delta_\lambda^{(1)}, &n=0,\\ \Delta_\lambda^{(n+1)}, &n\geq 1,\end{array}\right.\ \ {\rm and}\ \ m_\lambda:=\left\{\begin{array}{ll}m_\lambda^{(1)}, &n=0,\\ m_\lambda^{(n+1)}, &n\geq 1.\end{array}\right.$$

Weinstein \cite{MR53289} and Huber \cite{MR64284} considered a singular Laplace equation associated with this operator, which is given by
\begin{equation}\label{egg}
\Delta_{t, x}(u) := -\frac{\partial^2}{\partial t^2} u - \frac{\partial^2}{\partial x^2} u - \frac{2\lambda}{x} \frac{\partial}{\partial x} u = 0.
\end{equation}
They conducted an in-depth study on the generalised axially symmetric potentials, and solved fundamental problems related to the solutions of this equation, such as uniqueness problem,  extension problems   and boundary value problem for certain domains. Later, Muckenhoupt and Stein \cite{MS} introduced a notion of conjugacy  associated with $\Delta_{\lambda}^{(1)}$ and established  $L^p(\mathbb{R}_{+},dm_\lambda^{(1)})-$boundedness of Bessel Riesz transform, conjugate function, and fractional integrals associated with $\Delta_\lambda^{(1)}$. To further enrich this theory, lots of significant results in classical harmonic analysis has been established in the Bessel setting over the last twenty decades (See \cite{MR3922045,MR2496404,MR4361597,MR3829612,MR3689327,MR4231681,MR3481176} and the references therein). Some closely related works include the characterization of boundedness and compactness for Riesz transform commutator $[b,R_{\lambda,j}]$ (See \cite{DGKLWY,MR3829612,MR3689327}), where  $R_{\lambda,j}$ denotes the $j$-th Bessel Riesz transform 
$$ R_{\lambda,j} := {\partial\over \partial x_j} \Delta_\lambda^{-{1\over2}},\quad j=1,\ldots,n+1. $$
\subsection{Our aim}
To the best of our knowledge, the only known result regarding the Schatten--Lorentz class characterization for Riesz transform commutator is due to a pioneering work of Rochberg--Semmes \cite{RS}, where they  established such a  characterization in the classical setting. Despite the establishment of Schatten (and weak-Schatten) class characterizations for Riesz transform commutators in several contexts over the past decade, exploration of their Schatten--Lorentz class characterization remains largely unexplored. Moreover, for the Riesz transform commutator associated with the Bessel operator, there have been no results regarding their Schatten--Lorentz class characterization, even in the special case $p=q$ (i.e. Schatten class characterization).
The aim of this article is to establish  the Schatten--Lorentz characterization for Riesz transform commutator
$[b,R_{\lambda,j}]$ as mentioned above.

\subsection{Statement of main results}

\begin{definition}\label{oscdef}
Suppose $0<p<\infty$, $0<q\leq \infty$, $\lambda>0$ and $n\geq 0$. Then we say that $f$ belongs to oscillation space ${\rm OSC}_{p,q}(\mathbb{R}_+^{n+1},dm_\lambda)$ if $$\|f\|_{{\rm OSC}_{p,q}(\mathbb{R}_+^{n+1},dm_\lambda)}:=\sum_{\nu=1}^\kappa\|\{MO_Q(f)\}_{Q\in\mathcal{D}^\nu}\|_{\ell^{p,q}}<+\infty,$$
where $\{\mathcal{D}^\nu\colon \nu=1,2,\ldots ,\kappa\}$ is a collection of
adjacent systems of dyadic cubes over $\mathbb{R}_{+}^{n+1}$ (see Section \ref{s2} for precise definition), and $MO_E(f):=\fint_{E}|f(x)-(f)_E|dm_\lambda(x)$ denotes the mean oscillation of $f$ over a measurable set $E$.
\end{definition}

In the sequel, we denote $S_\lambda^{p,q}:=S^{p,q}(L^2(\mathbb{R}_{+}^{n+1},dm_\lambda))$ for short. Then our first main theorem can be stated as follows.

\begin{theorem}\label{mainLor}
Suppose $1<p< \infty$, $1\leq q\leq \infty$, $\lambda>0$, $n\geq 0$ and $b\in L^p_{\rm loc}(\mathbb{R}^{n+1}_+,dm_\lambda)$. Then for any $\ell\in\{1,2,...,n+1\}$, we have
\begin{align*}
\|[b,R_{\lambda,\ell}]\|_{S_\lambda^{p,q}}\approx\|b\|_{{\rm OSC}_{p,q}(\mathbb{R}_+^{n+1},dm_\lambda)}.
\end{align*}
\end{theorem}
Our next result demonstrates that for the diagonal case $p=q$, the Schatten norm of these commutators can be further characterized in terms of Besov norm of the symbol, which extends the  classical Euclidean  results of Peller \cite{P}, Janson--Wolff \cite{JW} and Rochberg--Semmes \cite{RS} to the Bessel setting. To illustrate this result,
for any $x\in\mathbb{R}_+^{n+1}$ and $r>0$, denote by $B_{\mathbb{R}_{+}^{n+1}}(x,\,r):=B(x,\,r)\cap \mathbb{R}_+^{n+1} $ the Bessel ball in $\mathbb{R}_{+}^{n+1}$, where $B(x,\,r)$ is a Euclidean ball with center $x$ and radius $r$. We next provide the definition of $B_{p}(\mathbb{R}_+^{n+1},dm_\lambda)$ norm, which is a natural extension of the classical Besov norm $B_{p,q}^{\alpha}(\mathbb{R}^{n+1})$  in the Bessel setting, with $q=p$ and $\alpha=\frac{n+1}{p}\in (0,1)$.
\begin{definition}\label{beso}
Suppose $n+1< p< \infty$, $\lambda>0$ and $n\geq 0$.  Then we say that a  function $f\in L^p_{{\rm loc}}(\mathbb{R}_+^{n+1},dm_\lambda)$ belongs to weighted Besov space $B_{p}(\mathbb{R}_+^{n+1},dm_\lambda)$ if
\begin{align*}
\|f\|_{B_{p}(\mathbb{R}_+^{n+1},dm_\lambda)}:=\left(\int_{\mathbb{R}_{+}^{n+1}}\int_{\mathbb{R}_{+}^{n+1}}\frac{|f(x)-f(y)|^p}{{m_\lambda(B_{\mathbb{R}_+^{n+1}}}(x,\,|x-y|))^2}dm_\lambda(y)dm_\lambda(x)\right)^{1/p}<+\infty.
\end{align*}
\end{definition}

\begin{theorem}\label{schatten}
Suppose $n+1<p<\infty$, $\lambda>0$, $n\geq 0$ and $b\in  L^p_{{\rm loc}}(\mathbb R_+^{n+1},dm_\lambda)$. Then for any $\ell\in\{1,2,...,n+1\}$, we have  $[b,R_{\lambda,\ell}]\in S_\lambda^p$
if and only if $b\in B_{p}(\mathbb R_+^{n+1},dm_\lambda)$.  Moreover, we have $$ \|[b,R_{\lambda,\ell}]\|_{S_\lambda^p}\approx \|b\|_{B_{p}(\mathbb R_+^{n+1},dm_\lambda)}.$$
%% ENUMERATE

\end{theorem}
Parallel to \cite{JW,RS}, it is natural to explore the cut--off point phenomenon in the Bessel setting.
The following theorem illustrates that under a priori assumption $b\in C^2(\mathbb{R}^{n+1}_+)$ with $n\geq 1$ (whether this assumption can be removed or not is still open), the Besov norm (or more generally, the {\rm OSC} norm) of $b$ also collapses to  constants for small index $p$ as in the classical setting.
\begin{theorem}\label{schatten1}
Let $\lambda>0$, $n\geq 1$ and $b\in C^2(\mathbb{R}^{n+1}_+)$. Suppose that $0<p<n+1$, $0<q\leq \infty$, or $p=n+1$, $0< q<\infty$. Then for any $\ell\in\{1,2,...,n+1\}$, we have  $[b,R_{\lambda,\ell}]\in S_\lambda^{p,q}$
if and only if $b$ is a constant on $\mathbb{R}_+^{n+1}$.

In particular, for any $0<p\leq n+1$, $[b,R_{\lambda,\ell}]\in S_\lambda^{p}$
if and only if $b$ is a constant on $\mathbb{R}_+^{n+1}$.
%% ENUMERATE
\end{theorem}
\begin{remark}
It is worth noting that $p=n+1$ is the critical index in Theorem \ref{schatten1}. That is, there does not exist any $p_0 > n+1$ such that if $[b,R_{\lambda,\ell}]\in S_\lambda^{p}$ (or more generally, $[b,R_{\lambda,\ell}]$ is in the Schatten--Lorentz class) for some $n+1<p\leq p_0$, then $b$ is a constant. This fact can be seen by Theorem \ref{schatten} in combination with the Appendix (Section \ref{Appe}). It is also worthwhile to note that $n+1$ is the lower dimension of $(\mathbb{R}_{+}^{n+1},\,|\cdot |,\,dm_\lambda)$, which did not play any role when characterizing the boundedness \cite{DGKLWY} and compactness \cite{MR3829612} of the Riesz transform commutator.
\end{remark}

\subsection{Difficulties and our strategies}

Our strategy of characterizing the Schatten--Lorentz norm is to use the nearly weakly orthogonal sequences (NWO) as introduced in \cite{RS}. Comparing to the framework in standard Euclidean space \cite{RS}, there are two main difficulties in the establishment of Schatten--Lorentz estimate of Riesz transform commutator in the Bessel setting.

$\bullet$ Firstly, Bessel--Riesz transform is a non-convolution type singular integral with Bessel measure, which means that Fourier analysis cannot be applied as effectively as in the classical setting. To overcome this difficulty, we will initiate a flexible approach to decompose the Bessel--Riesz transform kernel and its inverse locally into summations of NWO sequences with variable-separable forms and  suitable decay factors (see Lemmas \ref{jqss} and \ref{jq}). This will be achieved by combining a very recent ingenious construction of weighted Alpert wavelet established in \cite{MR4179877} with the derivation of a refined lower bound of Riesz transform kernel.

$\bullet$ Secondly, Bessel measure is not invariant under translation. Furthermore, $(\mathbb{R}_{+}^{n+1},\,|\cdot |,\,dm_\lambda)$ is a reverse doubling space, where the upper dimension differs from the lower one, in contrast to a regular space where both dimensions are equal. Consequently, even for compactly supported smooth function $\varphi$, an equality like $\frac{1}{t^{n+1}}\int_{\mathbb{R}_+^{n+1}}\varphi(\frac{x-y}{t})dm_\lambda(y)=1$ is not valid. To the best of our knowledge, similar issues never appear in the existing literature concerning Schatten (or weak Schatten) class characterization of Riesz transform commutators. To surmount these obstacles, we use the Bessel heat kernel as a suitable substitution of Euclidean convolution with compactly supported smooth function.

\subsection{Notation and organization of the paper}
The paper is organized as follows. Section \ref{Prel} consists of five parts:
the first part recalls the concept of adjacent systems of dyadic cubes on  spaces of homogenous type; the second part recalls the celebrated construction of weighted Alpert wavelets for $L^2(\mathbb{R}^n,\mu)$; the third part recalls the useful notion of nearly weakly orthogonal sequences of functions, and recall their closed relationship with Schatten--Lorentz estimate of compact operator on Hilbert space; the four and five parts establish the Heat kernel estimate and Riesz transform kernel estimate associated with Bessel operator, respectively, including pointwise estimates and smoothness estimates. At the end of the five part, we also establish a non-degenerate lower bound of Riesz transform kernel associated with Bessel operator. In Section \ref{Sec3}, we  establish a summability self-increased Lemma, which states that the definition of oscillation space norm does not depend on the integrability index of mean oscillation. In Sections \ref{Sec4} and \ref{Sec5}, we establish the upper bound and lower bound of Theorem \ref{mainLor}, respectively. In Section \ref{Sec6}, we build the bridge  between the oscillation space  and Besov space in the Bessel setting, and then provide the proof of Theorems \ref{schatten} and \ref{schatten1}.

\section{Notation and Preliminaries}\label{Prel}
\setcounter{equation}{0}

\subsection{Notation}
Conventionally, we set $\mathbb{N}$ be the set of positive integers and $\mathbb{Z}_+:=\{0\}\cup\mathbb{N}$. Throughout the whole paper, we denote by $C$ a positive constant which is independent of the main parameters, but it may
vary from line to line. We use $A\lesssim B$ to denote the statement that $A\leq CB$ for some constant $C>0$, and $A\thicksim B$ to denote the statement that $A\lesssim B$ and $B\lesssim A$. For any $\lambda>0$ and any Bessel ball $B=B_{\mathbb{R}_+^{n+1}}(x_B,\,r_B)$, we use the notation $\lambda B$ to denote the ball centered at $x_B$ with radius $\lambda r_B$. For any $1\leq p\leq\infty$, we denote by $p'$ the conjugate of $p$, which satisfies $\frac{1}{p}+\frac{1}{p'}=1$. We denote the average of a function $f$ over a ball $B$ by
\begin{align*}
(f)_B:=\fint_{B}f(x)dm_\lambda(x):=\frac{1}{m_\lambda(B)}\int_{B}f(x)dm_\lambda(x).
\end{align*}
For a $m_\lambda$-measurable set $E \subset \mathbb{R}_+^{n+1}$, we denote by $\chi_E$ its characteristic function. For two measurable Borel sets $E$ and $F$, we use the notation ${\rm dist}(E,\,F)$ to denote the distance between them.

{\bf Convention}: {\it In the sequel, unless  otherwise specific, we always assume that $\lambda>0$ and $n\geq 0$.}

\subsection{Preliminaries on space of homogeneous type}\label{s2}
For our purpose, we usually regard $\mathbb{R}_{+}^{n+1}$ ($n\geq 0$) as a space of homogeneous type, in the sense of Coifman and Weiss \cite{CWbook}, with Euclidean metric and weighted measure $dm_\lambda$.  It can be deduced from \cite{MR3829612} that for every $x=(x_1,\ldots,x_{n+1})\in\mathbb{R}_{+}^{n+1}$ and $r>0$,
\begin{align}\label{measureeee}
m_\lambda(B_{\mathbb{R}_+^{n+1}}(x,\,r))\sim r^{n+1}x_{n+1}^{2\lambda}+r^{n+1+2\lambda}.
\end{align}
Therefore, $(\mathbb{R}_{+}^{n+1},dm_\lambda)$ satisfies the following doubling inequality:  for every $x\in\mathbb{R}_{+}^{n+1}$ and $r>0$,
\begin{align}
\label{doub}2^{n+1}m_\lambda(B_{\mathbb{R}_{+}^{n+1}}(x,\,r))\leq m_\lambda(B_{\mathbb{R}_{+}^{n+1}}(x,\,2r))\leq 2^{2\lambda+n+1}m_\lambda(B_{\mathbb{R}_{+}^{n+1}}(x,\,r)).
\end{align}

In the following, we collect some properties about system of dyadic cubes on homogeneous space and adapt it to $\mathbb{R}_+^{n+1}$. A
countable family
$
    \mathcal{D}
    := \cup_{k\in\mathbb{Z}}\mathcal{D}_{k},$ where $
    \mathcal{D}_{k}
    :=\{Q^k_{\alpha}\colon \alpha\in \mathcal{A}_k\},
$
of Borel sets $Q^k_{\alpha}\subseteq \mathbb{R}_+^{n+1}$ is called \textit{a
system of dyadic cubes with parameter} $\delta\in (0,1)$  if it has the following properties:

\smallskip
(I) $    \mathbb{R}_+^{n+1}
    = \bigcup_{\alpha\in \mathcal{A}_k} Q^k_{\alpha}
    \quad\text{(disjoint union) for all}~k\in\Z$;

\smallskip
(II) $
    \text{If }\ell\geq k\text{, then either }
        Q^{\ell}_{\beta}\subseteq Q^k_{\alpha}\text{ or }
        Q^k_{\alpha}\cap Q^{\ell}_{\beta}=\emptyset$;

\smallskip
(III) $\text{For each }(k,\alpha)\text{ and each } \ell\leq k,
    \text{ there exists a unique } \beta
    \text{ such that }Q^{k}_{\alpha}\subseteq Q^\ell_{\beta}$;

\smallskip
(IV)
     \text{For each $(k,\alpha)$ there exists at most $M$
        (a fixed geometric constant)  $\beta$ such that }
     $$ Q^{k+1}_{\beta}\subseteq Q^k_{\alpha},\ {\rm and}\
        Q^k_{\alpha} =\bigcup_{{\substack{Q\in\mathcal{D}_{k+1}\\ Q\subseteq Q^k_{\alpha}}}}Q;$$

\smallskip
(V) For each $(k,\alpha)$, we have
\begin{equation}\label{eq:contain}
B_{\mathbb{R}_{+}^{n+1}}(x^k_{\alpha},\,\frac{1}{12}\delta^k)
    \subseteq Q^k_{\alpha}\subseteq B_{\mathbb{R}_{+}^{n+1}}(x^k_{\alpha},\,4\delta^k)
    =: B_{\mathbb{R}_{+}^{n+1}}(Q^k_{\alpha});
\end{equation}

\smallskip
(VI) \text{If }$\ell\geq k$\text{ and }
   $Q^{\ell}_{\beta}\subseteq Q^k_{\alpha}$\text{, then }
$$B_{\mathbb{R}_{+}^{n+1}}(Q^{\ell}_{\beta})\subseteq B_{\mathbb{R}_{+}^{n+1}}(Q^k_{\alpha}).$$
The set $Q^k_{\alpha}$ is called a \textit{dyadic cube of
generation} $k$ with center point $x^k_{\alpha}\in Q^k_{\alpha}$
and sidelength~$\delta^k$. In the sequel, for any cube $Q\subset\mathbb{R}_+^{n+1}$, we denote by $\ell(Q)$ the sidelength of $Q$.

It can be deduced from the properties of the dyadic system above that there exists a positive constant
$C_{0}$ such that for any $Q^k_{\alpha}$ and $Q^{k+1}_{\beta}$  with $Q^{k+1}_{\beta}\subset Q^k_{\alpha}$,
\begin{align}\label{Cmu0}
m_\lambda(Q^{k+1}_{\beta})\leq m_\lambda(Q^k_{\alpha})\leq C_0\delta^{-(2\lambda+n+1)}m_\lambda(Q^{k+1}_{\beta}).
\end{align}

A
system of dyadic cubes on $\mathbb{R}_{+}^{n+1}$ can be constructed in a standard way   as follows: let $\mathcal{D}^0:=\cup_{k\in\mathbb{Z}}\mathcal{D}_{k}^0$, where $\mathcal{D}_{k}^0$ is the standard dyadic partition of $\mathbb{R}_+^{n+1}$ into cubes with vertices at the sets $\{(2^{-k}m_1,\ldots,2^{-k}m_{n+1}):(m_1,\ldots,m_{n+1})\in\mathbb{Z}^{n}\times \mathbb{N}\}$.

A
finite collection $\{\mathcal{D}^\nu\colon \nu=1,2,\ldots ,\kappa\}$ of the dyadic
families  is called a collection of
adjacent systems of dyadic cubes over $\mathbb{R}_{+}^{n+1}$ with parameters $\delta\in
(0,1) $ and $1\leq C_{\rm adj}<\infty$ if it satisfies the
following properties:

$\bullet$ Each $\mathcal{D}^\nu$ is a
system of dyadic cubes with parameter $\delta\in (0,1)$;

$\bullet$ For each ball
$B_{\mathbb{R}_+^{n+1}}(x,\,r)\subseteq \mathbb{R}_+^{n+1}$ with $\delta^{k+3}<r\leq\delta^{k+2},
k\in\Z$, there exist $\nu \in \{1, 2, \ldots, \kappa\}$ and
$Q\in\mathcal{D}_{k}^\nu$ of generation $k$ and with center point
$^\nu x^k_{\alpha}$ such that $|x-{}^\nu x_{\alpha}^k| <
2\delta^{k}$ and
\begin{equation}\label{eq:ball;included}
    B_{\mathbb{R}_+^{n+1}}(x,\,r)\subseteq Q\subseteq B_{\mathbb{R}_+^{n+1}}(x,\,C_{\rm adj}r).
\end{equation}

We adapt the construction in \cite{HK} to our setting, which can be stated as follows.
\begin{lemma}\label{thm:existence2}
On $\mathbb{R}_+^{n+1}$ with Euclidean metric and weighted measure $dm_\lambda$, there exists a collection $\{\mathcal{D}^\nu\colon
    \nu = 1,2,\ldots ,\kappa\}$ of adjacent systems of dyadic cubes with
    parameters $\delta\in (0, \frac{1}{96}) $ and $C_{\rm adj} := 8\delta^{-3}$ such that the center points
    $^\nu x^k_{\alpha}$ of the cubes $Q\in\mathcal{D}^\nu_{k}$ satisfy, for each
    $\nu \in\{1,2,\ldots,\kappa\}$,
    \begin{equation*}
        |^\nu x_{\alpha}^k- {}^\nu x_{\beta}^k|
        \geq \frac{1}{4}\delta^k\quad(\alpha\neq\beta),\qquad
        \min_{\alpha}|x-{}^\nu x^k_{\alpha}|
        < 2 \delta^k\quad \text{for all}~x\in \mathbb{R}_+^{n+1}.
    \end{equation*}
    Furthermore, these adjacent systems can be constructed in such a
    way that each $\mathcal{D}^\nu$ satisfies the distinguished
    center point property: given a fixed point $x_{0}\in \mathbb{R}_{+}^{n+1}$, for every $k\in \Z$, there exists $\alpha\in\mathcal{A}_k$ such that
    $x_{0}
        = x^k_{\alpha},\text{ the center point of }
        Q^k_{\alpha}\in\mathcal{D}_{k}^\nu.$
\end{lemma}

\subsection{Weighted Alpert wavelets}
In this subsection, we recall the celebrated construction of weighted Alpert wavelets for $L^2(\mathbb{R}^n,\mu)$ from \cite[Theorem 1]{MR4179877} (see also \cite[Formula (2.9)]{MR4554083} for the $L^\infty$ control of projections) with $\mu$ being a doubling measure. These wavelets have three important properties of orthogonality, telescoping and moment vanishing, so in this paper, we will regard weighted Alpert wavelets as an appropriate substitute of Fourier expansion (see the proof of Lemmas \ref{jqss} and \ref{jq} for this argument).

Let $\mu$ be a locally finite positive doubling Borel measure on $\mathbb{R}^n$ and let $\mathcal{K} \in \mathbb{N}$. For any $Q$ with side parallel to the coordinate axes, denote by $\mathfrak{C}(Q)$ the sets of $2^n$ subcubes with disjoint interior points such that the side lengths of these subcubes are half the side length of $Q$. Moreover, we denote by $L^2_{Q,\mathcal{K}}(\mu)$ the finite dimensional subspace of $L^2(\mathbb{R}^n,\mu)$ that consists of linear combinations of the indicators of the children $\mathfrak{C}(Q)$ of $Q$ multiplied by polynomials of degree less than $\mathcal{K}$, and such that the linear combinations have vanishing $\mu$-moments on the cube $Q$ up to order $\mathcal{K} - 1$:
\begin{equation}
L^2_{Q,\mathcal{K}}(\mu) :=
\left\{
f = \sum_{Q' \in \mathfrak{C}(Q)} 1_{Q'} p_{Q',\mathcal{K}}(x) : \int_Q f(x) x^\beta d\mu(x) = 0, \text{ for } 0 \leq | \beta | < \mathcal{K}
\right\},
\end{equation}
where $p_{Q,\mathcal{K}}(x) = \sum_{\beta \in \mathbb{Z}_+^n : | \beta | \leq \mathcal{K} - 1} a_{Q',\beta} x^\beta$ is a polynomial in $\mathbb{R}^n$ of degree less than $\mathcal{K}$. Here $x^\beta = x_1^{\beta_1} x_2^{\beta_2} \cdots x_n^{\beta_n}$. Let $d_{Q,\mathcal{K}} := \dim L^2_{Q,\mathcal{K}}(\mu)$ be the dimension of the finite dimensional linear space $L^2_{Q,\mathcal{K}}(\mu)$. Let $\mathfrak{D}$ denote a dyadic grid on $\mathbb{R}^n$ and for $Q \in \mathfrak{D}$, let $\Delta^\mu_{Q,\mathcal{K}}$ denote orthogonal projection onto the finite dimensional subspace $L^2_{Q,\mathcal{K}}(\mu)$, and let $E^\mu_{Q,\mathcal{K}}$ denote orthogonal projection onto the finite dimensional subspace
$
P^n_{Q,\mathcal{K}}(\mu):= {\rm Span}\{1_{Q}x^\beta : 0 \leq | \beta | < \mathcal{K}\}.
$
\begin{lemma}\label{APWa}\cite[Theorem 1]{MR4179877}
Let $\mu$ be a doubling measure on $\mathbb{R}^n$. Let $\mathcal{K} \in \mathbb{N}$  and $\mathfrak{D}$ be a dyadic grid  in $\mathbb{R}^n$. Then we have the following statements.
\begin{enumerate}
  \item $\{\Delta^\mu_{Q,\mathcal{K}} \}_{Q \in D}$ is a complete set of orthogonal projections in $L^2(\mathbb{R}^n,\mu)$ and
  \begin{align*}
  &f = \sum_{Q \in \mathfrak{D}} \Delta^\mu_{Q,\mathcal{K}} f, \quad f \in L^2(\mathbb{R}^n,\mu),\\
  &\langle \Delta^\mu_{P,\mathcal{K}}f,\,\Delta^\mu_{Q,\mathcal{K}}f \rangle_{L^2(\mathbb{R}^n,\mu)}=0,\ {\rm for}\ P\neq Q,
  \end{align*}
  where convergence in the sum holds both in $L^2(\mathbb{R}^n,\mu)$ and pointwise $\mu$-almost everywhere.
  \item The following telescoping identities hold:
  \begin{equation}
  1_Q \sum_{I:Q\varsubsetneqq I\subset P}\Delta^\mu_{I,\mathcal{K}}=\mathbb{E}^\mu_{Q,\mathcal{K}}-\chi_Q \mathbb{E}^\mu_{P,\mathcal{K}} \quad \text{for }\ P, Q \in \mathfrak{D} \ \text{ with }\ Q \varsubsetneqq P.
  \end{equation}
  \item The following moment vanishing conditions hold:
  \begin{equation}
  \int_{\mathbb{R}^n} \Delta^\mu_{Q,\mathcal{K}} f(x) x^\beta d\mu(x) = 0, \quad \text{for }\ Q \in \mathfrak{D},\, \beta \in \mathbb{Z}^n,\, 0 \leq | \beta | < \mathcal{K}.
  \end{equation}
\end{enumerate}
\end{lemma}
From the construction of weighted Alpert wavelet in \cite[Theorem 1]{MR4179877} we see that with some trivial modification of their proof, the conclusion in Lemma \ref{APWa} is still valid if dyadic grid $\mathfrak{D}$ is replaced by any system of dyadic cubes $\mathcal{D}$ defined in Section \ref{s2}. Moreover, by letting $\{h_{Q,\mathcal{K}}^\epsilon\}_{\epsilon\in\Gamma_{n,k}}$ be an orthonormal basis of $L^2_{Q,\mathcal{K}}(\mathbb{R}^n,\mu)$, where $\Gamma_{n,\mathcal{K}}$ is a finite index set only depending on $n$ and $\mathcal{K}$ (See also explanation in \cite[Notation 10]{MR4554083}), recalling the $L^\infty$ control of projections in \cite[Formula (2.9)]{MR4554083} and then applying Lemma \ref{APWa} to the Bessel setting, we have the following version of weighted Alpert wavelet.
\begin{lemma}\label{AW}
Fix $\mathcal{K}\in\mathbb{N}$ and a system of dyadic cubes $\mathcal{D}$. Then there exist a finite index set $\Gamma_{n,\mathcal{K}}$ and a family of functions $\{h_{Q,\mathcal{K}}^\epsilon\}_{Q\in\mathcal{D},\epsilon\in\Gamma_{n,\mathcal{K}}}$ such that
%\begin{enumerate}
  $$\|h_{Q,\mathcal{K}}^\epsilon\|_{L^2(\mathbb{R}_+^{n+1},dm_\lambda)}=1,$$
  $$\supp h_{Q,\mathcal{K}}^\epsilon\subset Q, $$
  $$\|h_{Q,\mathcal{K}}^\epsilon\|_\infty\lesssim m_\lambda(Q)^{-1/2},$$
  $$\int_Q h_{Q,\mathcal{K}}^\epsilon(x)x^\alpha d\mu(x)=0,\ {\rm for}\ {\rm all}\ 0\leq |\alpha|<\mathcal{K},$$
and for any $f\in L^2(\mathbb{R}_+^{n+1},dm_\lambda)$, we have
  $$f=\sum_{Q\in\mathcal{D},\epsilon\in\Gamma_{n,k}}\langle f,\,h_{Q,\mathcal{K}}^\epsilon\rangle h_{Q,\mathcal{K}}^\epsilon,$$
  where the sum converges in $L^2(\mathbb{R}_+^{n+1},dm_\lambda)$ and pointwise $m_\lambda$-almost everywhere.
%\end{enumerate}
\end{lemma}

\subsection{Nearly weakly orthogonal sequences}\label{nwooo}
The notion of \emph{nearly weakly orthogonal  (NWO)} sequences of functions proposed by Rochberg and Semmes \cite{RS} plays a crucial role in establishing Schatten--Lorentz membership for compact operator on Hilbert space. This concept is formulated in the context of Euclidean space equipped with Lebesgue measure, but it can be adapted parallelly to many other setting beyond Euclidean framework. Fur our purpose, it is not necessary to recall the original definitions of this concept. Instead, it suffices to know that a sequence of functions $\{e_Q\}_{Q\in\mathcal{D}}$ with $\supp e_Q\subset Q$ and $\|e_Q\|_{L^p(\mathbb{R}_+^{n+1},dm_\lambda)}\leq m_\lambda(Q)^{1/p-1/2}$ for some $p>2$ is NWO (see \cite[page 239--240]{RS}), where $\mathcal{D}$ is a given system of dyadic cubes. Furthermore, their arguments (see \cite[(1.10) and Corollary 2.7]{RS}) adapted to the Bessel setting say that if $0<p<\infty$, $0<q\leq\infty$, then for any
bounded compact operator $ A $ on $ L ^2 (\mathbb R_+^{n+1},dm_\lambda)$ given by
$$A=\sum_{Q\in\mathcal{D}}\lambda_Q\langle \cdot,\,e_Q\rangle f_Q,$$
where $e_Q$ and $f_Q$ are NWO sequences and $\lambda_Q$ is a sequence of scalars, we have
\begin{align}\label{nwose}
\|A\|_{S_\lambda^{p,q}}\lesssim \|\{\lambda_Q\}_Q\|_{\ell^{p,q}}.
\end{align}

\subsection{Heat kernel estimate associated with Bessel operator}
For any $f$, $g\in L^1(\mathbb{R}_+,dm_\lambda)$, we define their Hankel convolution by
$$f\sharp_\lambda g(x):=\int_0^\infty f(y)\tau_x^{[\lambda]}g(y)dm_\lambda(y),\ {\rm for}\ {\rm all}\ x\in \mathbb{R}_+,$$
where $\tau_x^{[\lambda]}$ denotes the Hankel translation of $g$, which can be expressed as
$$\tau_x^{[\lambda]}g(y):=\frac{\Gamma(\lambda+1/2)}{\Gamma(\lambda)\sqrt{\pi}}\int_0^\pi g\Big(\sqrt{x^2+y^2-2xy\cos\theta}\Big)(\sin\theta)^{2\lambda-1}d\theta.$$
Let $\Gamma(x)$ denote the Gamma function. For any function $\phi$, we denote $\phi_t(y):=t^{-2\lambda-1}\phi(y/t)$. Then recall from \cite{MR2496404} that \begin{align}\label{heatex}
e^{-t\Delta_\lambda}f=f\sharp_\lambda W_{\sqrt{2t}}^{[\lambda]}
\end{align}
 for all $t\in\mathbb{R}_+$, where we denote
$$W^{[\lambda]}(x)=2^{(1-2\lambda)/2}\exp(-x^2/2)/\Gamma(\lambda+1/2).$$

The main result of this subsection is to establish the following heat kernel estimate associated with Bessel operator.
\begin{lemma}\label{heat kernel}
For all multi-index $\alpha\in\mathbb{N}^{n+1}$, there exist constants $C,C_\alpha,c>0$  such that
\begin{align*}
&(1) \ |K_{e^{-t^2 \Delta_\lambda}}(x,\,y)|\leq \frac{C}{m_\lambda(B_{\mathbb{R}_+^{n+1}}(x,\,t))}\exp\left(-c\frac{|x-y|^2}{t^2}\right),\\
&(2) \ |\partial_tK_{ e^{-t^2 \Delta_\lambda} }(x,\,y)|\leq \frac{C}{tm_\lambda(B_{\mathbb{R}_+^{n+1}}(x,\,t))}\exp\left(-c\frac{|x-y|^2}{t^2}\right),\\
&(3) \ |\partial_{x}^\alpha K_{ e^{-t^2 \Delta_\lambda} }(x,\,y)|\leq \frac{C_\alpha}{t^{|\alpha|}m_\lambda(B_{\mathbb{R}_+^{n+1}}(x,\,t))}\exp\left(-c\frac{|x-y|^2}{t^2}\right),\\
&(4) \ \int_{0}^\infty K_{e^{-t^2 \Delta_\lambda}}(x,\,z)dm_\lambda(z)=1
\end{align*}
for all $x,y\in\mathbb{R}_+^{n+1}$ and $t\in\mathbb{R}_+$.
\end{lemma}
\begin{proof}
We first show Lemma \ref{heat kernel} in the case of $n=1$. In this case, it follows from \eqref{heatex} that
\begin{align}
K_{e^{-t^2 \Delta_\lambda^{(1)}}}(x,\,y)
&=\frac{\Gamma(\lambda+1/2)}{\Gamma(\lambda)\sqrt{\pi}}\int_0^\pi W_{\sqrt{2}t}^{[\lambda]}\Big(\sqrt{x^2+y^2-2xy\cos\theta}\Big)(\sin\theta)^{2\lambda-1}d\theta\nonumber\\
&=2^{(1-2\lambda)/2}\frac{\Gamma(\lambda+1/2)}{\Gamma(\lambda)\sqrt{\pi}}\int_0^\pi (\sqrt{2}t)^{-2\lambda-1}\exp\left(-\frac{x^2+y^2-2xy\cos\theta}{4t^2}\right)(\sin\theta)^{2\lambda-1}d\theta\label{middd}\\
&\sim t^{-2\lambda-1}\int_0^\pi \exp\left(-\frac{x^2+y^2-2xy\cos\theta}{4t^2}\right)(\sin\theta)^{2\lambda-1}d\theta.\label{midd}
\end{align}
To continue, we divide the proof into two cases.

{\bf Case (i).} If $t\geq x$ or $|x-y|\geq x/2$, then it follows from \eqref{midd} and the doubling condition \eqref{doub} that
\begin{align*}
K_{e^{-t^2 \Delta_\lambda^{(1)}}}(x,\,y)&\lesssim t^{-2\lambda-1} \exp\left(-\frac{|x-y|^2}{4t^2}\right)
\lesssim \frac{1}{m_\lambda(B_{\mathbb{R}_+}(x,\,t))}\exp\left(-\frac{|x-y|^2}{8t^2}\right).
\end{align*}

{\bf Case (ii).} If $t< x$ and $|x-y|< x/2$, then observe that $x\sim y$, $\sin\theta\sim \theta$ for all $\theta\in (0,\pi/2)$ and $1-\cos\theta\geq 2(\theta/\pi)^2$. Now we decompose the right-hand side of \eqref{midd}   as
\begin{align*}
&Ct^{-2\lambda-1}\int_0^{\pi/2} \exp\left(-\frac{|x-y|^2+2xy(1-\cos\theta)}{4t^2}\right)(\sin\theta)^{2\lambda-1}d\theta\\
&\hspace{1.0cm}+Ct^{-2\lambda-1}\int_{\pi/2}^\pi \exp\left(-\frac{x^2+y^2-2xy\cos\theta}{4t^2}\right)(\sin\theta)^{2\lambda-1}d\theta=:{\rm A}+{\rm B}.
\end{align*}
For the term ${\rm A}$, by a change of variable, we have
\begin{align*}
{\rm A}&\lesssim t^{-2\lambda-1}\int_0^{\pi/2} \exp\left(-\frac{|x-y|^2+4xy\theta^2/\pi^2}{4t^2}\right)\theta^{2\lambda-1}d\theta\\
&\lesssim t^{-2\lambda-1}\exp\left(-\frac{|x-y|^2}{4t^2}\right)\int_0^{\infty} \exp\left(-\frac{4xy\theta^2/\pi^2}{4t^2}\right)\theta^{2\lambda-1}d\theta\\
&\lesssim (tx^{2\lambda}y^{2\lambda})^{-1}\exp\left(-\frac{|x-y|^2}{4t^2}\right)\\
&\lesssim \frac{1}{m_\lambda(B_{\mathbb{R}_+^{n+1}}(x,\,t))}\exp\left(-\frac{|x-y|^2}{4t^2}\right).
\end{align*}
For the term ${\rm B}$, we note that $\cos\theta<0$ for all $\theta\in (\pi/2,\pi)$. Thus,
\begin{align*}
{\rm B}&\lesssim t^{-2\lambda-1}\exp\left(-\frac{x^2}{4t^2}\right)\exp\left(-\frac{|x-y|^2}{4t^2}\right)\int_{\pi/2}^\pi (\sin\theta)^{2\lambda-1}d\theta\\
&\lesssim (tx^{2\lambda})^{-1}\exp\left(-\frac{|x-y|^2}{4t^2}\right)\\
&\lesssim \frac{1}{m_\lambda(B_{\mathbb{R}_+}(x,\,t))}\exp\left(-\frac{|x-y|^2}{4t^2}\right).
\end{align*}
This ends the proof of (1).

Now we show (2). It follows from \eqref{middd} that
\begin{align}\label{t2220}
\partial_t K_{e^{-t^2 \Delta_\lambda^{(1)}}}(x,\,y)
&\sim  t^{-2\lambda-2}\int_0^\pi \exp\left(-\frac{x^2+y^2-2xy\cos\theta}{4t^2}\right)(\sin\theta)^{2\lambda-1}d\theta\nonumber\\
&\hspace{1.0cm}+ t^{-2\lambda-4}\int_0^\pi (x^2+y^2-2xy\cos\theta) \exp\left(-\frac{x^2+y^2-2xy\cos\theta}{4t^2}\right)(\sin\theta)^{2\lambda-1}d\theta\nonumber\\
&\lesssim t^{-2\lambda-2}\int_0^\pi \exp\left(-\frac{x^2+y^2-2xy\cos\theta}{8t^2}\right)(\sin\theta)^{2\lambda-1}d\theta.
\end{align}
Note that the integral term above is exactly the integral term on the right-hand side of \eqref{midd}, up to a harmless constant in the exponential term. Therefore, we obtain (2).

Finally, we show (3). It follows from \eqref{middd} that
\begin{align*}
\partial_{x} K_{e^{-t^2 \Delta_\lambda^{(1)}}}(x,\,y)
&=  -C_\lambda\int_0^\pi t^{-2\lambda-1}\frac{x-y\cos\theta}{t^2}\exp\left(-\frac{x^2+y^2-2xy\cos\theta}{4t^2}\right)(\sin\theta)^{2\lambda-1}d\theta\\
&\lesssim t^{-2\lambda-2}\int_0^\pi \exp\left(-\frac{x^2+y^2-2xy\cos\theta}{8t^2}\right)(\sin\theta)^{2\lambda-1}d\theta.
\end{align*}
Note that up to an absolute constant, the term above is equal to the term on the right-hand side of \eqref{t2220}. This finishes the proof of (3) in the case of $\alpha=1$, whereas the proof in the case of $\alpha>1$ can be obtained similarly, so we skip the details. Therefore, we obtain (3), while (4) is a well-known result (see e.g. \cite[Lemma 2.1]{MR2823879}).

Next we turn to the high-dimensional case $n\geq 1$. To begin with, we note that \eqref{Dlambda} can be written as $\Delta_\lambda^{(n+1)}=\Delta^{(n)}+\Delta_\lambda^{(1)}$, where $\Delta^{(n)}$ denotes the standard Laplacian on $\mathbb{R}^n$. Then for any $(x,\,y)=(x',\,x_{n+1},\,y',\,y_{n+1})\in\mathbb{R}_+^{n+1}\times \mathbb{R}_+^{n+1}$, the heat kernel can be decomposed as
$$K_{e^{-t^2 \Delta_\lambda}}(x,\,y)=K_{e^{-t^2 \Delta^{(n)}}}(x',\,y')K_{e^{-t^2 \Delta_\lambda^{(1)}}}(x_{n+1},\,y_{n+1}).$$
Then the estimates of (1)--(4) are reduced to the established one-dimensional case. Therefore, the proof of Lemma \ref{heat kernel} is complete.
\end{proof}
\subsection{Riesz transform kernel estimate associated with Bessel operator}
Denote by $K_{\lambda,\ell}(x,\,y)$ the kernel of the $\ell$-th Riesz transform $R_{\lambda,\ell}$. The first purpose of this subsection is to establish the size estimate and smoothness estimate at any order for this kernel.
\begin{lemma}\label{CZO}
For any $\ell\in\{1,2,\ldots,n+1\}$ and all multi-index $\alpha,\beta\in\mathbb{N}^{n+1}$, there is a constant $C_{\alpha,\beta}>0$ such that
\begin{align}\label{czooo}
|\partial_x^\alpha\partial_y^\beta K_{\lambda,\ell}(x,\,y)|\leq \frac{C_{\alpha,\beta}}{m_\lambda(B_{\mathbb{R}_+^{n+1}}(x,\,|x-y|))|x-y|^{|\alpha|+|\beta|}}
\end{align}
for $x$, $y\in\mathbb{R}_+^{n+1}$.
\end{lemma}
\begin{proof}
To begin with, by spectral theory, we have the following subordinate formula:
\begin{align*}
(-\Delta_\lambda)^{-1/2}=\frac{1}{\sqrt{\pi}}\int_0^\infty t^{-1/2}e^{-t\Delta_\lambda}dt.
\end{align*}
This implies that for any multi-index $\alpha,\beta\in\mathbb{N}^{n+1}$,
\begin{align}\label{exxxt}
|\partial_x^\alpha\partial_y^\beta K_{\lambda,\ell}(x,\,y)|&=\left|\frac{1}{\sqrt{\pi}}\int_0^\infty t^{-1/2} \partial_x^{\alpha}\partial_y^\beta \partial_{x_\ell} K_{e^{-t \Delta_\lambda} }(x,\,y)dt\right|\nonumber\\
&\lesssim \int_0^\infty  \frac{1}{t^{\frac{|\alpha|+|\beta|}{2}}m_\lambda(B_{\mathbb{R}_+^{n+1}}(x,\,\sqrt{t}))}\exp\left(-c\frac{|x-y|^2}{t}\right)\frac{dt}{t}.
\end{align}
To continue, we note that inequality \eqref{doub} implies that
\begin{align*}
\frac{1}{m_\lambda(B_{\mathbb{R}_+^{n+1}}(x,\,\sqrt{t}))}\lesssim \frac{1}{m_\lambda(B_{\mathbb{R}_+^{n+1}}(x,\,|x-y|))}\left(1+\frac{|x-y|}{\sqrt{t}}\right)^{n+1+2\lambda}.
\end{align*}
Since the term $\left(1+\frac{|x-y|}{\sqrt{t}}\right)^{n+1+2\lambda}$ can be absorbed into the exponential term, we have
\begin{align*}
{\rm RHS}\ {\rm of}\ \eqref{exxxt}
&\lesssim \int_0^\infty  \frac{1}{t^{\frac{|\alpha|+|\beta|}{2}}m_\lambda(B_{\mathbb{R}_+^{n+1}}(x,\,|x-y|))}\exp\left(-c\frac{|x-y|^2}{2t}\right)\frac{dt}{t}\\
&\lesssim \frac{1}{m_\lambda(B_{\mathbb{R}_+^{n+1}}(x,\,|x-y|))|x-y|^{|\alpha|+|\beta|}}.
\end{align*}
This ends the proof of Lemma \ref{CZO}.
\end{proof}

The second purpose of this subsection is to establish the following non-degenerate lower bound of Riesz transform kernel associated with Bessel operator, which can be regarded as a suitable substitution of homogeneity in the case of classical Euclidean Riesz transform kernel.
\begin{lemma}\label{sign}
Given $\ell\in\{1,2,\ldots,n+1\}$ and a
system of dyadic cubes $\mathcal{D}
    := \cup_{k\in\mathbb{Z}}\mathcal{D}_{k}$ on $\mathbb{R}_+^{n+1}$ with parameter $\delta\in (0,\,1)$. Then there exist constants  $c_1,c_2>0$ such that for any $Q\in \mathcal{D}_{k}$, we can find a dyadic cube $\hat{Q}\in\mathcal{D}_k$ such that $c_1\delta^k\leq {\rm dist}(Q,\,\hat{Q})\leq c_2\delta^k$, and for all $(x,\,y)\in Q\times\hat{Q}$, $K_{\lambda,\ell}(x,\,y)$ does not change sign and satisfies
\begin{align}\label{looower}
|K_{\lambda,\ell}(x,\,y)|\geq \frac{C}{m_\lambda(Q)}
\end{align}
for some constant $C>0$.
\end{lemma}
\begin{proof}
To begin with, we recall from \cite[Propositions 7.5 and 7.6]{DGKLWY} that there are positive constants $c_0$ and $\bar{C}$ such that for any $x\in \mathbb{R}_+^{n+1}$ and $r>0$ such that there exists $y\in B_{\mathbb{R}_+^{n+1}}(x,\,\bar{C}r)\backslash B_{\mathbb{R}_+^{n+1}}(x,\,r)$ satisfying
$$|K_{\lambda,\ell}(x,\,y)|\geq \frac{c_0}{m_\lambda(B_{\mathbb{R}_+^{n+1}}(x,\,r))}.$$
For any $Q\in\mathcal{D}_k$, we let $x_0:=(x_{(1)},\ldots,x_{(n+1)})$ be its center and let $M$ be a sufficiently large constant to be chosen later. Applying the above fact and then inequality \eqref{doub}, there exists $y_0\in B_{\mathbb{R}_+^{n+1}}(x_0,\,\bar{C}M\delta^k)\backslash B_{\mathbb{R}_+^{n+1}}(x_0,\,M\delta^k)$ satisfying
$$|K_{\lambda,\ell}(x_0,\,y_0)|\geq \frac{c_0}{m_\lambda(B_{\mathbb{R}_+^{n+1}}(x_0,\,M\delta^k))}$$
for some constant $c_0>0$.
To continue, we let $\hat{Q}$ be the unique cube in $\mathcal{D}_k$ such that its closure contains $y_0$. Since $M$ is a sufficiently large constant,  there are  constants  $a_1,a_2>0$ such that $a_1M\delta^k\leq {\rm dist}(Q,\,\hat{Q})\leq a_2M\delta^k$.

Moreover, by Lemma \ref{CZO}, for any $(x,\,y)\in Q\times\hat{Q}$, we have
\begin{align*}
|K_{\lambda,\ell}(x,\,y)-K_{\lambda,\ell}(x_{0},\,y_{0})|&\leq |K_{\lambda,\ell}(x,\,y)-K_{\lambda,\ell}(x,\,y_{0})|+|K_{\lambda,\ell}(x,\,y_{0})-K_{\lambda,\ell}(x_{0},\,y_{0})|\\
&\lesssim\frac{|y-y_{0}|}{m_\lambda(B_{\mathbb{R}_+^{n+1}}(x,\,|x-y|))|x-y|}+\frac{|x-x_{0}|}{m_\lambda(B_{\mathbb{R}_+^{n+1}}(x_0,\,|x_0-y_0|))|x_{0}-y_{0}|}\\
&\leq \frac{C_1}{Mm_\lambda(B_{\mathbb{R}_+^{n+1}}(x_0,\,M\delta^k))}
\end{align*}
for some constant $C_1>0$. To continue, we choose $M>\frac{2C_1}{C_0}$ and then divide the proof into two cases.

{\bf Case 1.} If $K_{\lambda,\ell}(x_{0},\,y_{0})>0$, then
\begin{align*}
K_{\lambda,\ell}(x,\,y)
&\geq K_{\lambda,\ell}(x_{0},\,y_{0})-|K_{\lambda,\ell}(x,\,y)-K_{\lambda,\ell}(x_{0},\,y_{0})|\\
&\geq \frac{C_0}{m_\lambda(B_{\mathbb{R}_+^{n+1}}(x_0,\,M\delta^k))}-\frac{C_1}{Mm_\lambda(B_{\mathbb{R}_+^{n+1}}(x_0,\,M\delta^k))}\\
&= \frac{C_0}{2m_\lambda(B_{\mathbb{R}_+^{n+1}}(x_0,\,M\delta^k))}\\
&\geq \frac{C}{m_\lambda(Q)}.
\end{align*}

{\bf Case 2.} If $K_{\lambda,\ell}(x_{0},\,y_{0})<0$, then
\begin{align*}
K_{\lambda,\ell}(x,\,y)
&\leq K_{\lambda,\ell}(x_{0},\,y_{0})+|K_{\lambda,\ell}(x,\,y)-K_{\lambda,\ell}(x_{0},\,y_{0})|\\
&\leq -\frac{C_0}{m_\lambda(B_{\mathbb{R}_+^{n+1}}(x_0,\,M\delta^k))}+\frac{C_1}{Mm_\lambda(B_{\mathbb{R}_+^{n+1}}(x_0,\,M\delta^k))}\\
&=-\frac{C_0}{2m_\lambda(B_{\mathbb{R}_+^{n+1}}(x_0,\,M\delta^k))}\\
&\leq -\frac{C}{m_\lambda(Q)}.
\end{align*}
Applying inequality \eqref{doub} and then choosing the constants $c_i=Ma_i$, $i=1,2$, we completes the proof of Lemma \ref{sign}.
\end{proof}

\section{A summability self-increased Lemma}\label{Sec3}
\setcounter{equation}{0}
Inspired by \cite[Proposition 4.1]{RS}, we will establish a summability self-increased Lemma in this section, which states that the definition of oscillation space norm does not depend on the integrability index of mean oscillation. This property plays a technical role in establishing the upper bound for the Schatten--Lorentz norm of the Riesz transform commutator (see Proposition \ref{divide2}). To establish a summability self-increased Lemma, we first recall the following Schur's test Lemma from \cite[Lemma 3.2]{MR654483}.
\begin{lemma}\label{SchurL}
Suppose that $1<p<\infty$. Let $A=(a_{ij})$ be an infinite matrix with non-negative entries. Suppose that there is a constant $K>0$ and a non-negative sequence $\{b_i\}$ such that
$$\sum_ja_{ij}b_j^{p'}\leq Kb_i^{p'},\ {\rm for}\ {\rm any}\ j=1,2,\ldots,$$
$$\sum_ia_{ij}b_i^p\leq Kb_j^p,\ {\rm for}\ {\rm any}\ j=1,2,\ldots.$$
Then the map $T:\ (f_i)_i\rightarrow (\sum_{j}a_{ij}f_j)_i$ is bounded on $\ell^p$ with norm bounded by $K$.
\end{lemma}
To establish a summability self-increased Lemma in the Bessel setting, we need to introduce an auxiliary heat maximal function associated with Bessel operator.
For any $b\in L_{{\rm loc}}^1(\mathbb{R}_+^{n+1},dm_\lambda)$, this maximal function over Carleson box centered at $(x,\,t)\in\mathbb{R}_+^{n+1}$ is defined by $$S_\lambda(b)(x,\,t):=\sup\left\{s\left|\nabla e^{-s^2 {\Delta_\lambda}}b(y)\right|:\ y\in B_{\mathbb{R}_+^{n+1}}(x,\,t),\ \delta t\leq s\leq t\right\},$$
where $\nabla$ is the full gradient given by $\nabla=(\partial_{x_1},\ldots,\partial_{x_{n+1}},\partial_t)$.
\begin{lemma}\label{aujxi}
Let $1<p<\infty$, $1\leq q\leq \infty$ and $b\in L_{{\rm loc}}^1(\mathbb{R}_+^{n+1},dm_\lambda)$. Then for any $\nu\in\{1,2,\ldots ,\kappa\}$, we have
\begin{align*}
\left\|\left\{\sup\limits\left\{S_\lambda(b)(x,\,t):x\in Q,\ \delta\ell(Q)\leq t\leq \ell(Q)\right\}\right\}_{Q\in\mathcal{D}^\nu}\right\|_{\ell^{p,q}}\lesssim\sum_{\iota=1}^\kappa\left\|\left\{MO_Q(b)\right\}_{Q\in\mathcal{D}^\iota}\right\|_{\ell^{p,q}}.
\end{align*}
\end{lemma}
\begin{proof}
To begin with, recall from \eqref{eq:contain} that for any $Q\in\mathcal{D}_j^\nu$ centered at $c_Q\in\mathbb{R}_+^{n+1}$, there is a Bessel ball $B_Q:=B_{\mathbb{R}_{+}^{n+1}}(c_Q,\,4\delta^j)$ such that $Q\subset B_Q$. Next, we denote
 \begin{align}\label{UjB}
 U_k(B_Q):=\left\{\begin{array}{ll}\delta^{-1}B_Q, &k=0,\\ \delta^{-k-1}B_Q\setminus \delta^{-k}B_Q, &k\geq 1.\end{array}\right.
 \end{align}
  By Lemma \ref{heat kernel}, for any sufficiently large constant $N>0$ to be chosen later,
\begin{align}\label{hjlll}
&s\left|\nabla e^{-s^2{\Delta_\lambda}}b(y)\right|\nonumber\\
&=s\left|\nabla\int_{\mathbb{R}_+^{n+1}}(b(z)-(b)_{B_Q})K_{e^{-s^2{\Delta_\lambda}}}(y,\,z)dm_\lambda(z)\right|\nonumber\\
&\leq\int_{\mathbb{R}_+^{n+1}}|b(z)-(b)_{B_Q}||s\nabla K_{ e^{-s^2{\Delta_\lambda}}}(y,\,z)|dm_\lambda(z)\nonumber\\
&\lesssim\sum_{k\geq 0}\frac{1}{m_\lambda(B_{\mathbb{R}_+^{n+1}}(y,\,s))}\int_{U_k(B_Q)}|b(z)-(b)_{B_Q}| \left(\frac{s}{s+|y-z|}\right)^N dm_\lambda(z)\nonumber\\
&\lesssim \sum_{k\geq 0}\delta^{kN}\fint_{\delta^{-k-1}B_Q}|b(z)-(b)_{B_Q}|dm_\lambda(z)\nonumber\\
&\lesssim \sum_{k\geq 0}\delta^{k(N-n-1-2\lambda)}\fint_{\delta^{-k-1}B_Q}|b(z)-(b)_{\delta^{-k-1}B_Q}|dm_\lambda(z)
\end{align}
for some implicit constant $C>0$ independent of $x\in Q\in\mathcal{D}^\nu,\ \delta\ell(Q)\leq t\leq \ell(Q)$, $y\in B_{\mathbb{R}_+^{n+1}}(x,\,t)$ and $\delta t\leq s\leq t$,
where the last two inequalities used the doubling inequality \eqref{doub}. To continue, we apply Lemma \ref{thm:existence2} to see that for any $k\in\mathbb{Z}$ and $Q\in\mathcal{D}_j^\nu$, there is a $\iota\in\{1,2,\ldots ,\kappa\}$ and a cube $P_Q^k\in \mathcal{D}_{j-k-1}^\iota$ such that
\begin{equation}
    \delta^{-k-1}B_Q\subseteq P_Q^k\subseteq C_{{\rm adj}}\delta^{-k-1}B_Q.
\end{equation}
Therefore,
\begin{align}\label{rhss1}
{\rm RHS}\ {\rm of}\ \eqref{hjlll}\lesssim\sum_{k\geq 0}\delta^{k(N-n-1-2\lambda)}\fint_{P_Q^k}|b(z)-(b)_{P_Q^k}|dm_\lambda(z).
\end{align}
Note that for any $k\in\mathbb{Z}$ and $Q\in\mathcal{D}^\nu_j$, we have $Q\subset P_Q^k$. Therefore, the number of $Q'\in\mathcal{D}^\nu_j$ such that $P_Q^k=P_{Q'}^k$ is at most $c\delta^{-k(n+1+2\lambda)}$. This, in combination with inequality \eqref{rhss1}, yields
\begin{align*}
&\left\|\left\{\sup\limits\left\{S_\lambda(b)(x,\,t):x\in Q,\ \delta\ell(Q)\leq t\leq \ell(Q)\right\}\right\}_{Q\in\mathcal{D}^\nu}\right\|_{\ell^{p,q}}\\
&\lesssim \sum_{k\geq 0}\delta^{kN}\left\|\left\{\fint_{P_Q^k}|b(z)-(b)_{P_Q^k}|dm_\lambda(z)\right\}_{Q\in\mathcal{D}^\nu}\right\|_{\ell^{p,q}}\\
&\lesssim \sum_{\iota=1}^\kappa\sum_{k\geq 0}\delta^{k(N-2(n+1+2\lambda))}\left\|\left\{\fint_{Q}|b(z)-(b)_{Q}|dm_\lambda(z)\right\}_{Q\in\mathcal{D}^\iota}\right\|_{\ell^{p,q}}\\
&\lesssim \sum_{\iota=1}^\kappa\|\{MO_Q(b)\}_{Q\in\mathcal{D}^\iota}\|_{\ell^{p,q}},
\end{align*}
where the last inequality holds since $N$ can be chosen to be larger than $2(n+1+2\lambda)$. This ends the proof of Lemma \ref{aujxi}.
\end{proof}
For any $1\leq r<\infty$ and any $b\in L_{{\rm loc}}^r(\mathbb{R}_+^{n+1},dm_\lambda)$, we define its mean oscillation of power $r$ over a cube $Q$ by
\begin{align*}
MO_Q^r(b):=\left(\fint_{Q}|b(x)-(b)_Q|^rdm_\lambda(x)\right)^{1/r}.
\end{align*}
Then the summability self-increased Lemma is formulated as follows.
\begin{lemma}\label{technic}
Let  $1<p<\infty$ and $1\leq q\leq \infty$. Then the following statements are equivalent:
\begin{enumerate}
  \item For any $r\in [1,\infty)$ and any $\nu\in\{1,2,\ldots ,\kappa\}$, we have $\{MO_Q^r(b)\}_{Q\in\mathcal{D}^\nu}\in \ell^{p,q}$.
  \item For any $\nu\in\{1,2,\ldots ,\kappa\}$, we have $\{MO_Q(b)\}_{Q\in\mathcal{D}^\nu}\in \ell^{p,q}$.
\end{enumerate}

\end{lemma}
\begin{proof}
It is direct that (1) implies (2). Now we show that (2) implies (1).
To begin with, we have
\begin{align*}
\|\{MO_Q^r(b)\}_{Q\in\mathcal{D}^\nu}\|_{\ell^{p,q}}
&\lesssim \left\|\left\{\left(\fint_Q|b(x)-e^{-\ell(Q)^2 \Delta_\lambda }b(c_Q)|^rdm_\lambda(x)\right)^{1/r}\right\}_{Q\in\mathcal{D}^\nu}\right\|_{\ell^{p,q}}\\
&\lesssim \left\|\left\{\sup\limits_{x\in Q}\left|e^{-\ell(Q)^2 \Delta_\lambda}b(x)-e^{-\ell(Q)^2\Delta_\lambda}b(c_Q)\right|\right\}_{Q\in\mathcal{D}^\nu}\right\|_{\ell^{p,q}}\\
&\hspace{1.0cm}+\left\|\left\{\left(\fint_Q \left(\int_0^{\ell(Q)}\left|\frac{\partial}{\partial t}e^{-t^2 \Delta_\lambda }b(x)\right|dt\right)^rdm_\lambda(x)\right)^{1/r}\right\}_{Q\in\mathcal{D}^\nu}\right\|_{\ell^{p,q}}\\
&=:{\rm I}+{\rm II}.
\end{align*}

For the term ${\rm I}$, we apply the mean value theorem and Lemma \ref{aujxi} to conclude that
\begin{align*}
{\rm I}\lesssim\left\|\left\{\sup\limits_{x\in Q}\ell(Q)|\nabla e^{-\ell(Q)^2 \Delta_\lambda}b(x)|\right\}_{Q\in\mathcal{D}^\nu}\right\|_{\ell^{p,q}}
\lesssim \sum_{\iota=1}^\kappa\|\{MO_Q(b)\}_{Q\in\mathcal{D}^\iota}\|_{\ell^{p,q}}.
\end{align*}

For the term ${\rm II}$, we first apply Jensen's inequality to deduce that
\begin{align}\label{rhs123}
&\left(\fint_Q \left(\int_0^{\ell(Q)}\left|\frac{\partial}{\partial t}e^{-t^2 \Delta_\lambda }b(x)\right|dt\right)^rdm_\lambda(x)\right)^{1/r}\nonumber\\
&\lesssim  \left(\fint_Q \left(\sum_{j=-\infty}^{\log_{\delta^{-1}}\ell(Q)}\sup\limits_{\delta^{-j+1}\leq t\lesssim \delta^{-j}}t|\nabla e^{-t^2 \Delta_\lambda }b(x)(x,\,t)| \right)^rdm_\lambda(x)\right)^{1/r} \nonumber\\
&\lesssim  \left(\fint_Q \left(\sum_{j=-\infty}^{\log_{\delta^{-1}}\ell(Q)}S_\lambda(b)(x,\,\delta^{-j}) \right)^rdm_\lambda(x)\right)^{1/r} \nonumber\\
&=  C\left(\fint_Q \left(\sum_{j=-\infty}^{\log_{\delta^{-1}}\ell(Q)}(\log_{\delta^{-1}}\ell(Q)-j+1)^{-2}\times (\log_{\delta^{-1}}\ell(Q)-j+1)^2S_\lambda(b)(x,\,\delta^{-j}) \right)^rdm_\lambda(x)\right)^{1/r} \nonumber\\
&\lesssim  \left(\fint_Q \sum_{j=-\infty}^{\log_{\delta^{-1}}\ell(Q)}(\log_{\delta^{-1}}\ell(Q)-j+1)^{-2}\times \left( (\log_{\delta^{-1}}\ell(Q)-j+1)^2S_\lambda(b)(x,\,\delta^{-j}) \right)^rdm_\lambda(x)\right)^{1/r} \nonumber\\
&=C\left(\sum_{j=-\infty}^{\log_{\delta^{-1}}\ell(Q)}(\log_{\delta^{-1}}\ell(Q)-j+1)^{2r-2}\fint_QS_\lambda(b)(x,\,\delta^{-j})^rdm_\lambda(x) \right)^{1/r}.
\end{align}
To continue, for each $\nu\in\{1,2,\ldots ,\kappa\}$, $j\in (-\infty,\log_{\delta^{-1}}\ell(Q)]\cap\mathbb{Z}$ and $Q\in\mathcal{D}^\nu$, we let
$$\mathcal{D}_Q^\nu:=\{R\in\mathcal{D}^\nu:\ R\subseteq Q\},$$
$$\mathcal{D}_{Q,j}^\nu:=\{R\in\mathcal{D}_j^\nu:\ R\subseteq Q\}.$$
Then we have
\begin{align}\label{backo}
{\rm RHS}\ {\rm of}\ \eqref{rhs123}
&\lesssim\left(\sum_{j=-\infty}^{\log_{\delta^{-1}}\ell(Q)}\sum_{R\in\mathcal{D}_{Q,-j}^\nu}(\log_{\delta^{-1}}\ell(Q)-j+1)^{2r-2}m_\lambda(Q)^{-1}\int_{R}S_\lambda(b)(x,\,\delta^{-j})^rdm_\lambda(x) \right)^{1/r}\nonumber\\
&\lesssim\left(\sum_{R\in\mathcal{D}_{Q}^\nu}\frac{m_\lambda(R)}{m_\lambda(Q)}\left(\log_{\delta^{-1}}\frac{\ell(Q)}{\ell(R)}+1\right)^{2r-2}\sup\limits_{x\in R}S_\lambda(b)(x,\,\ell(R))^r \right)^{1/r}.
\end{align}
To continue, for any sequence $\{a_Q\}_{Q\in\mathcal{D}^\nu}$, we let
$$M_r(a)(Q)=\sum_{R\in\mathcal{D}_{Q}^\nu}\frac{m_\lambda(R)}{m_\lambda(Q)}\left(\log_{\delta^{-1}}\frac{\ell(Q)}{\ell(R)}+1\right)^{2r-2}|a_R|.$$
We {\bf claim} that $M_r$ is a bounded operator on $\ell^{p,q}$ for any $1<p<\infty$ and $1\leq q\leq \infty$. Before providing its proof, we first illustrate how it implies our desired inequality. we apply the $\ell^{p,q}$ boundedness of $M_r$ and inequality \eqref{backo}, together with Lemma \ref{aujxi}, to conclude that
\begin{align*}
{\rm II}&\leq C \left\|\left\{\left(\sum_{R\in\mathcal{D}_{Q}^\nu}\frac{m_\lambda(R)}{m_\lambda(Q)}\left(\log_{\delta^{-1}}\frac{\ell(Q)}{\ell(R)}+1\right)^{2r-2}\sup\limits_{x\in R}S_\lambda(b)(x,\,\ell(R))^r \right)^{1/r}\right\}_{Q\in\mathcal{D}^\nu}\right\|_{\ell^{p,q}}\\
&= C \left\|\left\{\sum_{R\in\mathcal{D}_{Q}^\nu}\frac{m_\lambda(R)}{m_\lambda(Q)}\left(\log_{\delta^{-1}}\frac{\ell(Q)}{\ell(R)}+1\right)^{2r-2}\sup\limits_{x\in R}S_\lambda(b)(x,\,\ell(R))^r \right\}_{Q\in\mathcal{D}^\nu}\right\|_{\ell^{p/r,q/r}}\\
&\leq C\left\|\left\{S_\lambda(b)(x,\,\ell(Q))^r \right\}_{Q\in\mathcal{D}^\nu}\right\|_{\ell^{p/r,q/r}}\\
&=C\left\|\left\{S_\lambda(b)(x,\,\ell(Q)) \right\}_{Q\in\mathcal{D}^\nu}\right\|_{\ell^{p,q}}\\
&\leq C \sum_{\iota=1}^\kappa\|\{MO_Q(b)\}_{Q\in\mathcal{D}^\iota}\|_{\ell^{p,q}}.
\end{align*}

Now we go back to the proof of the claim. By interpolation, it suffices to show that it is a bounded operator on $\ell^p$ for any $1<p<\infty$. To show this, we let
$$\mathcal{E}_R^\nu:=\{Q\in\mathcal{D}^\nu:\ Q\supseteq R\},$$
and then we will verify that for any $1<p<\infty$,
\begin{align}
&\sum_{R\in\mathcal{D}_{Q}^\nu}\frac{m_\lambda(R)}{m_\lambda(Q)}\left(\log_{\delta^{-1}}\frac{\ell(Q)}{\ell(R)}+1\right)^{2r-2}b_R^{p'}\lesssim b_Q^{p'},\label{ver10}\\
&\sum_{Q\in\mathcal{E}_{R}^\nu}\frac{m_\lambda(R)}{m_\lambda(Q)}\left(\log_{\delta^{-1}}\frac{\ell(Q)}{\ell(R)}+1\right)^{2r-2}b_Q^{p}\lesssim b_R^{p}\label{ver11},
\end{align}
where we choose $b_Q:=m_\lambda(Q)^\epsilon$ for some $\epsilon>0$ small enough. Indeed, to obtain \eqref{ver10}, by changing the order of the sum and then applying inequality \eqref{doub}, we see that
\begin{align}\label{intoa}
{\rm LHS}\ {\rm of}\ \eqref{ver10}&=m_\lambda(Q)^{\epsilon p'}\sum_{R\in\mathcal{D}_{Q}^\nu}\left(\frac{m_\lambda(R)}{m_\lambda(Q)}\right)^{1+\epsilon p'}\left(\log_{\delta^{-1}}\frac{\ell(Q)}{\ell(R)}+1\right)^{2r-2}\nonumber\\
&\leq C m_\lambda(Q)^{\epsilon p'}\sum_{R\in\mathcal{D}_{Q}^\nu}\left(\frac{m_\lambda(R)}{m_\lambda(Q)}\right)^{1+\epsilon p'/2}\nonumber\\
&= C m_\lambda(Q)^{\epsilon p'}\sum_{j:\delta^j\leq \ell(Q)}\sum_{R\in\mathcal{D}_{Q,j}^\nu}\left(\frac{m_\lambda(R)}{m_\lambda(Q)}\right)^{1+\epsilon p'/2},
\end{align}
where the constant $C$ depends only on $r$ and $p$. Since we are in the Bessel setting, in which the upper dimension is not equal to the lower one, one need to deal with the volumn term more delicately. To this end, we apply inequality \eqref{doub} to see that
\begin{align*}
\sum_{j:\delta^j\leq \ell(Q)}\sum_{R\in\mathcal{D}_{Q,j}^\nu}\left(\frac{m_\lambda(R)}{m_\lambda(Q)}\right)^{1+\epsilon p'/2}
&\leq C \sum_{j:\delta^j\leq \ell(Q)}\left(\frac{\delta^j}{\ell(Q)}\right)^{(n+1)\epsilon p'/2}\sum_{R\in\mathcal{D}_{Q,j}^\nu}\left(\frac{m_\lambda(R)}{m_\lambda(Q)}\right)\\
&=C\sum_{j:\delta^j\leq \ell(Q)}\left(\frac{\delta^j}{\ell(Q)}\right)^{(n+1)\epsilon p'/2}\left(\frac{m_\lambda(\bigcup_{R\in\mathcal{D}_{Q,j}^\nu}R)}{m_\lambda(Q)}\right)\\
&\leq C.
\end{align*}
Substituting the above inequality into \eqref{intoa}, we deduce that the right-hand side of \eqref{intoa} is dominated by $m_\lambda(Q)^{\epsilon p'}$. This verifies \eqref{ver10}. Next, we verify \eqref{ver11}. By inequality \eqref{doub},
\begin{align*}
{\rm LHS}\ {\rm of}\ \eqref{ver11}
&\lesssim m_\lambda(R)^{\epsilon p}\sum_{j:\delta^j\geq \ell(R)}\left(\frac{\ell(R)}{\delta^j}\right)^{(n+1)(1-\epsilon p)}\left(\log_{\delta^{-1}}\frac{\delta^j}{\ell(R)}+1\right)^{2r-2}\\
&\lesssim m_\lambda(R)^{\epsilon p},
\end{align*}
where in the first inequality we used the fact again that for any $R\in\mathcal{D}^\nu$ and $j\in\mathbb{Z}$ satisfying $\delta^j\geq \ell(R)$, there exists a unique $Q\in\mathcal{D}_j^\nu$ such that $Q\supseteq R$. Combining \eqref{ver10}, \eqref{ver11} with Lemma \ref{SchurL}, we see that $M_r$ is a bounded operator on $\ell^p$.
This completes the proof of Lemma \ref{technic}.
\end{proof}

\section{Proof of Theorem \ref{mainLor}:\ upper bound}\label{Sec4}
\setcounter{equation}{0}
This section is devoted to establishing the upper bound for the Schatten--Lorentz norm of the Riesz transform commutator. To begin with, we set $\Xi_+:=\{(x,\,y)\in\mathbb{R}_+^{n+1}\times \mathbb{R}_+^{n+1}:\ x=y\}$, then we have the following modified Whitney decomposition Lemma.
\begin{lemma}\label{Whitney}
There exists a family of closed standard dyadic cubes $\mathcal{P}=\{P_j\}_j$ on $(\mathbb{R}_+^{n+1}\times \mathbb{R}_+^{n+1})\backslash \Xi_+$ such that
\begin{enumerate}
  \item $\cup_jP_j=(\mathbb{R}_+^{n+1}\times \mathbb{R}_+^{n+1})\backslash \Xi_+$ {\rm(disjoint\ union)};
  \item $\sqrt{n}\ell(P_j)\leq {\rm dist}(P_j,\,\Xi_+)\leq 4\sqrt{n}\ell(P_j)$;
  \item If the boundaries of two cubes $P_j$ and $P_k$ touch, then $$\frac{1}{4}\leq \frac{\ell(P_j)}{\ell(P_k)}\leq 4;$$
  \item For a given $P_j$, there are at most $12^n-4^n$ cubes $P_k$ that touch it;
  \item Let $0<\epsilon<1/4$ and $P_j^*$ be the cube with the same center as $P_j$ and with the sidelength $(1+\epsilon)\ell(P_j)$. Then each point of $(\mathbb{R}_+^{n+1}\times \mathbb{R}_+^{n+1})\backslash \Xi_+$ is contained in at most $c_n$ of the cubes $P_j^*$.
      \end{enumerate}
\begin{proof}
Since the boundary of $(\mathbb{R}_+^{n+1}\times \mathbb{R}_+^{n+1})\backslash \Xi_+$ is $\Xi_+\cup \partial (\mathbb{R}_+^{n+1}\times \mathbb{R}_+^{n+1})$, one may not apply the standard dyadic Whitney decomposition on $(\mathbb{R}_+^{n+1}\times \mathbb{R}_+^{n+1})\backslash \Xi_+$ directly to get the desired result. Instead, we apply the standard dyadic Whitney decomposition (see e.g. \cite[Appendix J.1]{MR3243734}) on the region $(\mathbb{R}^{n+1}\times \mathbb{R}^{n+1})\backslash \Xi$, where $\Xi=\{(x,\,y)\in\mathbb{R}^{n+1}\times \mathbb{R}^{n+1}:\ x=y\}$,  to see that there exists a family of closed standard dyadic cubes $\mathcal{P}'=\{P_j'\}_j$ on $(\mathbb{R}^{n+1}\times \mathbb{R}^{n+1})\backslash \Xi$ such that (1)--(5) hold with $P_j$ and $(\mathbb{R}_+^{n+1}\times \mathbb{R}_+^{n+1})\backslash \Xi_+$ replaced by $P_j'$ and $(\mathbb{R}^{n+1}\times \mathbb{R}^{n+1})\backslash \Xi$, respectively. Since each $P_j'$ is contained in or disjoint with $(\mathbb{R}_+^{n+1}\times \mathbb{R}_+^{n+1})\backslash \Xi_+$, this allows us to choose $\mathcal{P}=\{P_j\}_j$ be a sub-family of $\mathcal{P}'$ such that $$\cup_jP_j= (\mathbb{R}_+^{n+1}\times \mathbb{R}_+^{n+1})\backslash \Xi_+.$$
For these $\{P_j\}_j$, we see that (2) holds since
$${\rm dist}(P_j,\, \Xi_+)\approx {\rm dist}(P_j,\,\Xi).$$
Moreover,  (3)--(5) hold directly due to the construction of $\{P_j'\}_j$. This ends the proof of Lemma \ref{Whitney}.
\end{proof}

\end{lemma}
Note that each $P\in\mathcal{P}$ can be written as $P=Q_P\times \hat{Q}_P$ for some $Q_P$ and $\hat{Q}_P$ in $\mathcal{D}^0$. This, combination with (2) in Lemma \ref{Whitney}, yields $${\rm dist}(Q_P,\,\hat{Q}_P)\approx \ell(P).$$
Now for any $P\in\mathcal{P}$, define
$$U_{P}(x,\,y):=K_{\lambda,\ell}(x,\,y)\chi_{Q_P}(x)\chi_{\hat{Q}_P}(y).$$
Here we omit the dependence on $\lambda$ and $\ell$ since these two parameters play no role in the proof below.

The following Lemma states that Bessel--Riesz transform kernel can be decomposed  locally into a summation of NWO sequences with a separate variable form and a suitable decay factor.
\begin{lemma}\label{jqss}
There is a finite index set $\mathcal{I}$ such that
for any $N\in\mathbb{Z}_+$, $\ell\in\{1,2,...,n+1\}$ and $P\in\mathcal{P}$, one can construct  functions $F_{P,k,v}$ , $G_{P,k,v}$ and a sequence $\{B_{P,k,v}\}$  of real numbers  satisfying for any $P\in\mathcal{P}$, $k\in\mathbb{Z}^4$ and $v\in\mathcal{I}$,
\begin{enumerate}
  \item ${\rm supp}(F_{P,k,v})\subset Q_P$, ${\rm supp}(G_{P,k,v})\subset \hat{Q}_P$,
  \item $\|F_{P,k,v}\|_{\infty}\leq m_\lambda(Q_P)^{-1/2}$, $\|G_{P,k,v}\|_{\infty}\leq m_\lambda(Q_P)^{-1/2}$,
  \item $|B_{P,k,v}|\leq C_N(1+|k|)^{-N}$ for some $C_N>0$ independent of  $P$ and $k$,
\item $U_{P}$ admits the following factorization
\begin{align}\label{reform}
U_{P}(x,\,y)=\sum_{v\in\mathcal{I}}\sum_{k\in\mathbb{Z}^4}B_{P,k,v}F_{P,k,v}(x)G_{P,k,v}(y).
\end{align}
\end{enumerate}
\end{lemma}
\begin{proof}

To begin with, we let $\{h_{I,\mathcal{K}}^\epsilon\}_{I\in\mathcal{D}^0,\epsilon\in\Gamma_{n,\mathcal{K}}}$ be the Alpert wavelet given in Lemma \ref{AW} with $\mathcal{D}:=\mathcal{D}^0$ and $\mathcal{K}$ being a sufficiently large integer. Then we define
\begin{align}
&D_{k}^{(1)}(f)(x,\,y)=\sum_{I\in\mathcal{D}^0_{k-1}}\sum_{\epsilon\in\Gamma_{n,\mathcal{K}}}\langle f(\cdot,\,y), h_{I,\mathcal{K}}^\epsilon\rangle h_{I,\mathcal{K}}^\epsilon(x),\label{dkdef}\\
&D_{k}^{(2)}(f)(x,\,y)=\sum_{I\in\mathcal{D}^0_{k-1}}\sum_{\epsilon\in\Gamma_{n,\mathcal{K}}}\langle f(x,\,\cdot), h_{I,\mathcal{K}}^\epsilon\rangle h_{I,\mathcal{K}}^\epsilon(y).\label{dkdef2}
\end{align}
In what follows, we omit the indexs $\mathcal{K}$ in $h_{I,\mathcal{K}}^\epsilon$ and $\Gamma_{n,\mathcal{K}}$ for simplicity.  By Lemma \ref{AW}, $U_P$ can be decomposed as follows.
\begin{align}
&U_{P}(x,\,y)\nonumber\\
&=\left(\sum_{k_1=k_P+1}^{+\infty}D_{k_1}^{(1)}+\sum_{k_1=-\infty}^{k_P}D_{k_1}^{(1)}\right)\left(\sum_{k_2=k_P+1}^{+\infty}D_{k_2}^{(2)}+\sum_{k_2=-\infty}^{k_P}D_{k_2}^{(2)}\right)(U_P)(x,\,y)\chi_{Q_P}(x)\chi_{\hat{Q}_P}(y)\nonumber\\
&=\left(\sum_{k_1=k_P+1}^{+\infty}\sum_{k_2=k_P+1}^{+\infty}+\sum_{k_1=-\infty}^{k_P}\sum_{k_2=k_P+1}^{+\infty}+\sum_{k_1=k_P+1}^{+\infty}\sum_{k_2=-\infty}^{k_P}+\sum_{k_1=-\infty}^{k_P}\sum_{k_2=-\infty}^{k_P}\right)D_{k_1}^{(1)}D_{k_2}^{(2)}(U_P)(x,\,y)\chi_{Q_P}(x)\chi_{\hat{Q}_P}(y)\nonumber\\
&=\left(\sum_{k_1=k_P+1}^{+\infty}\sum_{k_2=k_P+1}^{+\infty}+\sum_{k_1=-\infty}^{k_P}\sum_{k_2=k_P+1}^{+\infty}+\sum_{k_1=k_P+1}^{+\infty}\sum_{k_2=-\infty}^{k_P}+\sum_{k_1=-\infty}^{k_P}\sum_{k_2=-\infty}^{k_P}\right)\sum_{I_1\in\mathcal{D}^0_{k_1-1}}\sum_{I_2\in\mathcal{D}^0_{k_2-1}}\sum_{\epsilon_1\in\Gamma_n}\sum_{\epsilon_2\in\Gamma_n}\nonumber\\ \nonumber\\
&\hspace{3.7cm}\int_{I_1\cap Q_P}\int_{I_2\cap \hat{Q}_P}K_{\lambda,\ell}(z,\,w)h_{I_1}^{\epsilon_1}(z)h_{I_2}^{\epsilon_2}(w)dm_\lambda(w)dm_\lambda(z)h_{I_1}^{\epsilon_1}(x)h_{I_2}^{\epsilon_2}(y)\nonumber\\
&=:\sum_{u=1}^4{\rm I}_u,
\end{align}
where we choose $k_P$ to be the integer such that $2^{-k_P}\leq \ell(P)<2^{-k_P+1}$. Next we deal with these four terms separately.

For the term ${\rm I}_1$,  we note that for any integer $k_1\geq k_P+1 $, there are $2^{(k_1-k_P)(n+1)}$ cubes $I\in\mathcal{D}^0_{k_1-1}$ such that $I\subset Q_P$. We denote these cubes by $I_{1,k_1}^{1},\ldots,I_{1,k_1}^{2^{(k_1-k_P)(n+1)}}$. Moreover, for any integer $k_2\geq k_P+1 $, there are $2^{(k_2-k_P)(n+1)}$ cubes  $I\in\mathcal{D}^0_{k_2-1}$ such that $I\subset\hat{Q}_P$. We denote these cubes by $I_{2,k_2}^{1},\ldots,I_{2,k_2}^{2^{(k_2-k_P)(n+1)}}$. Then
\begin{align*}
{\rm I}_1
&=\sum_{k_1=k_P+1}^{+\infty}\sum_{k_2=k_P+1}^{+\infty}\sum_{j_1=1}^{2^{(k_1-k_P)(n+1)}}\sum_{j_2=1}^{2^{(k_2-k_P)(n+1)}}\sum_{\epsilon_1\in\Gamma_n}\sum_{\epsilon_2\in\Gamma_n}\\
&\hspace{1.0cm}\int_{I_{1,k_1}^{j_1}}\int_{I_{2,k_2}^{j_2}}K_{\lambda,\ell}(z,\,w)h_{I_{1,k_1}^{j_1}}^{\epsilon_1}(z)h_{I_{2,k_2}^{j_2}}^{\epsilon_2}(w)dm_\lambda(w)dm_\lambda(z)h_{I_{1,k_1}^{j_1}}^{\epsilon_1}(x)h_{I_{2,k_2}^{j_2}}^{\epsilon_2}(y).
\end{align*}
For any integer $k_1\geq k_P+1 $ and any integer $k_2\geq k_P+1$, we let $$F_{P,k_P-k_1,k_P-k_2}^{\epsilon_1,\epsilon_2,j_1,j_2}(x):=\left(\frac{m_\lambda(I_{1,k_1}^{j_1})}{m_\lambda(Q_P)}\right)^{1/2}h_{I_{1,k_1}^{j_1}}^{\epsilon_1}(x),\ \ G_{P,k_P-k_1,k_P-k_2}^{\epsilon_1,\epsilon_2,j_1,j_2}(y):=\left(\frac{m_\lambda(I_{2,k_2}^{j_2})}{m_\lambda(Q_P)}\right)^{1/2}h_{I_{2,k_2}^{j_2}}^{\epsilon_2}(y)$$ and
\begin{align*}
B_{P,k_P-k_1,k_P-k_2}^{\epsilon_1,\epsilon_2,j_1,j_2}:&=\left(\frac{m_\lambda(Q_P)}{m_{\lambda}(I_{1,k_1}^{j_1})}\right)^{1/2}\left(\frac{m_\lambda(Q_P)}{m_\lambda(I_{2,k_2}^{j_2})}\right)^{1/2}\int_{I_{1,k_1}^{j_1}}\int_{I_{2,k_2}^{j_2}}K_{\lambda,\ell}(z,\,w)h_{I_{1,k_1}^{j_1}}^{\epsilon_1}(z)h_{I_{2,k_2}^{j_2}}^{\epsilon_2}(w)dm_\lambda(w)dm_\lambda(z).
\end{align*}
Then, it is direct to see that $\supp F_{P,k_P-k_1,k_P-k_2}^{\epsilon_1,\epsilon_2,j_1,j_2}\subset Q_P$, $\supp G_{P,k_P-k_1,k_P-k_2}^{\epsilon_1,\epsilon_2,j_1,j_2}\subset \hat{Q}_P$ and $\|F_{P,k_P-k_1,k_P-k_2}^{\epsilon_1,\epsilon_2,j_1,j_2}\|_{\infty}+\|G_{P,k_P-k_1,k_P-k_2}^{\epsilon_1,\epsilon_2,j_1,j_2}\|_{\infty}\leq m_\lambda(Q_P)^{-1/2}$. Moreover, we apply Taylor's formula twice to deduce that there are points $\xi_{z,j_1,k_1}$ and $\eta_{w,k_2,j_2}$ lying in the line segments jointing $z$ with $c_{I_{1,k_1}^{j_1}}$ and jointing $w$ with $c_{I_{2,k_2}^{j_2}}$, respectively, such that
\begin{align*}
&K_{\lambda,\ell}(z,\,w)\\
&=\sum_{|\beta|\leq \mathcal{K}-1}\frac{\partial_z^{\beta}K_{\lambda,\ell}(c_{I_{1,k_1}^{j_1}},\,w)}{\beta!}(z-c_{I_{1,k_1}^{j_1}})^{\beta}+\sum_{|\beta|=\mathcal{K}}\sum_{|\gamma|\leq \mathcal{K}-1}\frac{\partial_z^\beta\partial_w^{\gamma}K_{\lambda,\ell}(\xi_{z,j_1,k_1},\,c_{I_{2,k_2}^{j_2}})}{\beta!\gamma!}(z-c_{I_{1,k_1}^{j_1}})^{\beta}(w-c_{I_{2,k_2}^{j_2}})^\gamma\\
&\hspace{1.0cm}+\sum_{|\beta|=\mathcal{K} }\sum_{|\gamma|=\mathcal{K} }\frac{\partial_z^\beta \partial_w^\gamma K_{\lambda,\ell}(\xi_{z,j_1,k_1},\,\eta_{w,j_2,k_2})}{\beta!\gamma!}(z-c_{I_{1,k_1}^{j_1}})^{\beta}(w-c_{I_{2,k_2}^{j_2}})^{\gamma}\\
&=:\mathcal{P}_{\lambda,\ell,k_1,k_2,j_1,j_2}^{\beta,\gamma,\mathcal{K}}(z,\,w)+\sum_{|\beta|=\mathcal{K} }\sum_{|\gamma|=\mathcal{K} }\frac{\partial_z^\beta \partial_w^\gamma K_{\lambda,\ell}(\xi_{z,j_1,k_1},\,\eta_{w,j_2,k_2})}{\beta!\gamma!}(z-c_{I_{1,k_1}^{j_1}})^{\beta}(w-c_{I_{2,k_2}^{j_2}})^{\gamma}.
\end{align*}
Next, we apply the cancellation property of Alpert wavelet given in Lemma \ref{AW}  and Lemma \ref{CZO} to deduce that
\begin{align*}
&\bigg|\int_{I_{1,k_1}^{j_1}}\int_{I_{2,k_2}^{j_2}}K_{\lambda,\ell}(z,\,w)h_{I_{1,k_1}^{j_1}}^{\epsilon_1}(z)h_{I_{2,k_2}^{j_2}}^{\epsilon_2}(w)dm_\lambda(w)dm_\lambda(z)\bigg|\\
&=\bigg|\int_{I_{1,k_1}^{j_1}}\int_{I_{2,k_2}^{j_2}}\big(K_{\lambda,\ell}(z,\,w)-\mathcal{P}_{\lambda,\ell,k_1,k_2,j_1,j_2}^{\beta,\gamma,\mathcal{K}}(z,\,w)\big) h_{I_{1,k_1}^{j_1}}^{\epsilon_1}(z)h_{I_{2,k_2}^{j_2}}^{\epsilon_2}(w)dm_\lambda(w)dm_\lambda(z)\bigg|\\
&=\bigg|\int_{I_{1,k_1}^{j_1}}\int_{I_{2,k_2}^{j_2}} \sum_{|\beta|=\mathcal{K} }\sum_{|\gamma|=\mathcal{K} }\frac{\partial_z^\beta \partial_w^\gamma K_{\lambda,\ell}(\xi_{z,j_1,k_1},\,\eta_{w,j_2,k_2})}{\beta!\gamma!}(z-c_{I_{1,k_1}^{j_1}})^{\beta}(w-c_{I_{2,k_2}^{j_2}})^{\gamma} h_{I_{1,k_1}^{j_1}}^{\epsilon_1}(z)h_{I_{2,k_2}^{j_2}}^{\epsilon_2}(w)dm_\lambda(w)dm_\lambda(z)\bigg|\\
&\lesssim 2^{-(k_1-k_P)\mathcal{K}}2^{-(k_2-k_P)\mathcal{K}}\left(\frac{m_\lambda(I_{1,k_1}^{j_1})}{m_\lambda(Q_P)}\right)^{1/2}\left(\frac{m_\lambda(I_{2,k_2}^{j_2})}{m_\lambda(Q_P)}\right)^{1/2}.
\end{align*}
That is,
\begin{align*}
\Big|B_{P,k_P-k_1,k_P-k_2}^{\epsilon_1,\epsilon_2,j_1,j_2}\Big|\lesssim 2^{-(k_1-k_P)\mathcal{K}}2^{-(k_2-k_P)\mathcal{K}}.
\end{align*}
Furthermore, a change of variable yields that $${\rm I}_1=\sum_{k_1\leq -1}\sum_{k_2\leq -1}\sum_{j_1=1}^{2^{-k_1(n+1)}}\sum_{j_2=1}^{2^{-k_2(n+1)}}\sum_{\epsilon_1\in\Gamma_n}\sum_{\epsilon_2\in\Gamma_n}B_{P,k_1,k_2}^{\epsilon_1,\epsilon_2,j_1,j_2}F_{P,k_1,k_2}^{\epsilon_1,\epsilon_2,j_1,j_2}(x)G_{P,k_1,k_2}^{\epsilon_1,\epsilon_2,j_1,j_2}(y).$$

For the term ${\rm I}_2$, we note that for any integer $k\leq k_P$, there is a unique $I\in\mathcal{D}^0_{k-1}$ such that $I\cap Q_P\neq\emptyset$. We denote this cube by $P_{k-1}$. Then
\begin{align*}
{\rm I}_2
&=\sum_{k_1=-\infty}^{k_P}\sum_{k_2=k_P+1}^{+\infty}\sum_{j_2=1}^{2^{(k_2-k_P)(n+1)}}\sum_{\epsilon_1\in\Gamma_n}\sum_{\epsilon_2\in\Gamma_n}\\
&\hspace{1.0cm}\int_{Q_P}\int_{I_{2,k_2}^{j_2}}K_{\lambda,\ell}(z,\,w)h_{P_{k_1-1}}^{\epsilon_1}(z)h_{I_{2,k_2}^{j_2}}^{\epsilon_2}(w)dm_\lambda(w)dm_\lambda(z)h_{P_{k_1-1}}^{\epsilon_1}(x)\chi_{Q_P}(x)h_{I_{2,k_2}^{j_2}}^{\epsilon_2}(y).
\end{align*}
For any integer $k_1\leq k_{P} $ and any integer $k_2\geq k_P+1$, we let $$F_{P,k_P-k_1,k_P-k_2}^{\epsilon_1,\epsilon_2,0,j_2}(x):=\left(\frac{m_{\lambda}(P_{k_1-1})}{m_\lambda(Q_P)}\right)^{1/2}h_{P_{k_1-1}}^{\epsilon_1}(x)\chi_{Q_P}(x),\ \  G_{P,k_P-k_1,k_P-k_2}^{\epsilon_1,\epsilon_2,0,j_2}(y):=\left(\frac{m_\lambda(I_{2,k_2}^{j_2})}{m_\lambda(Q_P)}\right)^{1/2}h_{I_{2,k_2}^{j_2}}^{\epsilon_2}(y)$$ and $$B_{P,k_P-k_1,k_P-k_2}^{\epsilon_1,\epsilon_2,0,j_2}:=\left(\frac{m_\lambda(Q_P)}{m_{\lambda}(P_{k_1-1})}\right)^{1/2}\left(\frac{m_\lambda(Q_P)}{m_\lambda(I_{2,k_2}^{j_2})}\right)^{1/2}\int_{Q_P}\int_{I_{2,k_2}^{j_2}}K_{\lambda,\ell}(z,\,w)h_{P_{k_1-1}}^{\epsilon_1}(z)h_{I_{2,k_2}^{j_2}}^{\epsilon_2}(w)dm_\lambda(w)dm_\lambda(z).$$
Then, it is direct to see that $\supp F_{P,k_P-k_1,k_P-k_2}^{\epsilon_1,\epsilon_2,0,j_2}\subset Q_P$, $\supp G_{P,k_P-k_1,k_P-k_2}^{\epsilon_1,\epsilon_2,0,j_2}\subset \hat{Q}_P$ and $\|F_{P,k_P-k_1,k_P-k_2}^{\epsilon_1,\epsilon_2,0,j_2}\|_{\infty}+\|G_{P,k_P-k_1,k_P-k_2}^{\epsilon_1,\epsilon_2,0,j_2}\|_{\infty}\leq m_\lambda(Q_P)^{-1/2}$. Moreover, we apply Taylor's formula with respect to the $w$-variable at $c_{I_{2,k_2}^{j_2}}$ to deduce that there is a point $\eta_{w,k_2,j_2}$ lying in the line segment jointing $w$ with $c_{I_{2,k_2}^{j_2}}$ such that
$$K_{\lambda,\ell}(z,\,w)=\sum_{|\gamma|\leq \mathcal{K}-1}\frac{\partial_w^\gamma K_{\lambda,\ell}(z,\,c_{I_{2,k_2}^{j_2}})}{\gamma!}(w-c_{I_{2,k_2}^{j_2}})^\gamma+\sum_{|\gamma|= \mathcal{K}}\frac{\partial_w^\gamma K_{\lambda,\ell}(z,\,\eta_{w,k_2,j_2})}{\gamma!}(w-c_{I_{2,k_2}^{j_2}})^\gamma.$$
Then we  apply the cancellation property of Alpert wavelet given in Lemma \ref{AW} and Lemma \ref{CZO} to deduce that
\begin{align*}
&\left|\int_{Q_P}\int_{I_{2,k_2}^{j_2}}K_{\lambda,\ell}(z,\,w)h_{P_{k_1-1}}^{\epsilon_1}(z)h_{I_{2,k_2}^{j_2}}^{\epsilon_2}(w)dm_\lambda(w)dm_\lambda(z)\right|\\
&=\left|\int_{Q_P}\int_{I_{2,k_2}^{j_2}}\left(K_{\lambda,\ell}(z,\,w)-\sum_{|\gamma|\leq \mathcal{K}-1}\frac{\partial_w^\gamma K_{\lambda,\ell}(z,\,c_{I_{2,k_2}^{j_2}})}{\gamma!}(w-c_{I_{2,k_2}^{j_2}})^\gamma\right)h_{P_{k_1-1}}^{\epsilon_1}(z)h_{I_{2,k_2}^{j_2}}^{\epsilon_2}(w)dm_\lambda(w)dm_\lambda(z)\right|\\
&=\left|\int_{Q_P}\int_{I_{2,k_2}^{j_2}}\sum_{|\gamma|= \mathcal{K}}\frac{\partial_w^\gamma K_{\lambda,\ell}(z,\,\eta_{w,k_2,j_2})}{\gamma!}(w-c_{I_{2,k_2}^{j_2}})^\gamma h_{P_{k_1-1}}^{\epsilon_1}(z)h_{I_{2,k_2}^{j_2}}^{\epsilon_2}(w)dm_\lambda(w)dm_\lambda(z)\right|\\
&\lesssim 2^{-(k_2-k_P)\mathcal{K}}\left(\frac{m_\lambda(I_{2,k_2}^{j_2})}{m_\lambda(P_{k_1-1})}\right)^{1/2}.
\end{align*}
This, in combination with inequality \eqref{doub}, yields
$$
\Big|B_{P,k_P-k_1,k_P-k_2}^{\epsilon_1,\epsilon_2,0,j_2}\Big|\lesssim 2^{-(k_2-k_P)\mathcal{K}}\frac{m_\lambda(Q_P)}{m_{\lambda}(P_{k_1-1})}\lesssim 2^{-(k_P-k_1)(n+1)} 2^{-(k_2-k_P)\mathcal{K}}.
$$
Furthermore, a change of variable yields that $${\rm I}_2=\sum_{k_1\geq 0}\sum_{k_2\leq -1}\sum_{j_2=1}^{2^{-k_2(n+1)}}\sum_{\epsilon_1\in\Gamma_n}\sum_{\epsilon_2\in\Gamma_n}B_{P,k_1,k_2}^{\epsilon_1,\epsilon_2,0,j_2} F_{P,k_1,k_2}^{\epsilon_1,\epsilon_2,0,j_2}(x)G_{P,k_1,k_2}^{\epsilon_1,\epsilon_2,0,j_2}(y).$$

For the term ${\rm I}_3$,  we note that for any integer $k\leq k_P$, there is a unique $I\in\mathcal{D}^0_{k-1}$ such that $I\cap \hat{Q}_P\neq\emptyset$. We denote this cube by $\hat{P}_{k-1}$. Then
\begin{align*}
{\rm I}_3
&=\sum_{k_1=k_P+1}^{+\infty}\sum_{k_2=-\infty}^{k_P}\sum_{j_1=1}^{2^{(k_1-k_P)(n+1)}}\sum_{\epsilon_1\in\Gamma_n}\sum_{\epsilon_2\in\Gamma_n}\\
&\hspace{1.0cm}\int_{I_{1,k_1}^{j_1}}\int_{\hat{Q}_P}K_{\lambda,\ell}(z,\,w)h_{I_{1,k_1}^{j_1}}^{\epsilon_1}(z)h_{\hat{P}_{k_2-1}}^{\epsilon_2}(w)dm_\lambda(w)dm_\lambda(z)h_{I_{1,k_1}^{j_1}}^{\epsilon_1}(x)h_{\hat{P}_{k_2-1}}^{\epsilon_2}(y)\chi_{\hat{Q}_P}(y).
\end{align*}
For any integer $k_1\geq k_{P}+1 $ and any integer $k_2\leq k_P$, we let $$F_{P,k_P-k_1,k_P-k_2}^{\epsilon_1,\epsilon_2,j_1,0}(x):=\left(\frac{m_\lambda(I_{1,k_1}^{j_1})}{m_\lambda(Q_P)}\right)^{1/2}h_{I_{1,k_1}^{j_1}}^{\epsilon_1}(x),\ \  G_{P,k_P-k_1,k_P-k_2}^{\epsilon_1,\epsilon_2,j_1,0}(y):=\left(\frac{m_{\lambda}(\hat{P}_{k_2-1})}{m_\lambda(Q_P)}\right)^{1/2}h_{\hat{P}_{k_2-1}}^{\epsilon_2}(y)\chi_{\hat{Q}_P}(y)$$ and $$B_{P,k_P-k_1,k_P-k_2}^{\epsilon_1,\epsilon_2,j_1,0}:=\left(\frac{m_\lambda(Q_P)}{m_{\lambda}(\hat{P}_{k_2-1})}\right)^{1/2}\left(\frac{m_\lambda(Q_P)}{m_\lambda(I_{1,k_1}^{j_1})}\right)^{1/2}\int_{I_{1,k_1}^{j_1}}\int_{\hat{Q}_P}K_{\lambda,\ell}(z,\,w)h_{I_{1,k_1}^{j_1}}^{\epsilon_1}(z)h_{\hat{P}_{k_2-1}}^{\epsilon_2}(w)dm_\lambda(w)dm_\lambda(z).$$
Then, it is direct to see that $\supp F_{P,k_P-k_1,k_P-k_2}^{\epsilon_1,\epsilon_2,j_1,0}\subset Q_P$, $\supp G_{P,k_P-k_1,k_P-k_2}^{\epsilon_1,\epsilon_2,j_1,0}\subset \hat{Q}_P$ and $\|F_{P,k_P-k_1,k_P-k_2}^{\epsilon_1,\epsilon_2,j_1,0}\|_{\infty}+\|G_{P,k_P-k_1,k_P-k_2}^{\epsilon_1,\epsilon_2,j_1,0}\|_{\infty}\leq m_\lambda(Q_P)^{-1/2}$. Moreover,  we apply Taylor's formula with respect to the $z$-variable at $c_{I_{1,k_1}^{j_1}}$ to deduce that there is a point $\xi_{z,k_1,j_1}$ lying in the line segment jointing $z$ with $c_{I_{1,k_1}^{j_1}}$ such that
$$K_{\lambda,\ell}(z,\,w)=\sum_{|\beta|\leq \mathcal{K}-1}\frac{\partial_z^\beta K_{\lambda,\ell}(c_{I_{1,k_1}^{j_1}},\,w)}{\beta!}(z-c_{I_{1,k_1}^{j_1}})^\beta+\sum_{|\beta|= \mathcal{K}}\frac{\partial_z^\beta K_{\lambda,\ell}(\xi_{z,k_1,j_1},\,w)}{\gamma!}(z-c_{I_{1,k_1}^{j_1}})^\beta.$$
Then we  apply the cancellation property of Alpert wavelet given in Lemma \ref{AW} and Lemma \ref{CZO} to deduce that
\begin{align*}
&\left|\int_{I_{1,k_1}^{j_1}}\int_{\hat{Q}_P}K_{\lambda,\ell}(z,\,w)h_{I_{1,k_1}^{j_1}}^{\epsilon_1}(z)h_{\hat{P}_{k_2-1}}^{\epsilon_2}(w)dm_\lambda(w)dm_\lambda(z)\right|\\
&=\left|\int_{I_{1,k_1}^{j_1}}\int_{\hat{Q}_P}\left(K_{\lambda,\ell}(z,\,w)-\sum_{|\beta|\leq \mathcal{K}-1}\frac{\partial_z^\beta K_{\lambda,\ell}(c_{I_{1,k_1}^{j_1}},\,w)}{\beta!}(z-c_{I_{1,k_1}^{j_1}})^\beta\right)h_{I_{1,k_1}^{j_1}}^{\epsilon_1}(z)h_{\hat{P}_{k_2-1}}^{\epsilon_2}(w)dm_\lambda(w)dm_\lambda(z)\right|\\
&=\left|\int_{I_{1,k_1}^{j_1}}\int_{\hat{Q}_P}\sum_{|\beta|= \mathcal{K}}\frac{\partial_z^\beta K_{\lambda,\ell}(\xi_{z,k_1,j_1},\,w)}{\gamma!}(z-c_{I_{1,k_1}^{j_1}})^\beta h_{I_{1,k_1}^{j_1}}^{\epsilon_1}(z)h_{\hat{P}_{k_2-1}}^{\epsilon_2}(w)dm_\lambda(w)dm_\lambda(z)\right|\\
&\lesssim 2^{-(k_1-k_P)\mathcal{K}}\left(\frac{m_\lambda(I_{1,k_1}^{j_1})}{m_\lambda(\hat{P}_{k_2-1})}\right)^{1/2}.
\end{align*}
This, in combination with inequality \eqref{doub}, yields
$$
\Big|B_{P,k_P-k_1,k_P-k_2}^{\epsilon_1,\epsilon_2,j_1,0}\Big|\lesssim 2^{-(k_1-k_P)\mathcal{K}}\frac{m_\lambda(Q_P)}{m_{\lambda}(\hat{P}_{k_2-1})}\lesssim 2^{-(k_1-k_P)\mathcal{K}}2^{-(k_P-k_2)(n+1)}.
$$
Furthermore, a change of variable yields that $${\rm I}_3=\sum_{k_1\leq -1}\sum_{k_2\geq 0}\sum_{j_1=1}^{2^{-k_1(n+1)}}\sum_{\epsilon_1\in\Gamma_n}\sum_{\epsilon_2\in\Gamma_n}B_{P,k_1,k_2}^{\epsilon_1,\epsilon_2,j_1,0} F_{P,k_1,k_2}^{\epsilon_1,\epsilon_2,j_1,0}(x)G_{P,k_1,k_2}^{\epsilon_1,\epsilon_2,j_1,0}(y).$$

For the term ${\rm I}_4$,  we see that
\begin{align*}
{\rm I}_4
&=\sum_{k_1=-\infty}^{k_P}\sum_{k_2=-\infty}^{k_P}\sum_{\epsilon_1\in\Gamma_n}\sum_{\epsilon_2\in\Gamma_n}\\
&\hspace{1.0cm}\int_{Q_P}\int_{\hat{Q}_P}K_{\lambda,\ell}(z,\,w)h_{P_{k_1-1}}^{\epsilon_1}(z)h_{\hat{P}_{k_2-1}}^{\epsilon_2}(w)dm_\lambda(w)dm_\lambda(z)h_{P_{k_1-1}}^{\epsilon_1}(x)\chi_{Q_P}(x)h_{\hat{P}_{k_2-1}}^{\epsilon_2}(y)\chi_{\hat{Q}_P}(y).
\end{align*}
For any integer $k_1\leq k_P $ and any integer $k_2\leq k_P$, we let $$F_{P,k_P-k_1,k_P-k_2}^{\epsilon_1,\epsilon_2,0,0}(x):=\Big(\frac{m_{\lambda}(P_{k_1-1})}{m_\lambda(Q_P)}\Big)^{1/2}h_{P_{k_1-1}}^{\epsilon_1}(x)\chi_{Q_P}(x),\ \ G_{P,k_P-k_1,k_P-k_2}^{\epsilon_1,\epsilon_2,0,0}(y):=\Big(\frac{m_{\lambda}(\hat{P}_{k_2-1})}{m_\lambda(Q_P)}\Big)^{1/2}h_{\hat{P}_{k_2-1}}^{\epsilon_2}(y)\chi_{\hat{Q}_P}(y)$$ and
\begin{align*}
B_{P,k_P-k_1,k_P-k_2}^{\epsilon_1,\epsilon_2,0,0}:&=\left(\frac{m_\lambda(Q_P)}{m_{\lambda}(P_{k_1-1})}\right)^{1/2}\left(\frac{m_\lambda(Q_P)}{m_\lambda(\hat{P}_{k_2-1})}\right)^{1/2}\int_{Q_P}\int_{\hat{Q}_P}K_{\lambda,\ell}(z,\,w)h_{P_{k_1-1}}^{\epsilon_1}(z)h_{\hat{P}_{k_2-1}}^{\epsilon_2}(w)dm_\lambda(w)dm_\lambda(z).
\end{align*}
Then, it is direct to see that $\supp F_{P,k_P-k_1,k_P-k_2}^{\epsilon_1,\epsilon_2,0,0}\subset Q_P$, $\supp G_{P,k_P-k_1,k_P-k_2}^{\epsilon_1,\epsilon_2,0,0}\subset \hat{Q}_P$ and $\|F_{P,k_P-k_1,k_P-k_2}^{\epsilon_1,\epsilon_2,0,0}\|_{\infty}+\|G_{P,k_P-k_1,k_P-k_2}^{\epsilon_1,\epsilon_2,0,0}\|_{\infty}\leq m_\lambda(Q_P)^{-1/2}$. Moreover, it follows from Lemma \ref{CZO} and inequality \eqref{doub} that
\begin{align*}
|B_{P,k_P-k_1,k_P-k_2}^{\epsilon_1,\epsilon_2,0,0}|&\lesssim \left(\frac{m_\lambda(Q_P)}{m_\lambda(P_{k_1-1})}\right)\left(\frac{m_\lambda(Q_P)}{m_\lambda(P_{k_2-1})}\right)\lesssim 2^{-(k_P-k_1)(n+1)}2^{-(k_P-k_2)(n+1)}.
\end{align*}
Furthermore, a change of variable yields that $${\rm I}_4=\sum_{\epsilon_1\in\Gamma_n}\sum_{\epsilon_2\in\Gamma_n}\sum_{k_1\geq0}\sum_{k_2\geq0}B_{P,k_1,k_2}^{\epsilon_1,\epsilon_2,0,0}F_{P,k_1,k_2}^{\epsilon_1,\epsilon_2,0,0}(x)G_{P,k_1,k_2}^{\epsilon_1,\epsilon_2,0,0}(y).$$

Finally, by relabeling $\nu=(\epsilon_1,\,\epsilon_2)$ and $k=(k_1,\,k_2,\,j_1,\,j_2)$, setting $\mathcal{I}=\Gamma_n\times \Gamma_n$ and setting some of the $B'$, $F'$, $G'$ to be the zero functions for notational simplicity, we can see that $U_P$ can be rewritten as the form \eqref{reform} with the coefficient satisfying the required estimates (the reason why we rewrite the sum as this form is only for simplicity). For example, when $k_1\geq 0$ and $k_2\leq -1$, we set $B_{P,k,v}=F_{P,k,v}=G_{P,k,v}=0$ if $j_1\neq 0$ or $j_2\leq 0$ or $j_2\geq 2^{-k_2(n+1)}$. In this case, if we choose $\mathcal{K}>(n+1)N$, then
\begin{align*}
|B_{P,k,v}|&\lesssim 2^{-k_1(n+1)} 2^{k_2\mathcal{K}}\chi_{[1,2^{-k_2(n+1)}]}(j_2)\chi_{\{0\}}(j_1)\\
&\leq C_N(1+|k_1|)^{-N}(1+|j_1|)^{-N}(1+|j_2|)^{-N}2^{k_2(\mathcal{K}-(n+1)N)}\\&\leq C_N (1+|k|)^{-N}.
\end{align*}
The proof of Lemma \ref{jqss} is complete.
\end{proof}
The following Lemma states that the kernel of Bessel--Riesz transform commutator can be also decomposed  locally into a summation of NWO sequences with a separate variable form and suitable decay factors.
\begin{lemma}\label{key2}
For any $N\in\mathbb{Z}_+$, there exist a finite index set $\mathcal{J}$, a sequence $B_{Q,k,v}$  satisfying
$$|B_{Q,k,v}|\leq C_N(1+|k|)^{-N},$$
a sequence $\beta_{Q,v}$ satisfying
$$|\beta_{Q,v}|\lesssim\left(\fint_{B_{\mathbb{R}_+^{n+1}}(c_Q,\,K\ell(Q))}|b(u)-(b)_{Q}|^3dm_\lambda(u)\right)^{1/3}\ {\rm for}\ {\rm some}\ {\rm large}\ {\rm constant}\ K,$$
 two functions $F_{Q,k,v}$ and $G_{Q,k,v}$ satisfying $\supp F_{Q,k,v}\subset Q$, $\supp G_{Q,k,v}\subset \hat{Q}$ and
$$\|F_{Q,k,v}\|_{L^3(\mathbb{R}_+^{n+1},dm_\lambda)}+\|G_{Q,k,v}\|_{L^3(\mathbb{R}_+^{n+1},dm_\lambda)}\lesssim m_\lambda(Q)^{-1/6}$$
for any $Q\in \mathcal{D}^0$, $k\in\mathbb{Z}$ and $v\in\mathcal{J}$,
such that the kernel of $[b,R_{\lambda,\ell}]$ can be factorized as
\begin{align*}
\sum_{Q\in\mathcal{D}^0}\sum_{v\in\mathcal{J}}\sum_{k\in\mathbb{Z}^4}B_{Q,k,v}\beta_{Q,v}F_{Q,k,v}(x)G_{Q,k,v}(y).
\end{align*}
\end{lemma}
\begin{proof}
By Lemma \ref{jqss}, the kernel  of $[b,R_{\lambda,\ell}]$ can be factorized as
\begin{align}\label{fur1}
\sum_{P\in\mathcal{P}}\sum_{v\in\mathcal{I}}\sum_{k\in\mathbb{Z}^4}(b(x)-b(y))B_{P,k,v}F_{P,k,v}(x)G_{P,k,v}(y),
\end{align}
for some functions $F_{P,k,v}$, $G_{P,k,v}$, a sequence $\{B_{P,k,v}\}$ and an index set $\mathcal{I}$ satisfying (1)--(3) listed in Lemma \ref{jqss}.
Now for any $P\in\mathcal{P}$, we write $b(x)-b(y)$ as $(b(x)-(b)_{Q_P})-(b(y)-(b)_{Q_P})$. Then \eqref{fur1} can be further factorized as
\begin{align}\label{fur2}
&\sum_{P\in\mathcal{P}}\sum_{v\in\mathcal{I}}\sum_{k\in\mathbb{Z}^4}B_{P,k,v}(b(x)-(b)_{Q_P})F_{P,k,v}(x)G_{P,k,v}(y)\nonumber\\
&\hspace{1.0cm}+\sum_{P\in\mathcal{P}}\sum_{v\in\mathcal{I}}\sum_{k\in\mathbb{Z}^4}B_{P,k,v}F_{P,k,v}(x)((b)_{Q_P}-b(y))G_{P,k,v}(y)\nonumber\\
&=\sum_{P\in\mathcal{P}}\sum_{v\in\mathcal{I}}\sum_{k\in\mathbb{Z}^4}\sum_{j=0}^1B_{P,k,v}\beta_{P}F_{P,k,v,j}(x)G_{P,k,v,j}(y),
\end{align}
where we introduce the following notations:
\begin{align*}
&\beta_{P}:=\left(m_\lambda(Q_P)^{-1}\int_{Q_P\cup\hat{Q}_P}|b(u)-(b)_{Q_P}|^3dm_\lambda(u)\right)^{1/3},\\
&F_{P,k,v,j}(x):=\beta_{P}^{-j}(b(x)-(b)_{Q_P})^jF_{P,k,v}(x),\ j=0,1,\\
&G_{P,k,v,j}(y):=\beta_{P}^{j-1}((b)_{Q_P}-b(y))^{1-j}G_{P,k,v}(y),\ j=0,1.
\end{align*}
Moreover, $\supp F_{P,k,v,j}\subset Q_P$, $\supp G_{P,k,v,j}\subset \hat{Q}_P$, and there is a large constant $K$ such that $Q_P\cup\hat{Q}_P\subset B_{\mathbb{R}_+^{n+1}}(c_{Q_P},\,K\ell(Q_P))$. Furthermore,
\begin{align}
&\beta_{P}\lesssim\left(\fint_{KQ_P}|b(u)-(b)_{Q_P}|^3dm_\lambda(u)\right)^{1/3},\label{requ0}\\
&\|F_{P,k,v,j}\|_{L^3(\mathbb{R}_+^{n+1},dm_\lambda)}\lesssim m_\lambda(Q_P)^{-1/6},\,j=0,1,\\
&\|G_{P,k,v,j}\|_{L^3(\mathbb{R}_+^{n+1},dm_\lambda)}\lesssim m_\lambda(Q_P)^{-1/6},\,j=0,1.\label{requ2}
\end{align}
To continue, we note that the index set in the sum is $\mathcal{P}$, which is a collection of cubes in $\mathbb{R}_+^{n+1}\times\mathbb{R}_+^{n+1}$, so we need to reorganize the sum such that the index set will be the cubes in $\mathcal{D}^0$. To this end, we recall from Lemma \ref{jqss} that for each $P=Q_P\times \hat{Q}_P\in\mathcal{P}$, we have
$$\ell(Q_P)=\ell(\hat{Q}_P)\approx {\rm dist}(Q_P,\,\hat{Q}_{P}).$$
This implies that there is a finite index set $\mathcal{A}$ such that for any $Q\in\mathcal{D}^0$, there are at most finite product-cubes of the form $Q\times R_{Q,s}\in\mathcal{P}$, $s\in\mathcal{A}$.
This allows us to relabel the $F_{P,k,v,j}\ 's$ and $G_{P,k,v,j}\ 's$ in \eqref{fur2} as $F_{Q,s,k,v,j}\ 's$ and $G_{Q,s,k,v,j}\ 's$, respectively, for any $P\in\mathcal{P}$. Then,  we rewrite the kernel  of $[b,R_{\lambda,\ell}]$ as
$$\sum_{Q\in\mathcal{D}^0}\sum_{v\in\mathcal{J}}\sum_{k\in\mathbb{Z}^4}B_{Q,k,v}\beta_{Q,v}F_{Q,k,v}(x)G_{Q,k,v}(y)$$
for some index set $\mathcal{J}=\mathcal{I}\times \mathcal{A}\times\{0,\,1\}$ with $B_{Q,k,v}$, $\beta_{Q,v}$, $F_{Q,k,v}$ and $G_{Q,k,v}$ satisfying the required estimates, which can be seen from Lemma \ref{jqss} together with inequalities \eqref{requ0}--\eqref{requ2}.
\end{proof}
The main result of this subsection is the following upper bound for  Schatten--Lorentz norm of Riesz transform commutator.
\begin{proposition}\label{divide2}
Suppose $1< p<\infty$, $1\leq q\leq \infty$ and $b\in L^p_{\rm loc}(\mathbb{R}^{n+1}_+,dm_\lambda)$. Then there is a constant $C>0$ such that for any $\ell\in\{1,2,...,n+1\}$,
\begin{align*}
\|[b,R_{\lambda,\ell}]\|_{S_\lambda^{p,q}}\leq C\|b\|_{{\rm OSC}_{p,q}(\mathbb{R}_+^{n+1},dm_\lambda)}.
\end{align*}
\end{proposition}
\begin{proof}
To begin with, we let $N$ be a fixed sufficiently large constant and then let $B_{Q,k,v}$, $\beta_{Q,v}$, $F_{Q,k,v}$ and $G_{Q,k,v}$ be the sequences or functions constructed in Lemma \ref{key2}.
Then we apply  inequality \eqref{nwose} to deduce that
\begin{align}\label{dvbh}
\|[b,R_{\lambda,\ell}]\|_{S_\lambda^{p,q}}
&\leq \sum_{v\in\mathcal{J}}\sum_{k\in\mathbb{Z}^4} \left\|\sum_{Q\in\mathcal{D}^0}B_{Q,k,v}\beta_{Q,v}\langle\cdot,\,G_{Q,k,v} \rangle F_{Q,k,v}\right\|_{S_\lambda^{p,q}}\nonumber\\
&\lesssim \sum_{v\in\mathcal{J}}\sum_{k\in\mathbb{Z}^4} \left\|\{B_{Q,k,v}\beta_{Q,v}\}_{Q\in\mathcal{D}^0}\right\|_{\ell^{p,q}}\nonumber\\
&\lesssim \sum_{v\in\mathcal{J}}\sum_{k\in\mathbb{Z}^4}(1+|k|)^{-N} \left\|\left\{\left(\fint_{B_{\mathbb{R}_+^{n+1}}(c_Q,\,K\ell(Q))}|b(u)-(b)_{Q}|^3dm_\lambda(u)\right)^{1/3}\right\}_{Q\in\mathcal{D}^0}\right\|_{\ell^{p,q}}\nonumber\\
&\lesssim \left\|\left\{\left(\fint_{B_{\mathbb{R}_+^{n+1}}(c_Q,\,K\ell(Q))}|b(u)-(b)_{Q}|^3dm_\lambda(u)\right)^{1/3}\right\}_{Q\in\mathcal{D}^0}\right\|_{\ell^{p,q}}\nonumber\\
&\lesssim \left\|\left\{\left(\fint_{B_{\mathbb{R}_+^{n+1}}(c_Q,\,K\ell(Q))}|b(u)-(b)_{B_{\mathbb{R}_+^{n+1}}(c_Q,\,K\ell(Q))}|^3dm_\lambda(u)\right)^{1/3}\right\}_{Q\in\mathcal{D}^0}\right\|_{\ell^{p,q}}
\end{align}
for some large constant $K>0$ given in Lemma \ref{key2}, where in the last inequality we used inequality \eqref{doub}. To continue, we apply Lemma \ref{thm:existence2} to see that for any $Q\in\mathcal{D}^0$, there is a $\iota\in\{1,2,\ldots ,\kappa\}$ and a cube $P_Q\in \mathcal{D}^\iota$ such that
\begin{equation*}
    {B_{\mathbb{R}_+^{n+1}}(c_Q,\,K\ell(Q))}\subseteq P_Q\subseteq C_{{\rm adj}}{B_{\mathbb{R}_+^{n+1}}(c_Q,\,K\ell(Q))}.
\end{equation*}
Note that for any $Q\in\mathcal{D}^0$, we have $Q\subset P_Q$. Therefore, the number of $Q'\in\mathcal{D}^0$ such that $P_Q=P_{Q'}$ is at most $c_n$. This, in combination with Lemma \ref{technic}, implies that
\begin{align}\label{dvbh2}
&\left\|\left\{\left(\fint_{B_{\mathbb{R}_+^{n+1}}(c_Q,\,K\ell(Q))}|b(u)-(b)_{B_{\mathbb{R}_+^{n+1}}(c_Q,\,K\ell(Q))}|^3dm_\lambda(u)\right)^{1/3}\right\}_{Q\in\mathcal{D}^0}\right\|_{\ell^{p,q}}\nonumber\\
&\lesssim \left\|\left\{\left(\fint_{P_Q}|b(u)-(b)_{P_Q}|^3dm_\lambda(u)\right)^{1/3}\right\}_{Q\in\mathcal{D}^0}\right\|_{\ell^{p,q}}\nonumber\\
&\lesssim \sum_{\iota=1}^\kappa\|\{MO_Q^3(b)\}_{Q\in\mathcal{D}^\iota}\|_{\ell^{p,q}}\nonumber\\
&\lesssim \|b\|_{{\rm OSC}_{p,q}(\mathbb{R}_+^{n+1},dm_\lambda)}.
\end{align}
Substituting this inequality into \eqref{dvbh}, we finishes the proof of Proposition \ref{divide2}.
\end{proof}

%---------------------------------------------------------------------------------------
%%
\section{Proof of Theorem \ref{mainLor}:\ lower bound}\label{Sec5}
\setcounter{equation}{0}
This section is devoted to establishing the lower bound for the Schatten--Lorentz norm of the Riesz transform commutator. To begin with, for any $\nu\in\{1,2,\ldots,\kappa\}$ and $Q\in\mathcal{D}^\nu$, let $\hat{Q}$ be the cube chosen in Lemma \ref{sign}. Define
$$J_{Q}(x,\,y):=m_\lambda(Q)^{-1}m_\lambda(\hat{Q})^{-1}K_{\lambda,\ell}(y,\,x)^{-1}\chi_Q(x)\chi_{\hat{Q}}(y).$$
Here we omit the dependence on $\lambda$ and $\ell$ since these two parameters play no role in the proof below.

The following Lemma states that the inverse of Bessel--Riesz transform kernel can be decomposed  locally into a summation of NWO sequences with a separate variable form and a suitable decay factor.
\begin{lemma}\label{jq}
There is a finite index set $\mathcal{I}$ such that
for any $N\in\mathbb{Z}_+$, $\ell\in\{1,2,\ldots,n+1\}$, $\nu\in\{1,2,\ldots,\kappa\}$ and $Q\in\mathcal{D}^\nu$, one can construct  functions $f_{Q,k,u}$ , $g_{Q,k,u}$ and a sequence $\{C_{Q,k,u}\}$ of real numbers satisfying
\begin{enumerate}
  \item ${\rm supp} f_{Q,k,u} \subset Q$, ${\rm supp} g_{Q,k,u} \subset \hat{Q}$,
  \item $\|f_{Q,k,u}\|_{\infty}\leq m_\lambda(Q)^{-1/2}$, $\|g_{Q,k,u}\|_{\infty}\leq m_\lambda(Q)^{-1/2}$,
  \item $|C_{Q,k,u}|\leq C_N(1+|k|)^{-N}$ for some $C_N>0$ independent of  $Q$ and $k$,
\item each $J_{Q}$ admits the following factorization
\begin{align*}
J_{Q}(x,\,y)=\sum_{u\in\mathcal{I}}\sum_{k\in\mathbb{Z}^4}C_{Q,k,u}f_{Q,k,u}(x)g_{Q,k,u}(y),\ x,y\in\mathbb{R}_+^{n+1}.
\end{align*}
\end{enumerate}
\end{lemma}
\begin{proof}
To begin with, we let $\{h_{I,\mathcal{K}}^\epsilon\}_{I\in\mathcal{D}^\nu,\epsilon\in\Gamma_{n,\mathcal{K}}}$ be the Alpert wavelet given in Lemma \ref{AW} with $\mathcal{D}:=\mathcal{D}^\nu$ and $\mathcal{K}$ being a sufficiently large integer. Then we define
\begin{align}
&D_{k,\nu}^{(1)}(f)(x,\,y)=\sum_{I\in\mathcal{D}^\nu_{k-1}}\sum_{\epsilon\in\Gamma_{n,\mathcal{K}}}\langle f(\cdot,\,y), h_{I,\mathcal{K}}^\epsilon\rangle h_{I,\mathcal{K}}^\epsilon(x),\label{dkdef}\\
&D_{k,\nu}^{(2)}(f)(x,\,y)=\sum_{I\in\mathcal{D}^\nu_{k-1}}\sum_{\epsilon\in\Gamma_{n,\mathcal{K}}}\langle f(x,\,\cdot), h_{I,\mathcal{K}}^\epsilon\rangle h_{I,\mathcal{K}}^\epsilon(y).\label{dkdef2}
\end{align}
In what follows, we omit the index $\mathcal{K}$ in $h_{I,\mathcal{K}}^\epsilon$ for simplicity.   By Lemma \ref{AW}, $J_Q$ can be decomposed as follows.
\begin{align}
&J_{Q}(x,\,y)\nonumber\\
&=\left(\sum_{k_1=k_Q+1}^{+\infty}D_{k_1,\nu}^{(1)}+\sum_{k_1=-\infty}^{k_Q}D_{k_1,\nu}^{(1)}\right)\left(\sum_{k_2=k_Q+1}^{+\infty}D_{k_2,\nu}^{(2)}+\sum_{k_2=-\infty}^{k_Q}D_{k_2,\nu}^{(2)}\right)(J_Q)(x,\,y)\chi_{Q}(x)\chi_{\hat{Q}}(y)\nonumber\\
&=\left(\sum_{k_1=k_Q+1}^{+\infty}\sum_{k_2=k_Q+1}^{+\infty}+\sum_{k_1=-\infty}^{k_Q}\sum_{k_2=k_Q+1}^{+\infty}+\sum_{k_1=k_Q+1}^{+\infty}\sum_{k_2=-\infty}^{k_Q}+\sum_{k_1=-\infty}^{k_Q}\sum_{k_2=-\infty}^{k_Q}\right)D_{k_1,\nu}^{(1)}D_{k_2,\nu}^{(2)}(J_Q)(x,\,y)\chi_{Q}(x)\chi_{\hat{Q}}(y)\nonumber\\
&=\left(\sum_{k_1=k_Q+1}^{+\infty}\sum_{k_2=k_Q+1}^{+\infty}+\sum_{k_1=-\infty}^{k_Q}\sum_{k_2=k_Q+1}^{+\infty}+\sum_{k_1=k_Q+1}^{+\infty}\sum_{k_2=-\infty}^{k_Q}+\sum_{k_1=-\infty}^{k_Q}\sum_{k_2=-\infty}^{k_Q}\right)\sum_{I_1\in\mathcal{D}^\nu_{k_1-1}}\sum_{I_2\in\mathcal{D}^\nu_{k_2-1}}\sum_{\epsilon_1\in\Gamma_n}\sum_{\epsilon_2\in\Gamma_n}\nonumber\\ \nonumber\\
&\hspace{1.0cm}m_\lambda(Q)^{-1}m_\lambda(\hat{Q})^{-1}\int_{I_1\cap Q}\int_{I_2\cap \hat{Q}}K_{\lambda,\ell}(z,\,w)^{-1}h_{I_1}^{\epsilon_1}(z)h_{I_2}^{\epsilon_2}(w)dm_\lambda(w)dm_\lambda(z)h_{I_1}^{\epsilon_1}(x)h_{I_2}^{\epsilon_2}(y)\nonumber\\
&=:\sum_{u=1}^4{\rm I}_u,
\end{align}
where we choose $k_Q$ to be the integer such that $\delta^{k_Q}\leq \ell(Q)<\delta^{k_Q-1}$. Next we deal with these four terms separately.

For the term ${\rm I}_1$,  we note that for any integer $k_1\geq k_Q+1 $, there are at most $\delta^{-(k_1-k_Q)(2\lambda+n+1)}$ cubes (up to an absolute constant)  $I\in\mathcal{D}^\nu_{k_1-1}$ such that $I\subset Q$. We denote these cubes by $I_{1,k_1}^{1},\ldots,I_{1,k_1}^{a_{k_1-k_Q,\nu,Q}}$ with $a_{k_1-k_Q,\nu,Q}\lesssim \delta^{-(k_1-k_Q)(2\lambda+n+1)}$. Moreover, for any integer $k_2\geq k_Q+1 $, there are at most $\delta^{-(k_2-k_Q)(2\lambda+n+1)}$ cubes (up to an absolute constant) $I\in\mathcal{D}^\nu_{k_2-1}$ such that $I\subset\hat{Q}$. These cubes are denoted by $I_{2,k_2}^{1},\ldots,I_{2,k_2}^{b_{k_2-k_Q,\nu,Q}}$ with $b_{k_2-k_Q,\nu,Q}\lesssim \delta^{-(k_2-k_Q)(2\lambda+n+1)}$. Then
\begin{align*}
{\rm I}_1
&=\sum_{k_1=k_Q+1}^{+\infty}\sum_{k_2=k_Q+1}^{+\infty}\sum_{j_1=1}^{a_{k_1-k_Q,\nu,Q}}\sum_{j_2=1}^{b_{k_2-k_Q,\nu,Q}}\sum_{\epsilon_1\in\Gamma_n}\sum_{\epsilon_2\in\Gamma_n}m_\lambda(Q)^{-1}m_\lambda(\hat{Q})^{-1}\times\\
&\hspace{1.0cm}\times\int_{I_{1,k_1}^{j_1}}\int_{I_{2,k_2}^{j_2}}K_{\lambda,\ell}(z,\,w)^{-1}h_{I_{1,k_1}^{j_1}}^{\epsilon_1}(z)h_{I_{2,k_2}^{j_2}}^{\epsilon_2}(w)dm_\lambda(w)dm_\lambda(z)h_{I_{1,k_1}^{j_1}}^{\epsilon_1}(x)h_{I_{2,k_2}^{j_2}}^{\epsilon_2}(y).
\end{align*}
For any integer $k_1\geq k_Q+1 $ and any integer $k_2\geq k_Q+1$, we let $$f_{Q,k_Q-k_1,k_Q-k_2}^{\epsilon_1,\epsilon_2,j_1,j_2}(x):=\left(\frac{m_\lambda(I_{1,k_1}^{j_1})}{m_\lambda(Q)}\right)^{1/2}h_{I_{1,k_1}^{j_1}}^{\epsilon_1}(x),\ \ g_{Q,k_Q-k_1,k_Q-k_2}^{\epsilon_1,\epsilon_2,j_1,j_2}(y):=\left(\frac{m_\lambda(I_{2,k_2}^{j_2})}{m_\lambda(Q)}\right)^{1/2}h_{I_{2,k_2}^{j_2}}^{\epsilon_2}(y)$$ and
\begin{align*}
C_{Q,k_Q-k_1,k_Q-k_2}^{\epsilon_1,\epsilon_2,j_1,j_2}:&=m_\lambda(Q)^{-1}m_\lambda(\hat{Q})^{-1}\left(\frac{m_\lambda(Q)}{m_{\lambda}(I_{1,k_1}^{j_1})}\right)^{1/2}\left(\frac{m_\lambda(Q)}{m_\lambda(I_{2,k_2}^{j_2})}\right)^{1/2}\times\\
&\hspace{2.0cm}\times\int_{I_{1,k_1}^{j_1}}\int_{I_{2,k_2}^{j_2}}K_{\lambda,\ell}(z,\,w)^{-1}h_{I_{1,k_1}^{j_1}}^{\epsilon_1}(z)h_{I_{2,k_2}^{j_2}}^{\epsilon_2}(w)dm_\lambda(w)dm_\lambda(z).
\end{align*}
Then, it is direct to see that $\supp f_{Q,k_Q-k_1,k_Q-k_2}^{\epsilon_1,\epsilon_2,j_1,j_2}\subset Q$, $\supp g_{Q,k_Q-k_1,k_Q-k_2}^{\epsilon_1,\epsilon_2,j_1,j_2}\subset \hat{Q}$ and $\|f_{Q,k_Q-k_1,k_Q-k_2}^{\epsilon_1,\epsilon_2,j_1,j_2}\|_{\infty}+\|g_{Q,k_Q-k_1,k_Q-k_2}^{\epsilon_1,\epsilon_2,j_1,j_2}\|_{\infty}\leq m_\lambda(Q)^{-1/2}$. Moreover, we apply Taylor's formula twice to deduce that there are points $\xi_{z,j_1,k_1}$ and $\eta_{w,k_2,j_2}$ lying in the line segments jointing $z$ with $c_{I_{1,k_1}^{j_1}}$ and jointing $w$ with $c_{I_{2,k_2}^{j_2}}$, respectively, such that
\begin{align*}
K_{\lambda,\ell}(z,\,w)^{-1}
&=\sum_{|\beta|\leq \mathcal{K}-1}\frac{\partial_z^{\beta}[K_{\lambda,\ell}(c_{I_{1,k_1}^{j_1}},\,w)^{-1}]}{\beta!}(z-c_{I_{1,k_1}^{j_1}})^{\beta}\\
&\hspace{1.0cm}+\sum_{|\beta|=\mathcal{K}}\sum_{|\gamma|\leq \mathcal{K}-1}\frac{\partial_z^\beta\partial_w^{\gamma}[K_{\lambda,\ell}(\xi_{z,j_1,k_1},\,c_{I_{2,k_2}^{j_2}})^{-1}]}{\beta!\gamma!}(z-c_{I_{1,k_1}^{j_1}})^{\beta}(w-c_{I_{2,k_2}^{j_2}})^\gamma\\
&\hspace{1.0cm}+\sum_{|\beta|=\mathcal{K} }\sum_{|\gamma|=\mathcal{K} }\frac{\partial_z^\beta \partial_w^\gamma [K_{\lambda,\ell}(\xi_{z,j_1,k_1},\,\eta_{w,j_2,k_2})^{-1}]}{\beta!\gamma!}(z-c_{I_{1,k_1}^{j_1}})^{\beta}(w-c_{I_{2,k_2}^{j_2}})^{\gamma}\\
&=:\mathcal{H}_{\lambda,\ell,k_1,k_2,j_1,j_2}^{\beta,\gamma,\mathcal{K}}(z,\,w)\\
&\hspace{1.0cm}+\sum_{|\beta|=\mathcal{K} }\sum_{|\gamma|=\mathcal{K} }\frac{\partial_z^\beta \partial_w^\gamma [K_{\lambda,\ell}(\xi_{z,j_1,k_1},\,\eta_{w,j_2,k_2})]^{-1}}{\beta!\gamma!}(z-c_{I_{1,k_1}^{j_1}})^{\beta}(w-c_{I_{2,k_2}^{j_2}})^{\gamma}.
\end{align*}
Next, we apply Lemmas \ref{AW}, \ref{CZO} and \ref{sign} to deduce that
\begin{align*}
&\bigg|\int_{I_{1,k_1}^{j_1}}\int_{I_{2,k_2}^{j_2}}K_{\lambda,\ell}(z,\,w)^{-1}h_{I_{1,k_1}^{j_1}}^{\epsilon_1}(z)h_{I_{2,k_2}^{j_2}}^{\epsilon_2}(w)dm_\lambda(w)dm_\lambda(z)\bigg|\\
&=\bigg|\int_{I_{1,k_1}^{j_1}}\int_{I_{2,k_2}^{j_2}}\big(K_{\lambda,\ell}(z,\,w)^{-1}-\mathcal{H}_{\lambda,\ell,k_1,k_2,j_1,j_2}^{\beta,\gamma,\mathcal{K}}(z,\,w)\big) h_{I_{1,k_1}^{j_1}}^{\epsilon_1}(z)h_{I_{2,k_2}^{j_2}}^{\epsilon_2}(w)dm_\lambda(w)dm_\lambda(z)\bigg|\\
&=\bigg|\int_{I_{1,k_1}^{j_1}}\int_{I_{2,k_2}^{j_2}} \sum_{|\beta|=\mathcal{K} }\sum_{|\gamma|=\mathcal{K} }\frac{\partial_z^\beta \partial_w^\gamma [K_{\lambda,\ell}(\xi_{z,j_1,k_1},\,\eta_{w,j_2,k_2})^{-1}]}{\beta!\gamma!}\times\\
&\hspace{3.9cm}\times(z-c_{I_{1,k_1}^{j_1}})^{\beta}(w-c_{I_{2,k_2}^{j_2}})^{\gamma} h_{I_{1,k_1}^{j_1}}^{\epsilon_1}(z)h_{I_{2,k_2}^{j_2}}^{\epsilon_2}(w)dm_\lambda(w)dm_\lambda(z)\bigg|\\
&\lesssim \delta^{(k_1-k_Q)\mathcal{K}}\delta^{(k_2-k_Q)\mathcal{K}}\left(\frac{m_\lambda(I_{1,k_1}^{j_1})}{m_\lambda(Q)}\right)^{1/2}\left(\frac{m_\lambda(I_{2,k_2}^{j_2})}{m_\lambda(Q)}\right)^{1/2}m_\lambda(Q)m_\lambda(\hat{Q}).
\end{align*}
That is,
\begin{align*}
\Big|C_{Q,k_Q-k_1,k_Q-k_2}^{\epsilon_1,\epsilon_2,j_1,j_2}\Big|\lesssim \delta^{(k_1-k_Q)\mathcal{K}}\delta^{(k_2-k_Q)\mathcal{K}}.
\end{align*}
Furthermore, a change of variable yields that $${\rm I}_1=\sum_{k_1\leq -1}\sum_{k_2\leq -1}\sum_{j_1=1}^{a_{-k_1,\nu,Q}}\sum_{j_2=1}^{b_{-k_2,\nu,Q}}\sum_{\epsilon_1\in\Gamma_n}\sum_{\epsilon_2\in\Gamma_n}C_{Q,k_1,k_2}^{\epsilon_1,\epsilon_2,j_1,j_2}f_{Q,k_1,k_2}^{\epsilon_1,\epsilon_2,j_1,j_2}(x)g_{Q,k_1,k_2}^{\epsilon_1,\epsilon_2,j_1,j_2}(y).$$

For the term ${\rm I}_2$, we note that for any integer $k\leq k_Q$, there is a unique $I\in\mathcal{D}^\nu_{k-1}$ such that $I\cap Q\neq\emptyset$. We denote this cube by $Q_{k-1}$. Then
\begin{align*}
{\rm I}_2
&=\sum_{k_1=-\infty}^{k_Q}\sum_{k_2=k_Q+1}^{+\infty}\sum_{j_2=1}^{b_{k_2-k_Q,\nu,Q}}\sum_{\epsilon_1\in\Gamma_n}\sum_{\epsilon_2\in\Gamma_n}m_\lambda(Q)^{-1}m_\lambda(\hat{Q})^{-1}\times\\
&\hspace{1.0cm}\times\int_{Q}\int_{I_{2,k_2}^{j_2}}K_{\lambda,\ell}(z,\,w)^{-1}h_{Q_{k_1-1}}^{\epsilon_1}(z)h_{I_{2,k_2}^{j_2}}^{\epsilon_2}(w)dm_\lambda(w)dm_\lambda(z)h_{Q_{k_1-1}}^{\epsilon_1}(x)\chi_{Q}(x)h_{I_{2,k_2}^{j_2}}^{\epsilon_2}(y).
\end{align*}
For any integer $k_1\leq k_{Q} $ and any integer $k_2\geq k_Q+1$, we let $$f_{Q,k_Q-k_1,k_Q-k_2}^{\epsilon_1,\epsilon_2,0,j_2}(x):=\left(\frac{m_{\lambda}(Q_{k_1-1})}{m_\lambda(Q)}\right)^{1/2}h_{Q_{k_1-1}}^{\epsilon_1}(x)\chi_{Q}(x),\ \  g_{Q,k_Q-k_1,k_Q-k_2}^{\epsilon_1,\epsilon_2,0,j_2}(y):=\left(\frac{m_\lambda(I_{2,k_2}^{j_2})}{m_\lambda(Q)}\right)^{1/2}h_{I_{2,k_2}^{j_2}}^{\epsilon_2}(y)$$ and \begin{align*}
C_{Q,k_Q-k_1,k_Q-k_2}^{\epsilon_1,\epsilon_2,0,j_2}:&=m_\lambda(Q)^{-1}m_\lambda(\hat{Q})^{-1}\left(\frac{m_\lambda(Q)}{m_{\lambda}(Q_{k_1-1})}\right)^{1/2}\left(\frac{m_\lambda(Q)}{m_\lambda(I_{2,k_2}^{j_2})}\right)^{1/2}\times\\
&\hspace{1.0cm}\times\int_{Q}\int_{I_{2,k_2}^{j_2}}K_{\lambda,\ell}(z,\,w)^{-1}h_{Q_{k_1-1}}^{\epsilon_1}(z)h_{I_{2,k_2}^{j_2}}^{\epsilon_2}(w)dm_\lambda(w)dm_\lambda(z).\end{align*}
Then, it is direct to see that $\supp f_{Q,k_Q-k_1,k_Q-k_2}^{\epsilon_1,\epsilon_2,0,j_2}\subset Q$, $\supp g_{Q,k_Q-k_1,k_Q-k_2}^{\epsilon_1,\epsilon_2,0,j_2}\subset \hat{Q}$ and $\|f_{Q,k_Q-k_1,k_Q-k_2}^{\epsilon_1,\epsilon_2,0,j_2}\|_{\infty}+\|g_{Q,k_Q-k_1,k_Q-k_2}^{\epsilon_1,\epsilon_2,0,j_2}\|_{\infty}\leq m_\lambda(Q)^{-1/2}$. Moreover, we apply Taylor's formula with respect to the $w$-variable at $c_{I_{2,k_2}^{j_2}}$ to deduce that there is a point $\eta_{w,k_2,j_2}$ lying in the line segment jointing $w$ with $c_{I_{2,k_2}^{j_2}}$ such that
$$K_{\lambda,\ell}(z,\,w)^{-1}=\sum_{|\gamma|\leq \mathcal{K}-1}\frac{\partial_w^\gamma [K_{\lambda,\ell}(z,\,c_{I_{2,k_2}^{j_2}})^{-1}]}{\gamma!}(w-c_{I_{2,k_2}^{j_2}})^\gamma+\sum_{|\gamma|= \mathcal{K}}\frac{\partial_w^\gamma [K_{\lambda,\ell}(z,\,\eta_{w,k_2,j_2})^{-1}]}{\gamma!}(w-c_{I_{2,k_2}^{j_2}})^\gamma.$$
Then we  apply Lemmas \ref{AW}, \ref{CZO} and \ref{sign} to deduce that
\begin{align*}
&\left|\int_{Q}\int_{I_{2,k_2}^{j_2}}K_{\lambda,\ell}(z,\,w)^{-1}h_{Q_{k_1-1}}^{\epsilon_1}(z)h_{I_{2,k_2}^{j_2}}^{\epsilon_2}(w)dm_\lambda(w)dm_\lambda(z)\right|\\
&=\left|\int_{Q}\int_{I_{2,k_2}^{j_2}}\left(K_{\lambda,\ell}(z,\,w)^{-1}-\sum_{|\gamma|\leq \mathcal{K}-1}\frac{\partial_w^\gamma [K_{\lambda,\ell}(z,\,c_{I_{2,k_2}^{j_2}})^{-1}]}{\gamma!}(w-c_{I_{2,k_2}^{j_2}})^\gamma\right)h_{Q_{k_1-1}}^{\epsilon_1}(z)h_{I_{2,k_2}^{j_2}}^{\epsilon_2}(w)dm_\lambda(w)dm_\lambda(z)\right|\\
&=\left|\int_{Q}\int_{I_{2,k_2}^{j_2}}\sum_{|\gamma|= \mathcal{K}}\frac{\partial_w^\gamma [K_{\lambda,\ell}(z,\,\eta_{w,k_2,j_2})^{-1}]}{\gamma!}(w-c_{I_{2,k_2}^{j_2}})^\gamma h_{Q_{k_1-1}}^{\epsilon_1}(z)h_{I_{2,k_2}^{j_2}}^{\epsilon_2}(w)dm_\lambda(w)dm_\lambda(z)\right|\\
&\lesssim \delta^{(k_2-k_Q)\mathcal{K}}\left(\frac{m_\lambda(I_{2,k_2}^{j_2})}{m_\lambda(Q_{k_1-1})}\right)^{1/2}m_\lambda(Q)m_\lambda(\hat{Q}).
\end{align*}
This, in combination with inequality \eqref{doub}, yields
$$
\Big|C_{Q,k_Q-k_1,k_Q-k_2}^{\epsilon_1,\epsilon_2,0,j_2}\Big|\lesssim \delta^{(k_2-k_Q)\mathcal{K}}\frac{m_\lambda(Q)}{m_{\lambda}(Q_{k_1-1})}\lesssim \delta^{(k_Q-k_1)(n+1)} \delta^{(k_2-k_Q)\mathcal{K}}.
$$
Furthermore, a change of variable yields that $${\rm I}_2=\sum_{k_1\geq 0}\sum_{k_2\leq -1}\sum_{j_2=1}^{b_{-k_2,\nu,Q}}\sum_{\epsilon_1\in\Gamma_n}\sum_{\epsilon_2\in\Gamma_n}C_{Q,k_1,k_2}^{\epsilon_1,\epsilon_2,0,j_2} f_{Q,k_1,k_2}^{\epsilon_1,\epsilon_2,0,j_2}(x)g_{Q,k_1,k_2}^{\epsilon_1,\epsilon_2,0,j_2}(y).$$

For the term ${\rm I}_3$,  we note that for any integer $k\leq k_Q$, there is a unique $I\in\mathcal{D}^\nu_{k-1}$ such that $I\cap \hat{Q}\neq\emptyset$. We denote this cube by $\hat{Q}_{k-1}$. Then
\begin{align*}
{\rm I}_3
&=\sum_{k_1=k_Q+1}^{+\infty}\sum_{k_2=-\infty}^{k_Q}\sum_{j_1=1}^{a_{k_1-k_Q,\nu,Q}}\sum_{\epsilon_1\in\Gamma_n}\sum_{\epsilon_2\in\Gamma_n}m_\lambda(Q)^{-1}m_\lambda(\hat{Q})^{-1}\times\\
&\hspace{1.0cm}\times\int_{I_{1,k_1}^{j_1}}\int_{\hat{Q}}K_{\lambda,\ell}(z,\,w)^{-1}h_{I_{1,k_1}^{j_1}}^{\epsilon_1}(z)h_{\hat{Q}_{k_2-1}}^{\epsilon_2}(w)dm_\lambda(w)dm_\lambda(z)h_{I_{1,k_1}^{j_1}}^{\epsilon_1}(x)h_{\hat{Q}_{k_2-1}}^{\epsilon_2}(y)\chi_{\hat{Q}}(y).
\end{align*}
For any integer $k_1\geq k_{Q}+1 $ and any integer $k_2\leq k_Q$, we let $$f_{Q,k_Q-k_1,k_Q-k_2}^{\epsilon_1,\epsilon_2,j_1,0}(x):=\left(\frac{m_\lambda(I_{1,k_1}^{j_1})}{m_\lambda(Q)}\right)^{1/2}h_{I_{1,k_1}^{j_1}}^{\epsilon_1}(x),\ \  g_{Q,k_Q-k_1,k_Q-k_2}^{\epsilon_1,\epsilon_2,j_1,0}(y):=\left(\frac{m_{\lambda}(\hat{Q}_{k_2-1})}{m_\lambda(Q)}\right)^{1/2}h_{\hat{Q}_{k_2-1}}^{\epsilon_2}(y)\chi_{\hat{Q}}(y)$$ and \begin{align*}
C_{Q,k_Q-k_1,k_Q-k_2}^{\epsilon_1,\epsilon_2,j_1,0}:&=m_\lambda(Q)^{-1}m_\lambda(\hat{Q})^{-1}\left(\frac{m_\lambda(Q)}{m_{\lambda}(\hat{Q}_{k_2-1})}\right)^{1/2}\left(\frac{m_\lambda(Q)}{m_\lambda(I_{1,k_1}^{j_1})}\right)^{1/2}\times\\
&\hspace{1.0cm}\times \int_{I_{1,k_1}^{j_1}}\int_{\hat{Q}}K_{\lambda,\ell}(z,\,w)^{-1}h_{I_{1,k_1}^{j_1}}^{\epsilon_1}(z)h_{\hat{Q}_{k_2-1}}^{\epsilon_2}(w)dm_\lambda(w)dm_\lambda(z).
\end{align*}
Then, it is direct to see that $\supp f_{Q,k_Q-k_1,k_Q-k_2}^{\epsilon_1,\epsilon_2,j_1,0}\subset Q$, $\supp g_{Q,k_Q-k_1,k_Q-k_2}^{\epsilon_1,\epsilon_2,j_1,0}\subset \hat{Q}$ and $\|f_{Q,k_Q-k_1,k_Q-k_2}^{\epsilon_1,\epsilon_2,j_1,0}\|_{\infty}+\|g_{Q,k_Q-k_1,k_Q-k_2}^{\epsilon_1,\epsilon_2,j_1,0}\|_{\infty}\leq m_\lambda(Q)^{-1/2}$. Moreover,  we apply Taylor's formula with respect to the $z$-variable at $c_{I_{1,k_1}^{j_1}}$ to deduce that there is a point $\xi_{z,k_1,j_1}$ lying in the line segment jointing $z$ with $c_{I_{1,k_1}^{j_1}}$ such that
$$K_{\lambda,\ell}(z,\,w)^{-1}=\sum_{|\beta|\leq \mathcal{K}-1}\frac{\partial_z^\beta [K_{\lambda,\ell}(c_{I_{1,k_1}^{j_1}},\,w)^{-1}]}{\beta!}(z-c_{I_{1,k_1}^{j_1}})^\beta+\sum_{|\beta|= \mathcal{K}}\frac{\partial_z^\beta [K_{\lambda,\ell}(\xi_{z,k_1,j_1},\,w)^{-1}]}{\gamma!}(z-c_{I_{1,k_1}^{j_1}})^\beta.$$
Then we  apply Lemmas \ref{AW}, \ref{CZO} and \ref{sign} to deduce that
\begin{align*}
&\bigg|\int_{I_{1,k_1}^{j_1}}\int_{\hat{Q}}K_{\lambda,\ell}(z,\,w)^{-1}h_{I_{1,k_1}^{j_1}}^{\epsilon_1}(z)h_{\hat{Q}_{k_2-1}}^{\epsilon_2}(w)dm_\lambda(w)dm_\lambda(z)\bigg|\\
&=\bigg|\int_{I_{1,k_1}^{j_1}}\int_{\hat{Q}}\bigg(K_{\lambda,\ell}(z,\,w)^{-1}-\sum_{|\beta|\leq \mathcal{K}-1}\frac{\partial_z^\beta [K_{\lambda,\ell}(c_{I_{1,k_1}^{j_1}},\,w)^{-1}]}{\beta!}(z-c_{I_{1,k_1}^{j_1}})^\beta\bigg)\times \\
&\hspace{7.6cm}\times h_{I_{1,k_1}^{j_1}}^{\epsilon_1}(z)h_{\hat{Q}_{k_2-1}}^{\epsilon_2}(w)dm_\lambda(w)dm_\lambda(z)\bigg|\\
&=\bigg|\int_{I_{1,k_1}^{j_1}}\int_{\hat{Q}}\sum_{|\beta|= \mathcal{K}}\frac{\partial_z^\beta [K_{\lambda,\ell}(\xi_{z,k_1,j_1},\,w)^{-1}]}{\gamma!}(z-c_{I_{1,k_1}^{j_1}})^\beta h_{I_{1,k_1}^{j_1}}^{\epsilon_1}(z)h_{\hat{Q}_{k_2-1}}^{\epsilon_2}(w)dm_\lambda(w)dm_\lambda(z)\bigg|\\
&\lesssim \delta^{(k_1-k_Q)\mathcal{K}}\bigg(\frac{m_\lambda(I_{1,k_1}^{j_1})}{m_\lambda(\hat{Q}_{k_2-1})}\bigg)^{1/2}m_\lambda(Q)m_\lambda(\hat{Q}).
\end{align*}
This, in combination with inequality \eqref{doub}, yields
$$
\Big|C_{Q,k_Q-k_1,k_Q-k_2}^{\epsilon_1,\epsilon_2,j_1,0}\Big|\lesssim \delta^{(k_1-k_Q)\mathcal{K}}\frac{m_\lambda(Q)}{m_{\lambda}(\hat{Q}_{k_2-1})}\lesssim \delta^{(k_1-k_Q)\mathcal{K}}\delta^{(k_Q-k_2)(n+1)}.
$$
Furthermore, a change of variable yields that $${\rm I}_3=\sum_{k_1\leq -1}\sum_{k_2\geq 0}\sum_{j_1=1}^{a_{-k_1,\nu,Q}}\sum_{\epsilon_1\in\Gamma_n}\sum_{\epsilon_2\in\Gamma_n}C_{Q,k_1,k_2}^{\epsilon_1,\epsilon_2,j_1,0} f_{Q,k_1,k_2}^{\epsilon_1,\epsilon_2,j_1,0}(x)g_{Q,k_1,k_2}^{\epsilon_1,\epsilon_2,j_1,0}(y).$$

For the term ${\rm I}_4$,  we see that
\begin{align*}
{\rm I}_4
&=\sum_{k_1=-\infty}^{k_Q}\sum_{k_2=-\infty}^{k_Q}\sum_{\epsilon_1\in\Gamma_n}\sum_{\epsilon_2\in\Gamma_n}m_\lambda(Q)^{-1}m_\lambda(\hat{Q})^{-1}\times\\
&\hspace{1.0cm}\times\int_{Q}\int_{\hat{Q}}K_{\lambda,\ell}(z,\,w)^{-1}h_{Q_{k_1-1}}^{\epsilon_1}(z)h_{\hat{Q}_{k_2-1}}^{\epsilon_2}(w)dm_\lambda(w)dm_\lambda(z)h_{Q_{k_1-1}}^{\epsilon_1}(x)\chi_{Q}(x)h_{\hat{Q}_{k_2-1}}^{\epsilon_2}(y)\chi_{\hat{Q}}(y).
\end{align*}
For any integer $k_1\leq k_Q $ and any integer $k_2\leq k_Q$, we let $$f_{Q,k_Q-k_1,k_Q-k_2}^{\epsilon_1,\epsilon_2,0,0}(x):=\Big(\frac{m_{\lambda}(Q_{k_1-1})}{m_\lambda(Q)}\Big)^{1/2}h_{Q_{k_1-1}}^{\epsilon_1}(x)\chi_{Q}(x),\ \ g_{Q,k_Q-k_1,k_Q-k_2}^{\epsilon_1,\epsilon_2,0,0}(y):=\Big(\frac{m_{\lambda}(\hat{Q}_{k_2-1})}{m_\lambda(Q)}\Big)^{1/2}h_{\hat{Q}_{k_2-1}}^{\epsilon_2}(y)\chi_{\hat{Q}}(y)$$ and
\begin{align*}
C_{Q,k_Q-k_1,k_Q-k_2}^{\epsilon_1,\epsilon_2,0,0}:&=m_\lambda(Q)^{-1}m_\lambda(\hat{Q})^{-1}\left(\frac{m_\lambda(Q)}{m_{\lambda}(Q_{k_1-1})}\right)^{1/2}\left(\frac{m_\lambda(Q)}{m_\lambda(\hat{Q}_{k_2-1})}\right)^{1/2}\times\\
&\hspace{1.0cm}\times \int_{Q}\int_{\hat{Q}}K_{\lambda,\ell}(z,\,w)^{-1}h_{Q_{k_1-1}}^{\epsilon_1}(z)h_{\hat{Q}_{k_2-1}}^{\epsilon_2}(w)dm_\lambda(w)dm_\lambda(z).
\end{align*}
Then, it is direct to see that $\supp f_{Q,k_Q-k_1,k_Q-k_2}^{\epsilon_1,\epsilon_2,0,0}\subset Q$, $\supp g_{Q,k_Q-k_1,k_Q-k_2}^{\epsilon_1,\epsilon_2,0,0}\subset \hat{Q}$ and $\|f_{Q,k_Q-k_1,k_Q-k_2}^{\epsilon_1,\epsilon_2,0,0}\|_{\infty}+\|g_{Q,k_Q-k_1,k_Q-k_2}^{\epsilon_1,\epsilon_2,0,0}\|_{\infty}\leq m_\lambda(Q)^{-1/2}$. Moreover, it follows from Lemmas \ref{CZO}, \ref{sign} and inequality \eqref{doub} that
\begin{align*}
|C_{Q,k_Q-k_1,k_Q-k_2}^{\epsilon_1,\epsilon_2,0,0}|&\lesssim \left(\frac{m_\lambda(Q)}{m_\lambda(Q_{k_1-1})}\right)\left(\frac{m_\lambda(Q)}{m_\lambda(Q_{k_2-1})}\right)\lesssim \delta^{(k_Q-k_1)(n+1)}\delta^{(k_Q-k_2)(n+1)}.
\end{align*}
Furthermore, a change of variable yields that $${\rm I}_4=\sum_{\epsilon_1\in\Gamma_n}\sum_{\epsilon_2\in\Gamma_n}\sum_{k_1\geq0}\sum_{k_2\geq0}C_{Q,k_1,k_2}^{\epsilon_1,\epsilon_2,0,0}f_{Q,k_1,k_2}^{\epsilon_1,\epsilon_2,0,0}(x)g_{Q,k_1,k_2}^{\epsilon_1,\epsilon_2,0,0}(y).$$

Finally, by relabeling $\nu=(\epsilon_1,\,\epsilon_2)$ and $k=(k_1,\,k_2,\,j_1,\,j_2)$, setting $\mathcal{I}=\Gamma_n\times \Gamma_n$ and setting some of the $C'$, $f'$, $g'$ to be the zero functions for notational simplicity, we can see that $J_Q$ can be rewritten as the form \eqref{reform} with the coefficient satisfying the required estimates (the reason why we rewrite the sum as this form is only for simplicity). For example, when $k_1\geq 0$ and $k_2\leq -1$, we set $C_{Q,k,v}=f_{Q,k,v}=g_{Q,k,v}=0$ if $j_1\neq 0$ or $j_2\leq 0$ or $j_2\geq b_{-k_2,\nu,Q}$. In this case, if we choose $\mathcal{K}>(n+1)N$, then
\begin{align*}
|C_{Q,k,v}|&\lesssim 2^{-k_1(n+1)} 2^{k_2\mathcal{K}}\chi_{[1,b_{-k_2,\nu,Q}]}(j_2)\chi_{\{0\}}(j_1)\\
&\leq C_N(1+|k_1|)^{-N}(1+|j_1|)^{-N}(1+|j_2|)^{-N}2^{k_2(\mathcal{K}-(n+1)N)}\\&\leq C_N (1+|k|)^{-N}.
\end{align*}
The proof of Lemma \ref{jq} is complete.
\end{proof}
The main result of this subsection is the following lower bound for  Schatten--Lorentz norm of Riesz transform commutator.
\begin{proposition}\label{weakpro}
Suppose $1<p<\infty$, $1\leq q\leq \infty$ and $b\in L^p_{\rm loc}(\mathbb{R}^{n+1}_+,dm_\lambda)$. Then there is a constant $C>0$ such that for any $\ell\in\{1,2,...,n+1\}$,
\begin{align*}
\|b\|_{{\rm OSC}_{p,q}(\mathbb{R}_+^{n+1},dm_\lambda)}\leq C\|[b,R_{\lambda,\ell}]\|_{S_\lambda^{p,q}}.
\end{align*}

\end{proposition}
\begin{proof}
By Lemma \ref{sign}, Then there exist constants  $c_1,c_2>0$ such that for any $\nu\in\{1,2,\ldots,\kappa\}$ and $Q\in \ \mathcal{D}^\nu_{k}$ with center $c_Q$, one can find a cube $\hat{Q}\in\mathcal{D}_k^\nu$ such that $c_1\delta^k\leq {\rm dist}(Q,\,\hat{Q})\leq c_2\delta^k$, and for all $(x,\,y)\in Q\times\hat{Q}$, $K_{\lambda,\ell}(x,\,y)$ does not change sign and satisfies
\begin{align*}
|K_{\lambda,\ell}(x,\,y)|\gtrsim \frac{1}{m_\lambda(Q)}.
\end{align*}
To continue, we let $s_Q(x):={\rm sgn}(b(x)-(b)_{\hat{Q}})\chi_Q(x)$ and deduce that
\begin{align}\label{tr1}
MO_Q(b)
&\lesssim \fint_Q|b(x)-(b)_{\hat{Q}}|dm_\lambda(x)\nonumber\\
&=C\fint_Q(b(x)-(b)_{\hat{Q}})s_Q(x)dm_\lambda(x).
\end{align}
Furthermore, we let $$L_{Q}(f)(x):=\int_{\hat{Q}}s_Q(x)J_{Q}(x,\,y)f(y)dm_\lambda(y),$$ where $J_{Q}$ is given in Lemma \ref{jq}. Then a direct calculation yields that
$$[b,R_{\lambda,\ell}]L_{Q}(f)(x)=\fint_Q\fint_{\hat{Q}}(b(x)-b(z))K_{\lambda,\ell}(x,\,z)s_Q(z)K_{\lambda,\ell}(y,\,z)^{-1}f(y)dm_\lambda(y)dm_\lambda(z).$$
This implies that
\begin{align}\label{tr2}
{\rm Trace}([b,R_{\lambda,\ell}]L_{Q})=\fint_Q\fint_{\hat{Q}}(b(x)-b(y))s_Q(x)dm_\lambda(y)dm_\lambda(x).
\end{align}
Combining inequality \eqref{tr1} with equality \eqref{tr2}, we see that
\begin{align*}
MO_Q(b)\lesssim |{\rm Trace}([b,R_{\lambda,\ell}]L_{Q})|.
\end{align*}
Using the duality between Lorentz spaces $\ell^{p',q'}$ and $\ell^{p,q}$, where $1/p+1/p'=1$ and $1/q+1/q'=1$, we deduce that for any $\nu\in\{1,2,\ldots,\kappa\}$,
\begin{align}\label{uuuno}
\left\|\left\{MO_Q(b)\right\}_{Q\in\mathcal{D}^\nu}\right\|_{\ell^{p,q}}
&\lesssim \left\|\left\{{\rm Trace}([b,R_{\lambda,\ell}]L_{Q})\right\}_{Q\in\mathcal{D}^\nu}\right\|_{\ell^{p,q}}\nonumber\\
&=C\sup\limits_{\|\{a_Q\}_{Q\in\mathcal{D}^\nu}\|_{\ell^{p',q'}}=1}\Bigg|\sum_{Q\in\mathcal{D}^\nu}{\rm Trace}([b,R_{\lambda,\ell}]L_{Q})a_Q\Bigg|\nonumber\\
&=C\sup\limits_{\|\{a_Q\}_{Q\in\mathcal{D}^\nu}\|_{\ell^{p',q'}}=1}\Big\|[b,R_{\lambda,\ell}]\Big(\sum_{Q\in\mathcal{D}^\nu}a_QL_{Q}\Big)\Big\|_{S_\lambda^1}\nonumber\\
&\lesssim \|[b,R_{\lambda,\ell}]\|_{S_\lambda^{p,q}}\sup\limits_{\|\{a_Q\}_{Q\in\mathcal{D}^\nu}\|_{\ell^{p',q'}}=1}\Big\|\sum_{Q\in\mathcal{D}^\nu}a_QL_{Q}\Big\|_{S_\lambda^{p',q'}}.
\end{align}
To continue, we let $N$ be a sufficiently large constant in Lemma \ref{jq} to see that $\sum_{Q\in\mathcal{D}^\nu}a_QL_{Q}$ can be written as
\begin{align*}
\sum_{Q\in\mathcal{D}^\nu}a_QL_{Q}(f)(x)=\sum_{u\in\mathcal{I}}\sum_{k\in\mathbb{Z}^4}\sum_{Q\in\mathcal{D}^\nu}a_QC_{Q,k,u}h_{Q,k,u}(x)\langle f ,g_{Q,k,u}\rangle,
\end{align*}
where $h_{Q,k,u}(x):=s_Q(x)f_{Q,k,u}(x)$ and $g_{Q,k,u}(x)$ are two \emph{nearly weakly orthogonal  (NWO)} sequences of functions. Therefore, recall from Section \ref{nwooo} that
\begin{align*}
\sup\limits_{\|\{a_Q\}_{Q\in\mathcal{D}^\nu}\|_{\ell^{p',q'}}=1}\Big\|\sum_{Q\in\mathcal{D}^\nu}a_QL_{Q}\Big\|_{S_\lambda^{p',q'}}
&\leq \sup\limits_{\|\{a_Q\}_{Q\in\mathcal{D}^\nu}\|_{\ell^{p',q'}}=1}\sum_{u\in\mathcal{I}}\sum_{k\in\mathbb{Z}^4}\Big\|\sum_{Q\in\mathcal{D}^\nu}a_QC_{Q,k,u}h_{Q,k,u}(x)\langle f ,g_{Q,k,u}\rangle\Big\|_{S_\lambda^{p',q'}}\\
&\lesssim \sup\limits_{\|\{a_Q\}_{Q\in\mathcal{D}^\nu}\|_{\ell^{p',q'}}=1} \sum_{u\in\mathcal{I}}\sum_{k\in\mathbb{Z}^4}\|\{a_QC_{Q,k,u}\}_{Q\in\mathcal{D}^\nu}\|_{\ell^{p',q'}}\\
&\lesssim \sup\limits_{\|\{a_Q\}_{Q\in\mathcal{D}^\nu}\|_{\ell^{p',q'}}=1} \sum_{k\in\mathbb{Z}^4}(1+|k|)^{-N}\|\{a_Q\}_{Q\in\mathcal{D}^\nu}\|_{\ell^{p',q'}}\\
&\lesssim 1.
\end{align*}
Substituting the above inequality into \eqref{uuuno}, we complete the proof of Proposition \ref{weakpro}.
\end{proof}

\section{Characterizations of oscillation space: Proof of Theorems \ref{schatten} and \ref{schatten1}}\label{Sec6}
\setcounter{equation}{0}
This section is devoted to building the bridge between the oscillation space  and Besov space in the Bessel setting, and then to providing the proof of Theorems \ref{schatten} and \ref{schatten1}.
\subsection{The case: $\ n+1< p=q<\infty$}
\begin{proposition}\label{beeee1}
Suppose $n+1< p<\infty$. Assume that $f\in L^p_{\rm loc}(\mathbb{R}^{n+1}_+,dm_\lambda)$. Then $f\in{\rm OSC}_{p,p}(\mathbb{R}_+^{n+1},dm_\lambda)$ if and only if $f\in B_{p}(\mathbb{R}_+^{n+1},dm_\lambda)$.
\end{proposition}
\begin{proof}
To begin with, for any $f\in L^p_{\rm loc}(\mathbb{R}^{n+1}_+,dm_\lambda)$, it follows from inequality \eqref{doub} that
\begin{align}\label{kqbe}
\|f\|_{B_{p}(\mathbb{R}_+^{n+1},dm_\lambda)}
&=\left(\int_{\mathbb{R}_{+}^{n+1}}\int_{\mathbb{R}_{+}^{n+1}}\frac{|f(x)-f(y)|^{p}}{m_\lambda(B_{\mathbb{R}_+^{n+1}}(x,|x-y|))^{2}}dm_\lambda(x)dm_\lambda(y)\right)^{1/p}\nonumber\\
&\approx\left(\sum_{k\in\mathbb{Z}}\iint_{\{(x,\,y)\in \mathbb{R}_+^{n+1}\times\mathbb{R}_+^{n+1}:2^{-k-1}\leq|x-y|\leq 2^{-k}\}}\frac{|f(x)-f(y)|^{p}}{m_\lambda(B_{\mathbb{R}_+^{n+1}}(x,\,2^{-k}))^{2}}dm_\lambda(x)dm_\lambda(y)\right)^{1/p}\nonumber\\
&\approx\left(\sum_{k\in\mathbb{Z}}\sum_{Q\in\mathcal{D}_k^0}\fint_Q\fint_{B_{\mathbb{R}_+^{n+1}}(c_Q,\,c_n2^{-k})}|f(x)-f(y)|^{p}dm_\lambda(x)dm_\lambda(y)\right)^{1/p},
\end{align}
where $c_Q$ is denoted by the center of $Q$ and $c_n$ is a constant depending on $n$. Applying inequalities \eqref{doub}, \eqref{dvbh} and \eqref{dvbh2}, we deduce that
\begin{align*}
{\rm RHS}\ {\rm of}\ \eqref{kqbe} 
&\lesssim\left(\sum_{Q\in\mathcal{D}^0}\fint_{B_{\mathbb{R}_+^{n+1}}(c_Q,\,c_n\ell(Q))}|f(x)-(f)_Q|^{p}dm_\lambda(x)\right)^{1/p}\\
&\hspace{1.0cm}+\left(\sum_{Q\in\mathcal{D}^0}\fint_Q|f(y)-(f)_Q|^{p}dm_\lambda(y)\right)^{1/p}\\
&\lesssim \|f\|_{{\rm OSC}_{p,p}(\mathbb{R}_+^{n+1},dm_\lambda)}.
\end{align*}

On the other hand, for any $\nu\in\{1,2,\ldots,\kappa\}$, by inequality \eqref{doub},
\begin{align*}
\left(\sum_{Q\in\mathcal{D}^\nu}\fint_Q|f(x)-f(y)|^pdm_\lambda(x)\right)^{1/p}
&\leq \left(\sum_{Q\in\mathcal{D}^\nu}\fint_Q\left|\fint_Qf(x)-f(y)dm_\lambda(y)\right|^pdm_\lambda(x)\right)^{1/p}\\
&\lesssim \left(\sum_{Q\in\mathcal{D}^\nu}\int_Q\int_Q\frac{|f(x)-f(y)|^p}{m_\lambda(B_{\mathbb{R}_+^{n+1}}(x,\,|x-y|))^{2}}dm_\lambda(y)dm_\lambda(x)\right)^{1/p}\\
&\lesssim \|f\|_{B_{p}(\mathbb{R}_+^{n+1},dm_\lambda)}.
\end{align*}
This implies that $\|f\|_{{\rm OSC}_{p,p}(\mathbb{R}_+^{n+1},dm_\lambda)}\lesssim \|f\|_{B_{p}(\mathbb{R}_+^{n+1},dm_\lambda)}$ and then ends the proof of Proposition \ref{beeee1}.
\end{proof}

{\it Proof of Theorem \ref{schatten}.}

Combining Propositions \ref{divide2}, \ref{weakpro} and \ref{beeee1}, we finish the proof of Theorem \ref{schatten} directly.
\subsection{The case: $\ 0<p<n+1, 0<q\leq \infty,\ {\rm or}\ p=n+1, 0< q<\infty$}
\begin{lemma}\label{twocubes}
For any $k\in\mathbb{Z}$, cube $Q\in \mathcal{D}_{k}^0$ and $a_{Q,j}=\pm 1(j=1,2,\ldots,n+1)$, there are cubes $A_Q \in \mathcal{D}_{k+2}^0, B_Q \in \mathcal{D}_{k+2}^0$ such that $A_Q\subset Q, B_Q \subset Q$ and if $x=(x_1,\,x_2,\,\ldots,\,x_{n+1})\in A_Q$, $y= (y_1,\,y_2,\,\ldots,\,y_{n+1})\in B_Q$, then $a_{Q,j}(x_j-y_j)\geq 2^{-k-1}$ for $j=1,2,\ldots,n+1$.
\end{lemma}
\begin{proof}
By symmetry, it suffices to consider the case that $Q\subset\mathbb{R}_+\times\cdots\times\mathbb{R}_+$. Suppose that the vertex of $Q$ is at $2^{-k}m$ for some $k\in\mathbb{Z}$ and $m=(m_1,\ldots,m_{n+1})\in\mathbb{N}_+\times\cdots\times\mathbb{N}_+$. Then we choose cubes $A_Q \in \mathcal{D}_{k+2}^0, B_Q \in \mathcal{D}_{k+2}^0$ with vertices at $(2^{-k}m_1+2^{-k-2}+a_{Q,1}2^{-k-2},\ldots,2^{-k}m_{n+1}+2^{-k-2}+a_{Q,n+1}2^{-k-2})$ and $(2^{-k}m_1+2^{-k-2}-a_{Q,1}2^{-k-2},\ldots,2^{-k}m_{n+1}+2^{-k-2}-a_{Q,n+1}2^{-k-2})$, respectively. Then these cubes satisfy the required properties in the statement.
\end{proof}

Now, we provide a lower bound for a local pseudo oscillation of the symbol $b$.
\begin{lemma}\label{lowerboundcommu}
Let $f\in C^2(\mathbb{R}_+^{n+1})$. Suppose that there is a point $x_0\in \mathbb{R}_+^{n+1}$ such that $\nabla f(x_0) \neq 0$. Then there exist constants $C>0, \epsilon>0$ and  $N>0$ such that if $k>N$, then for any cube $Q \in \mathcal{D}_{k}^0$ satisfying $|{\rm center}(Q)- x_0|<\epsilon$, we have
\begin{equation}\label{contonb}
\fint_{A_Q}|f(x)-(f)_{B_Q}|dm_\lambda(x)\geq C2^{-k}|\nabla f(x_0)|,
\end{equation}
where $A_Q$ and $B_Q$ are the cubes chosen in Lemma \ref{twocubes}.

\end{lemma}

\begin{proof}
Denote by $c_Q := (c_Q^1,c_Q^2,\ldots,c_Q^{n+1})$ the center of $Q$ and $x = (x_1, x_2,\ldots,x_{n+1})$, then it follows from Taylor's formula that
\begin{equation}\label{usetay}
 f(x) = f(c_Q) + \sum_{j=1}^{n+1}{(\partial_{x_j}f)(c_Q)}(x_j-c_{Q}^j) + R(x,\,c_Q),
\end{equation}
where the remainder term $R(x,\,c_Q)$ satisfies
\begin{equation}\label{laste}
 |R(x,\,c_Q)|\leq C\sum_{j=1}^{n+1}\sum_{k=1}^{n+1} \sup_{\theta\in[0,1]}\big|(\partial_{x_j}\partial_{x_k}f)(x+\theta(c_Q-x))\big| |x - c_Q|^{2}
\end{equation}
for some $\theta\in[0,1]$. Note that if $x\in Q$, then
$$|x+\theta(c_Q-x)-c_Q|\lesssim 2^{-k}.$$
By  Lemma \ref{twocubes}, for $x'=(x_1',\,x_2',\,...,\,x_{n+1}')\in A_Q$ and $x''=(x_{1}'',\,x_{2}'',\,...,\,x_{n+1}'')\in B_Q$, we have
\begin{equation*}
    {\rm sgn}(\partial_{x_j}f)(c_Q)(x_j'-x_j'')\geq 2^{-k}, \quad j=1,2,\ldots,n+1.
\end{equation*}
Combining the above facts, we deduce that
\begin{align*}
\fint_{A_Q}|f(x)-(f)_{B_Q}|dm_\lambda(x) & \geq  \fint_{A_Q}\bigg|\fint_{B_Q}\sum_{j=1}^{n+1}{(\partial_{x_j}f)(c_Q)}(x_j'-x_j'')dm_\lambda(x'')\bigg| dm_\lambda(x') \\
&\hspace{1.0cm}- \fint_{A_Q}|R(x',\,c_Q)|dm_\lambda(x') - \fint_{B_Q}|R(x'',\,c_Q)|dm_\lambda(x'')  \\&\geq C\sum_{j=1}^{n+1}|(\partial_{x_j}f)(c_Q)|2^{-k} -C2^{-2k}\|\nabla^2 f\|_{L^\infty(B(x_0,\,1))}\\
&\geq C2^{-k} |\nabla f(x_0)|,
\end{align*}
where in the last inequality $\epsilon$ is chosen to be a sufficiently small constant depending on $f$. This completes the proof of Lemma \ref{lowerboundcommu}.
\end{proof}

\begin{proposition}\label{aux}
Suppose that $0<p<n+1$, $0<q\leq \infty$, or $p=n+1$, $0< q<\infty$. Assume that $f\in C^2(\mathbb{R}^{n+1}_+)$.  Then $f\in{\rm OSC}_{p,q}(\mathbb{R}_+^{n+1},dm_\lambda)$ if and only if $f$ is a constant on $\mathbb{R}_+^{n+1}$.
\end{proposition}
\begin{proof}
We will show Proposition \ref{aux} by contradiction. If $f$ is not a constant on $\mathbb{R}_+^{n+1}$, then there exists a point $x_0=(x_0^{(1)},\ldots,x_0^{(n+1)})\in \mathbb{R}_{+}^{n+1}$ such that $\nabla f(x_0) \neq 0$. By Lemma \ref{lowerboundcommu}, there exist constants $C>0, \epsilon>0$ and  $N>0$ such that if $k>N$, then for any cube $Q \in \mathcal{D}_{k}^0$ satisfying $|{\rm center}(Q)- x_0|<\epsilon$, we have
\begin{equation}\label{contonb}
\fint_{A_Q}|f(x)-(f)_{B_Q}|dm_\lambda(x)\geq C2^{-k}|\nabla f(x_0)|,
\end{equation}
where $A_Q$ and $B_Q$ are the cubes chosen in Lemma \ref{lowerboundcommu}. By inequality \eqref{doub}, for any $Q\in\mathcal{D}^0$,
\begin{align}
\fint_{A_Q}|f(x)-(f)_{B_Q}|dm_\lambda(x)
&\lesssim\fint_Q|f(x)-(f)_Q|dm_\lambda(x).\label{rhsssq}
\end{align}
Denote by $\mathcal{A}_k(x_0)$ the set consisting of $Q\in \mathcal{D}_{k}^0$ satisfying $|{\rm center}(Q)-x_0|<\epsilon$. Observe that for any $k>N$, the number of cubes in $\mathcal{A}_k(x_0)$ is at least $2^{k(n+1)}$, up to a harmless constant depending on $\epsilon$. To continue, note that the condition $f\in{\rm OSC}_{p,q}(\mathbb{R}_+^{n+1},dm_\lambda)$, together with \eqref{eq:ball;included}, implies that the $\ell^{p,q}$-norm of right-hand side on \eqref{rhsssq} indexed by $Q\in\mathcal{D}^0$ is finite. However, we will deduce a contradiction by dividing our proof into two cases.

 {\bf Case 1.} If $0<p\leq n+1$ and $q<\infty$, then
\begin{align*}
&\left\|\left\{\fint_Q|f(x)-(f)_Q|dm_\lambda(x)\right\}_{Q\in\mathcal{D}^0}\right\|_{\ell^{p,q}}\\
&\geq C\left\|\left\{\ell(Q)|\nabla f(x_0)|\right\}_{Q\in\cup_{k>N}\mathcal{A}_k(x_0)}\right\|_{\ell^{p,q}}\\
&= C\left(\int_0^{2^{-N}|\nabla f(x_0)|}\left(\left(\sum_{k>N}\#\left\{Q\in\mathcal{A}_k(x_0):2^{-k}|\nabla f(x_0)|>\lambda\right\}\right)\lambda^p\right)^{q/p}\frac{d\lambda}{\lambda}\right)^{1/q}\\
&= C\left(\sum_{\ell=N}^\infty\int_{2^{-\ell-1}|\nabla f(x_0)|}^{2^{-\ell}|\nabla f(x_0)|}\left(\left(\sum_{k>N}\#\left\{Q\in\mathcal{A}_k(x_0):2^{-k}|\nabla f(x_0)|>\lambda\right\}\right)\lambda^p\right)^{q/p}\frac{d\lambda}{\lambda}\right)^{1/q}\\
&= C\left(\sum_{\ell=N}^\infty\int_{2^{-\ell-1}|\nabla f(x_0)|}^{2^{-\ell}|\nabla f(x_0)|}\left(\left(\sum_{N<k<\ell}2^{k(n+1)}\right)(2^{-\ell}|\nabla f(x_0)|)^p\right)^{q/p}\frac{d\lambda}{\lambda}\right)^{1/q}\\
&\geq C|\nabla f(x_0)|\left(\sum_{\ell=N}^\infty 2^{\frac{\ell q(n+1-p)}{p}}\right)^{1/q}=+\infty,
\end{align*}
which is a contradiction.

{\bf Case 2.} If $0<p< n+1$ and  $q=\infty$, then
\begin{align*}
&\left\|\left\{\fint_Q|f(x)-(f)_Q|dm_\lambda(x)\right\}_{Q\in\mathcal{D}^0}\right\|_{\ell^{p,\infty}}\\
&\geq C\left\|\left\{\ell(Q)|\nabla f(x_0)|\right\}_{Q\in\cup_{k>N}\mathcal{A}_k(x_0)}\right\|_{\ell^{p,\infty}}\\
&= C\sup\limits_{\ell\in\mathbb{Z}:N\leq \ell<+\infty}\sup\limits_{2^{-\ell-1}|\nabla f(x_0)|\leq\lambda<2^{-\ell}|\nabla f(x_0)|}\lambda\left(\sum_{k>N}\#\left\{Q\in\mathcal{A}_k(x_0):\ell(Q)|\nabla f(x_0)|>\lambda\right\}\right)^{1/p}\\
&\geq C\sup\limits_{\ell\in\mathbb{Z}:N\leq \ell<+\infty}2^{-\ell}|\nabla f(x_0)|\left(\sum_{N<k<\ell}\#\left\{Q\in\mathcal{A}_k(x_0)\right\}\right)^{1/p}\\
&\geq C|\nabla f(x_0)|\sup\limits_{\ell\in\mathbb{Z}:N\leq \ell<+\infty}2^{\frac{\ell(n+1-p)}{p}}=+\infty,
\end{align*}
which is a contradiction.  Thus, the proof of Proposition \ref{aux} is complete.
\end{proof}

{\it Proof of Theorem \ref{schatten1}.}

By Propositions \ref{weakpro} and \ref{aux}, it remains to consider the case $0<p\leq 1$ or $0<q< 1$. To this end, since $\ell^{p,q}$ is increasing as the exponents $p$ and $q$ increase, the condition $[b,R_{\lambda,\ell}]\in S_\lambda^{p,q}$ implies that $[b,R_{\lambda,\ell}]\in S_\lambda^{p_0,q_0}$ for some $1<p_0\leq n+1$ and $1<q_0\leq \infty$ with $(p_0,\,q_0)\neq (n+1,\,\infty)$. This, in combination with Propositions \ref{weakpro} and \ref{aux} again, indicates that $b$ is a constant. Thus, the proof of Theorem \ref{schatten1} is complete.
\section{Appendix}\label{Appe}
In the Appendix, we show that $n+1$ is the optimal cut-off point of Schatten--Lorentz characterization stated in Theorems \ref{schatten} and \ref{schatten1}.
\begin{proposition}\label{count}
There exists a $C^\infty$ non-trivial function $f$ such that $f\in B_p(\mathbb R_+^{n+1},dm_\lambda)$ with $p>n+1$.
\end{proposition}
\begin{proof}
Let $x_0=(0,\ldots,0,x_{0,n+1})$ with $x_{0,n+1}$ being a sufficiently large constant. Let $f(x)$ be a smooth cut-off function supported in $B(x_0,\,1)$. Now we show that $f\in B_p(\mathbb R_+^{n+1},dm_\lambda)$ for any $p>n+1$.  Indeed,
\begin{align*}
\|f\|_{B_p}
&\leq\left(\int_{B(x_0,\,1)}\int_{B(x_0,\,1)}\frac{|f(x)-f(y)|^p}{{m_\lambda(B_{\mathbb{R}_+^{n+1}}}(x,\,|x-y|))^2}dm_\lambda(y)dm_\lambda(x)\right)^{1/p}\\
&\hspace{1.0cm}+\left(\int_{\mathbb{R}_+^{n+1}\backslash B(x_0,\,1)}\int_{B(x_0,\,1)}\frac{|f(x)-f(y)|^p}{{m_\lambda(B_{\mathbb{R}_+^{n+1}}}(x,\,|x-y|))^2}dm_\lambda(y)dm_\lambda(x)\right)^{1/p}\\
&\hspace{1.0cm}+\left(\int_{B(x_0,\,1)}\int_{\mathbb{R}_+^{n+1}\backslash B(x_0,\,1)}\frac{|f(x)-f(y)|^p}{{m_\lambda(B_{\mathbb{R}_+^{n+1}}}(x,\,|x-y|))^2}dm_\lambda(y)dm_\lambda(x)\right)^{1/p}\\
&={\rm I}+{\rm II}+{\rm III}.
\end{align*}

For the term ${\rm I}$, by inequality \eqref{doub} in combination with Taylor's formula,
\begin{align*}
{\rm I}&\lesssim \left(\int_{B(x_0,\,1)}\int_{B(x_0,\,1)}\frac{|f(x)-f(y)|^p}{|x-y|^{2(n+1)}}\prod_{j=1}^{n+1}dy_j\prod_{j=1}^{n+1}dx_j\right)^{1/p}\\
&\lesssim \|\nabla f\|_{\infty} \left(\int_{B(x_0,\,1)}\int_{B(x_0,\,1)}\frac{1}{|x-y|^{2(n+1)-p}}\prod_{j=1}^{n+1}dy_j\prod_{j=1}^{n+1}dx_j\right)^{1/p}\\
&<+\infty.
\end{align*}

For the term ${\rm II}$, by inequality \eqref{doub},
\begin{align*}
{\rm II}&\lesssim \left(\int_{B(x_0,\,1)}\int_{(\mathbb{R}_+^{n+1}\backslash B(x_0,\,1))\cap\{y\in\mathbb{R}_+^{n+1}:|x-y|\leq 10x_{0,n+1}\}}\frac{|f(x)-f(y)|^p}{|x-y|^{2(n+1)}}\prod_{j=1}^{n+1}dy_j\prod_{j=1}^{n+1}dx_j\right)^{1/p}\\
&\hspace{1.0cm}+ \left(\int_{B(x_0,\,1)}\int_{(\mathbb{R}_+^{n+1}\backslash B(x_0,\,1))\cap\{y\in\mathbb{R}_+^{n+1}:|x-y|\geq 10x_{0,n+1}\}}\frac{|f(x)-f(y)|^p}{|x-y|^{2(n+1+2\lambda)}}y_{n+1}^{2\lambda}\prod_{j=1}^{n}dy_jx_{n+1}^{2\lambda}\prod_{j=1}^{n}dx_j\right)^{1/p}\\
&=:{\rm II}_1+{\rm II}_2.
\end{align*}
To estimate the term ${\rm II}_1$, we apply Taylor's formula to see that
\begin{align*}
{\rm II}_1\lesssim \|\nabla f\|_\infty\left(\int_{B(x_0,\,1)}\int_{\{y\in\mathbb{R}_+^{n+1}:|x-y|\leq 10x_{0,n+1}\}}\frac{1}{|x-y|^{2(n+1)-p}}\prod_{j=1}^{n+1}dy_j\prod_{j=1}^{n+1}dx_j\right)^{1/p}<+\infty.
\end{align*}
To estimate the term ${\rm II}_2$, we note that $x_{n+1}\lesssim 1$ and $y_{n+1}\leq |y-x|+|x-x_0|+|x_0|\lesssim |x-y| $. This means that
\begin{align*}
{\rm II}_2\lesssim \|f\|_\infty\left(\int_{B(x_0,\,1)}\int_{\{y\in\mathbb{R}_+^{n+1}:|x-y|\geq 10x_{0,n+1}\}}\frac{1}{|x-y|^{2(n+1)+2\lambda}}\prod_{j=1}^{n}dy_j\prod_{j=1}^{n}dx_j\right)^{1/p}<+\infty.
\end{align*}
Combining these estimates, we conclude that $f\in B_p(\mathbb R_+^{n+1},dm_\lambda)$.
\end{proof}
\begin{remark}
By a similar argument, we can see that the function constructed in Proposition \ref{count} also satisfies $f\in {\rm OSC}_{p,q}(\mathbb{R}_+^{n+1},dm_\lambda)$ for all $p> n+1$ and $0<q< \infty$, or $p\geq n+1$ and $q=\infty$.
\end{remark}

\bigskip

 \noindent
 {\bf Acknowledgements:}

Z. Fan is Supported by the China Postdoctoral Science Foundation (No. 2023M740799) and Postdoctoral Fellowship Program of CPSF (No. GZB20230175). J. Li is supported by the Australian Research Council (ARC) through the research grant DP220100285. X. Xiong is Supported by National Natural Science Foundation of China, General Project 12371138. Z. Fan would like to thank Prof. Dongyong Yang for his talk about harmonic analysis associated with Bessel operator.


\begin{thebibliography}{10}

\bibitem{AFP}
J.~Arazy, S.~D. Fisher, and J.~Peetre.
\newblock Hankel operators on weighted {B}ergman spaces.
\newblock {\em Amer. J. Math.}, 110(6):989--1053, 1988.

\bibitem{MR3922045}
J.J. Betancor, X.T. Duong, J.~Li, B.D. Wick, and D.~Yang.
\newblock Product {H}ardy, {BMO} spaces and iterated commutators associated
  with {B}essel {S}chr\"{o}dinger operators.
\newblock {\em Indiana Univ. Math. J.}, 68(1):247--289, 2019.

\bibitem{MR2496404}
J.J. Betancor, J.~Dziuba\'{n}ski, and J.L. Torrea.
\newblock On {H}ardy spaces associated with {B}essel operators.
\newblock {\em J. Anal. Math.}, 107:195--219, 2009.

\bibitem{MR4361597}
Y.~Chen, X.T. Duong, J.~Li, W.~Tao, and D.~Yang.
\newblock Square roots of the {B}essel operators and the related
  {L}ittlewood-{P}aley estimates.
\newblock {\em Studia Math.}, 263(1):19--58, 2022.

\bibitem{CWbook}
R.~Coifman and G.~Weiss.
\newblock {\em Analyse harmonique non-commutative sur certains espaces
  homog\`enes}.
\newblock Lecture Notes in Mathematics, Vol. 242. Springer-Verlag, Berlin-New
  York, 1971.
\newblock \'{E}tude de certaines int\'{e}grales singuli\`eres.

\bibitem{MR4660138}
J.M. Conde-Alonso, A.M. Gonz\'{a}lez-P\'{e}rez, J.~Parcet, and E.~Tablate.
\newblock Schur multipliers in {S}chatten--von {N}eumann classes.
\newblock {\em Ann. of Math. (2)}, 198(3):1229--1260, 2023.

\bibitem{Connes}
A.~Connes.
\newblock {\em Noncommutative geometry}.
\newblock Academic Press, Inc., San Diego, CA, 1994.

\bibitem{CST}
A.~Connes, D.~Sullivan, and N.~Teleman.
\newblock Quasiconformal mappings, operators on {H}ilbert space, and local
  formulae for characteristic classes.
\newblock {\em Topology}, 33(4):663--681, 1994.

\bibitem{MR3212723}
J.~Delgado and M.~Ruzhansky.
\newblock Schatten classes on compact manifolds: kernel conditions.
\newblock {\em J. Funct. Anal.}, 267(3):772--798, 2014.

\bibitem{MR4312282}
J.~Delgado and M.~Ruzhansky.
\newblock Schatten--von {N}eumann classes of integral operators.
\newblock {\em J. Math. Pures Appl. (9)}, 154:1--29, 2021.

\bibitem{MR3650324}
J.~Delgado, M.~Ruzhansky, and N.~Tokmagambetov.
\newblock Schatten classes, nuclearity and nonharmonic analysis on compact
  manifolds with boundary.
\newblock {\em J. Math. Pures Appl. (9)}, 107(6):758--783, 2017.

\bibitem{DGKLWY}
X.T. Duong, R.~Gong, M.~Kuffner, J.~Li, B.D. Wick, and D.~Yang.
\newblock Two weight commutators on spaces of homogeneous type and
  applications.
\newblock {\em J. Geom. Anal.}, 31(1):980--1038, 2021.

\bibitem{MR3829612}
X.T. Duong, J.~Li, S.~Mao, H.~Wu, and D.~Yang.
\newblock Compactness of {R}iesz transform commutator associated with {B}essel
  operators.
\newblock {\em J. Anal. Math.}, 135(2):639--673, 2018.

\bibitem{MR3689327}
X.T. Duong, J.~Li, B.D. Wick, and D.~Yang.
\newblock Factorization for {H}ardy spaces and characterization for {BMO}
  spaces via commutators in the {B}essel setting.
\newblock {\em Indiana Univ. Math. J.}, 66(4):1081--1106, 2017.

\bibitem{MR4231681}
X.T. Duong, J.~Li, B.D. Wick, and D.~Yang.
\newblock Characterizations of product {H}ardy spaces in {B}essel setting.
\newblock {\em J. Fourier Anal. Appl.}, 27(2):Paper No. 24, 65, 2021.

\bibitem{MR3481176}
J.~Dziuba\'{n}ski, M.~Preisner, L.~Roncal, and P.R. Stinga.
\newblock Hardy spaces for {F}ourier-{B}essel expansions.
\newblock {\em J. Anal. Math.}, 128:261--287, 2016.

\bibitem{MR1359927}
D.E. Edmunds and V.D. Stepanov.
\newblock On the singular numbers of certain {V}olterra integral operators.
\newblock {\em J. Funct. Anal.}, 134(1):222--246, 1995.

\bibitem{MR1323687}
D.~Fan and Z.~Wu.
\newblock Norm estimates for iterated commutators on the {B}ergman spaces of
  the unit ball.
\newblock {\em Amer. J. Math.}, 117(2):523--543, 1995.

\bibitem{MR4552581}
Z.~Fan, M.~Lacey, and J.~Li.
\newblock Schatten classes and commutators of {R}iesz transform on {H}eisenberg
  group and applications.
\newblock {\em J. Fourier Anal. Appl.}, 29(2):Paper No. 17, 28, 2023.

\bibitem{MR4654013}
Z.~Fan, J.~Li, E.~McDonald, F.~Sukochev, and D.~Zanin.
\newblock Endpoint weak {S}chatten class estimates and trace formula for
  commutators of {R}iesz transforms with multipliers on {H}eisenberg groups.
\newblock {\em J. Funct. Anal.}, 286(1):Paper No. 110188, 72, 2024.

\bibitem{FR}
M.~Feldman and R.~Rochberg.
\newblock Singular value estimates for commutators and {H}ankel operators on
  the unit ball and the {H}eisenberg group.
\newblock In {\em Analysis and partial differential equations}, volume 122 of
  {\em Lecture Notes in Pure and Appl. Math.}, pages 121--159. Dekker, New
  York, 1990.

\bibitem{MR4549699}
R.L. Frank, F.~Sukochev, and D.~Zanin.
\newblock Asymptotics of singular values for quantum derivatives.
\newblock {\em Trans. Amer. Math. Soc.}, 376(3):2047--2088, 2023.

\bibitem{MR4667745}
Z.~Gong, J.~Li, and B.D. Wick.
\newblock Besov spaces, {S}chatten classes and weighted versions of the
  quantised derivative.
\newblock {\em Anal. Math.}, 49(4):971--1006, 2023.

\bibitem{MR3243734}
L.~Grafakos.
\newblock {\em Classical {F}ourier analysis}, volume 249 of {\em Graduate Texts
  in Mathematics}.
\newblock Springer, New York, third edition, 2014.

\bibitem{MR64284}
A.~Huber.
\newblock On the uniqueness of generalized axially symmetric potentials.
\newblock {\em Ann. of Math. (2)}, 60:351--358, 1954.

\bibitem{HK}
T.P. Hyt\"{o}nen and A.~Kairema.
\newblock Systems of dyadic cubes in a doubling metric space.
\newblock {\em Colloq. Math.}, 126(1):1--33, 2012.

\bibitem{MR3158507}
J.~Isralowitz.
\newblock Schatten {$p$} class commutators on the weighted {B}ergman space
  {$L_a^2(\Bbb B_n,{\rm d}v_\gamma)$} for {$2n/(n+1+\gamma)<p<\infty$}.
\newblock {\em Indiana Univ. Math. J.}, 62(1):201--233, 2013.

\bibitem{JW}
S.~Janson and T.H. Wolff.
\newblock Schatten classes and commutators of singular integral operators.
\newblock {\em Ark. Mat.}, 20(2):301--310, 1982.

\bibitem{LLW}
M.T. Lacey, J.~Li, and B.D. Wick.
\newblock Schatten classes and commutator in the two weight setting, I.
  {H}ilbert transform.
\newblock {\em Potential Anal.}, 60(2): 875–894,  2024. 

\bibitem{LMSZ}
S.~Lord, E.~McDonald, F.~Sukochev, and D.~Zanin.
\newblock Quantum differentiability of essentially bounded functions on
  {E}uclidean space.
\newblock {\em J. Funct. Anal.}, 273(7):2353--2387, 2017.

\bibitem{MR899655}
D.H. Luecking.
\newblock Trace ideal criteria for {T}oeplitz operators.
\newblock {\em J. Funct. Anal.}, 73(2):345--368, 1987.

\bibitem{MSX2018}
E.~McDonald, F.~Sukochev, and X.~Xiong.
\newblock Quantum differentiability on quantum tori.
\newblock {\em Comm. Math. Phys.}, 371(3):1231--1260, 2019.

\bibitem{MSX2019}
E.~McDonald, F.~Sukochev, and X.~Xiong.
\newblock Quantum differentiability on noncommutative {E}uclidean spaces.
\newblock {\em Comm. Math. Phys.}, 379(2):491--542, 2020.

\bibitem{MS}
B.~Muckenhoupt and E.~M. Stein.
\newblock Classical expansions and their relation to conjugate harmonic
  functions.
\newblock {\em Trans. Amer. Math. Soc.}, 118:17--92, 1965.

\bibitem{Ne}
Z.~Nehari.
\newblock On bounded bilinear forms.
\newblock {\em Ann. of Math. (2)}, 65:153--162, 1957.

\bibitem{P}
V.V. Peller.
\newblock Nuclearity of hankel operators.
\newblock {\em Mat. Sbornik}, 113:538--581, 1980.

\bibitem{MR4179877}
R.~Rahm, E.T. Sawyer, and B.D. Wick.
\newblock Weighted {A}lpert wavelets.
\newblock {\em J. Fourier Anal. Appl.}, 27(1):Paper No. 1, 41, 2021.

\bibitem{MR654483}
R.~Rochberg.
\newblock Interpolation by functions in {B}ergman spaces.
\newblock {\em Michigan Math. J.}, 29(2):229--236, 1982.

\bibitem{RS1988}
R.~Rochberg and S.~Semmes.
\newblock End point results for estimates of singular values of singular
  integral operators.
\newblock In {\em Contributions to operator theory and its applications
  ({M}esa, {AZ}, 1987)}, volume~35 of {\em Oper. Theory Adv. Appl.}, pages
  217--231. Birkh\"{a}user, Basel, 1988.

\bibitem{RS}
R.~Rochberg and S.~Semmes.
\newblock Nearly weakly orthonormal sequences, singular value estimates, and
  {C}alderon-{Z}ygmund operators.
\newblock {\em J. Funct. Anal.}, 86(2):237--306, 1989.

\bibitem{MR3494249}
G.~Rozenblum, M.~Ruzhansky, and D.~Suragan.
\newblock Isoperimetric inequalities for {S}chatten norms of {R}iesz
  potentials.
\newblock {\em J. Funct. Anal.}, 271(1):224--239, 2016.

\bibitem{MR4554083}
E.T. Sawyer and B.D. Wick.
\newblock Two weight {S}obolev norm inequalities for smooth
  {C}alder\'{o}n-{Z}ygmund operators and doubling weights.
\newblock {\em Math. Z.}, 303(4):Paper No. 81, 74, 2023.

\bibitem{MR3010276}
K.~Seip and E.H. Youssfi.
\newblock Hankel operators on {F}ock spaces and related {B}ergman kernel
  estimates.
\newblock {\em J. Geom. Anal.}, 23(1):170--201, 2013.

\bibitem{MR1766114}
V.D. Stepanov.
\newblock On the lower bounds for {S}chatten-von {N}eumann norms of certain
  {V}olterra integral operators.
\newblock {\em J. London Math. Soc. (2)}, 61(3):905--922, 2000.

\bibitem{MR4687434}
P.~Villarroya.
\newblock The {S}chatten {C}lasses of {C}alder\'{o}n--{Z}ygmund {O}perators.
\newblock {\em J. Fourier Anal. Appl.}, 30(1):Paper no. 9, 2024.

\bibitem{MR53289}
A.~Weinstein.
\newblock Generalized axially symmetric potential theory.
\newblock {\em Bull. Amer. Math. Soc.}, 59:20--38, 1953.

\bibitem{MR2823879}
D.~Yang and D.~Yang.
\newblock Real-variable characterizations of {H}ardy spaces associated with
  {B}essel operators.
\newblock {\em Anal. Appl. (Singap.)}, 9(3):345--368, 2011.

\bibitem{MR1087805}
K.~Zhu.
\newblock Schatten class {H}ankel operators on the {B}ergman space of the unit
  ball.
\newblock {\em Amer. J. Math.}, 113(1):147--167, 1991.

\end{thebibliography}
\end{document}